\documentclass[11pt]{amsart}

\usepackage{amsmath}
\usepackage{amsfonts}
\usepackage{amscd}
\usepackage[dvips]{graphicx,color}
\usepackage{pifont}
\usepackage{epsfig}

\newtheorem{thm}{Theorem}
\newtheorem{lemma}{Lemma}[section]
\newtheorem{cor}[lemma]{Corollary}
\newtheorem{remark}{Remark}[section]

\newtheorem{prop}{Proposition}

\newcommand\be{\begin{equation}}
\newcommand\ee{\end{equation}}
\newcommand\bea{\begin{eqnarray}}
\newcommand\eea{\end{eqnarray}}
\newcommand\beaa{\begin{eqnarray*}}
\newcommand\eeaa{\end{eqnarray*}}
\newcommand\bay{\begin{array}}
\newcommand\eay{\end{array}}
\newcommand\ba{\begin{align}}
\newcommand\ea{\end{align}}

\newcommand{\e}{\epsilon}

\begin{document}

\title[strong competition system]
{Sharp estimates for {the spreading speeds} of the Lotka-Volterra diffusion system with strong competition}

\author{Rui Peng}
\address{School of Mathematics and Statistics, Jiangsu Normal University, Xuzhou, 221116, Jiangsu Province, People's Republic of China}
\email{pengrui\_seu@163.com}

\author{Chang-Hong Wu}
\address{Department of Applied Mathematics, National Chiao Tung University, Hsinchu 300, Taiwan, Republic of China}
\email{changhong@math.nctu.edu.tw}

\author{Maolin Zhou}
\address{School of Science and Technology, University of New England, Armidale, NSW 2351, Australia}
\email{zhouutokyo@gmail.com}

\thanks{Date: \today.}

\thanks{{\em 2010 Mathematics Subject Classification.} 35K57, 35K45, 92D25.}

\thanks{{\em Key words and phrases: Lotka-Volterra diffusion system, strong competition, traveling waves, long-time behavior, spreading speed and profile. } }

\begin{abstract}
This paper {is concerned with} the classical two-species Lotka-Volterra
diffusion system with strong competition. The sharp dynamical behavior of the solution is established in two different situations:
either one species is an invasive one and the other is a native one or both are invasive species.
Our results seem to be the first that provide a precise spreading
speed and profile for such a strong competition system.
Among other things, our analysis relies on the construction of new types of supersolution and subsolution,
which are optimal in certain sense.
\end{abstract}

\maketitle \setlength{\baselineskip}{16pt}
\section{Introduction}\label{introduction}
\setcounter{equation}{0}


In this paper, we consider the classical two-species Lotka-Volterra competition-diffusion system:
\bea\label{LV-sys}
\begin{cases}
u_t=du_{xx}+ru(1-u-a v),\ \  &t>0,\ x\in\mathbb{R},\\
v_t=v_{xx}+v(1-v-b u),\ \ &t>0,\ x\in\mathbb{R}
\end{cases}
\eea
with initial data
\bea\label{LV-ic}
u(0,x)=u_0(x),\quad v(0,x)=v_0(x),\ \ x\in\mathbb{R},
\eea
where $u(t,x)$ and $v(t,x)$ represent the population {densities} of two competing species at the position
$x$ and time $t$; $d$ stands for the diffusion rate of $u$; $r$ represents the intrinsic growth rate of $u$;
$a$ and $b$ represent the competition coefficient for two species, respectively.
All parameters are assumed to be positive. Note that {the} system \eqref{LV-sys} has been reduced into the dimensionless
form using a standard scaling (see, e.g., \cite{Murray}).

Since the pioneering works of Fisher \cite{Fisher} and Kolmogorov, Petrovsky and Piskunov \cite{KPP},
reaction-diffusion equations have been the subject of a large amount of research aiming at
the understanding of the spread dynamics of invasive species. More precisely, when an invasive species is introduced into a new environment, the mathematical approach of \cite{Fisher,KPP} to describe the spreading of species is based on the study of
the long time behavior of the solution of the following Fisher-KPP equation:
\bea\label{Single eq}
\begin{cases}
w_t=d w_{xx}+rw(1-w),\quad &t>0,\ x\in\mathbb{R},\\
w(0,t)=w_0(x),\quad &x\in\mathbb{R},
\end{cases}
\eea
where $w(t,x)$ stands for the population density for the invasive species at time $t$ and position $x$.

When $w_0\not\equiv0$ is nonnegative with compact support in $\mathbb{R}$, the classical result  of Aronson and Weinberger \cite{AW1975,AW1978} shows that there exists a unique $c^*=2\sqrt{rd}$ such that the solution $w$ to \eqref{Single eq} satisfies
\beaa
&&\lim_{t\rightarrow\infty}\max_{|x|\geq ct} w(x,t)=0\quad \mbox{for any}\ c>c^*;\\
&&\lim_{t\rightarrow\infty}\max_{|x|\leq ct} [1-w(x,t)]=0\quad \mbox{for any}\ c\in(0,c^*).
\eeaa
Such a spreading behavior describes the {invasion} phenomenon of the unstable state $0$ by the stable state $1$, and the quantity $c^*$ is often referred to as the (asymptotic) spreading speed of the species and
has been used to predict the spreading speed for various invasive species in {nature} \cite{SK1997}. Furthermore,
$c^*$ coincides with the minimal speed of the traveling wave solution of the form: $w(x-ct)$
connecting $1$ and $0$; that is, if and only if $c\geq c^*$, the following problem
 \beaa\label{Single eq-tw}
\begin{cases}
dw''+cw'+rw(1-w)=0,\ \ w>0\ \ \mbox{in}\ \mathbb{R},\\
w(-\infty)=1,\ \ \ w(\infty)=0\ \
\end{cases}
\eeaa
admits a unique solution (up to translation).

In the absence of the species $v$ (resp. $u$), {the} system \eqref{LV-sys} is reduced to the Fisher-KPP equation \eqref{Single eq}, which admits a unique traveling wave solution (up to translation), denoted by $U_{KPP}(x-ct)$ (resp. $V_{KPP}(x-ct)$) connecting $1$ and $0$ if and only if $c\geq 2\sqrt{rd}$ (resp. $c\geq 2$). For sake of convenience, we denote in this paper
 $$
 c_u=2\sqrt{rd},\ \ \ c_v=2.
 $$
Clearly, $c_u$ (resp. $v$) is the spreading speed of the species $u$ (resp. $v$) in the absence of the species $v$ (resp. $u$) of \eqref{LV-sys}.

Traveling wave solutions play a crucial role in understanding the spreading of invasive species.
As far as one species is concerned, great progress {has} been made in recent decades to determine the spreading dynamics via
the associated traveling wave solutions; one may refer to, for instance, \cite{Bramson,Chen,HNRR,Lau,Rothe1978,Sattinger1976,Uchiyama1978} and references therein.

When multiple species interact, there is a wide literature on (asymptotic) spreading speeds for various {kinds of} evolutional systems; see, e.g., \cite{FZ2015,LLW2002,LWL2005,LZ2007,LL2012,WLL2002} and references therein.
However, to the best of our knowledge, there have been only few papers devoted to the rigorous study of long-time dynamics of a multiple-species system. One of the mathematical difficulties lies in that in general different spreading speeds may occur in different species, which brings highly nontrivial challenges when one deals with the convergence of solutions.
Indeed, even for the simplest yet most classical Lotka-Volterra system \eqref{LV-sys}, its global dynamics
is still poorly understood except for some cases which will be mentioned briefly below.

In the remarkable work \cite{Girardin-Lam2018}, Girardin and Lam investigated {the} system \eqref{LV-sys} in the strong-weak
({one of the monostable cases}) competition case (i.e., $a<1<b$) with initial data being null or exponentially decaying in a right half-line. By constructing very technical pairs of  supersolutions and subsolutions, they gained a rather complete understanding of the spreading properties of \eqref{LV-sys}. Among other things, they found the acceleration phenomena during the period of invasion in some cases; see \cite{Girardin-Lam2018}
{for more results and more precision}. One may also refer to Lewis, Li and Weinberger \cite{LLW2002,LWL2005} for previous studies in the monostable case. On the other hand, the analogous problem with free boundaries was addressed in \cite{DuWu2018}, where the behavior of the slower species is determined by some semi-wave system studied in \cite{DWZ2017}.

In the weak competition case (i.e., $a,b<1$), Lin and Li \cite{LL2012} considered \eqref{LV-sys}, where both the initial functions
have compact support. They obtained the spreading speed of the faster species and some estimates for the
speed of the slower species.
Recently, Liu, Liu and Lam \cite{LLL2020a,LLL2020b} obtained a rather complete result by using {a large deviations approach}.
{It is worth mentioning that Iida, Lui and Ninomiya \cite{ILN11} considered stacked invasion waves in cooperative systems of $N$-species
with equal
diffusion coefficients. Under certain conditions, they found that species develop into stacked fronts
and spread at different speeds.}

In the strong (bistable) competition case (i.e., $a,b>1$), Carrere \cite{Carrere} considered \eqref{LV-sys}.
It was proved that if the two species are initially absent from the
right half-line $x>0$, and the slower one dominates the faster one on $x<0$, then the latter will invade the
right space at its Fisher-KPP speed, and will be replaced by or will invade the former, depending on the
parameters, at a slower speed.
We also mention the work \cite{DGM2019}, therein the authors proved that prey-predator systems can develop different spreading speeds.

{
The term "propagating terraces" (a layer of several traveling fronts) introduced by Ducrot, Giletti and Matano \cite{DGM2014}
(see also \cite{GM} for more general results) were used to investigate complicated propagation dynamics between the
two equilibria in spatially periodic equations. This notion can be found in the work of Fife and McLeod \cite{FM1977} in homogeneous equations, but under the name "minimal decomposition".
More general results for semilinear parabolic equations with front-like initial data in homogeneous environments were established by Pol${\rm \acute{a}\check{c}}$ik \cite{Polacik}.
}

The current paper focuses on the strong competition case, and our primary goal is to derive the
sharp dynamical behavior of the solution of \eqref{LV-sys} {when the successful spread of $u$ occurs}.
We are concerned with two typical situations:
either one species is an invasive one and the other is a native one or both are invasive species.
{The results we obtained here substantially complement}
and improve those in \cite{Carrere}. To our knowledge, the main results of this paper seem to be the first that give the precise estimates for the spreading speed of {the} system \eqref{LV-sys} with strong competition.

Since the competition model enjoys the comparison principle, our main results are established by the delicate construction of supersolutions and subsolutions. To this aim, we first derive some good decay estimates of the solution as $t$ is sufficiently large. Based on such estimates, we then construct various types of supersolutions and subsolutions, which turn out to be very new and optimal in certain sense. It is worth mentioning that in \cite{Girardin-Lam2018}, Girardin and Lam also adopted the approach of supersolution and subsolution to establish their main results. Nevertheless, the pairs of supersolutions and subsolutions constructed here are rather different from those used in \cite{Girardin-Lam2018}, mainly due to the essential differences between the strong competition problem and strong-weak competition problem. On the other hand, to derive the convergence results including a Bramson correction (refer to Theorem \ref{thm2} and Theorem \ref{thm3} below), we reduce {the} system \eqref{LV-sys} into a perturbed Fisher-KPP equation and then the argument used in \cite{HNRR} can be applied to obtain the Bramson correction. See also \cite{CTW2017} for the Bramson correction in an SIS model.

Before presenting the main results of the paper, we need to state some assumptions and introduce some notations.
From now on, we always assume
\begin{itemize}
\item[{\bf(H1)}] the strong competition:\ \ $a,\,b>1$.
\end{itemize}

Under {\bf(H1)}, let us recall the well-known results on traveling front solutions corresponding to {the} system \eqref{LV-sys}, which are vital in describing the global dynamics of \eqref{LV-sys}. By a traveling front solution, we mean a solution of \eqref{LV-sys} with the form $$(u(t,x),v(x,t))=(U(x-ct), V(x-ct))$$ and {existing and unequal limits $(U,V)(\pm \infty)$},
where $c$ is called the wave speed.
From Gardner \cite{G1982} and Kan-on \cite{Kan-on95}, {the} system
\eqref{LV-sys} admits a unique (up to a translation) traveling front solution connecting steady states $(1,0)$ and $(0,1)$.
More precisely, there exists a unique speed
 $$c_{uv}\in(-2,2\sqrt{rd})$$
such that when $c=c_{uv}$, the following problem
\bea\label{TW sys}
\begin{cases}
cU'+d U''+rU(1-U-a V)=0,\quad \xi\in\mathbb{R},\\
cV'+V''+V(1-V-b U)=0,\quad\ \ \ \xi\in\mathbb{R},\\
(U,V)(-\infty)=(1,0),\quad (U,V)(\infty)=(0,1),\\
U'<0,\ V'>0,\quad \quad \xi\in\mathbb{R}
\end{cases}
\eea
has a unique (up to a translation) solution $(U,V)\in [C^2(\mathbb{R})]^2$. By our notation, $c_{uv}<c_u$.

In this paper, we also assume that
\begin{itemize}
\item[{\bf(H2)}] $c_{uv}>0$.
\end{itemize}
{Some sufficient conditions} to guarantee {\bf(H2)} will be mentioned later. It is noted that if {\bf(H1)}
and $c_{uv}<0$ are fulfilled, the global dynamics of \eqref{LV-sys} may depend on the initial repartition of $u$ and $v$;
such a case shall not be studied in this paper.

Besides, we assume that the species $u$ always spreads successfully in the following sense:
\begin{itemize}
\item[{\bf(H3)}] (Successful invasion of $u$) $\lim_{t\to\infty}(u,v)(t,x)=(1,0)$ locally uniformly in $\mathbb{R}$.
\end{itemize}

\begin{remark}
Whether {the condition {\bf(H3)}} holds or not depends on initial data $(u_0,v_0)$.
{Roughly speaking, it holds if $(u_0,v_0)\approx (1,0)$ in a sufficiently large interval.}
Further discussion will be addressed in Remark~\ref{rk-H3} {after we construct a suitable subsolution}.
\end{remark}

Regarding initial data $(u_0,v_0)$, we consider two different scenarios:
\begin{itemize}
\item[{\bf(A1)}] $u_0\in C(\mathbb{R})\setminus\{0\}$, $u_0\geq0$ with compact support; $v_0\in C(\mathbb{R})\cap L^{\infty}(\mathbb{R})$ with a positive lower bound.
\item[{\bf(A2)}] $u_0,v_0\in C(\mathbb{R})\setminus\{0\}$, $u_0,\, v_0\geq0$ with compact support.
\end{itemize}
Scenario {\bf(A1)} means that species $u$ is the invasive species that initially occupies some bounded interval
and species $v$ is the native species that has already occupied the whole space; while
scenario {\bf(A2)} means that both two species are invasive species that initially occupy only open bounded intervals.

For convenience, let us lump conditions {{\bf(H1)}-{\bf(H3)}} together as condition {\bf(H)}.
{In this paper, {\bf(H)} is always assumed.}
We are now in a position to present the main results obtained in this paper.

Our first main result concerns scenario {\bf(A1)} and
{show the spreading profile of $u$ and $v$ under}
the successful invasion of species $u$
if $v$ is the native species.

\begin{thm}\label{thm1}
Assume that {\bf(H)} and {\bf(A1)} hold. Then there exists a constant $\hat{h}$ such that the solution $(u,v)$ of \eqref{LV-sys}-\eqref{LV-ic}
satisfies
\bea\label{AS-thm1}
&&\lim_{t\to\infty}\left[\sup_{x\in[0,\infty)}\Big|u(t,x)-U(x-c_{uv}t-\hat{h})\Big|+\sup_{x\in[0,\infty)}\Big|v(t,x)-V(x-c_{uv}t-\hat{h})\Big|\right]=0,
\eea
where $(c_{uv},U,V)$ is a solution of \eqref{TW sys}.
\end{thm}

Theorem~\ref{thm1} is related to the stability of traveling fronts;
{a classical reference on this issue is the monograph \cite[Chapter 5]{VVV}.
Theorem~\ref{thm1} is proved by the super-sub solutions approach.
Another approach to study the convergence to bistable waves is the dynamical systems approach \cite{Zhao2003}.
See, e.g., \cite{BW2013} that investigated the existence and stability of pulsating waves
in time periodic environments.
}
We {also} refer to \cite{FH2019} for critical pulled fronts of \eqref{LV-sys} with $a<1<b$ and
\cite{Tsai2007} for a buffered bistable system.

Our next two main results concern scenario {\bf(A2)}; that is, both species are invasive ones.
It turns out that $c_u$ and $c_v$ play an important role to determine the dynamical behavior of solutions.

We first consider the case $c_{u}>c_{v}$. In this case, the following result shows that $u$ spreads faster than $v$; $u$ will drive $v$ to extinction in the long-run while $u$ converges to a shifted traveling front with a Bramson correction \cite{Bramson,HNRR,Lau,Uchiyama1978}.

\begin{thm}\label{thm2}
Assume that {\bf(H)} and {\bf(A2)} hold.
If  $c_{u}>c_{v}$, then the solution $(u,v)$ of \eqref{LV-sys}-\eqref{LV-ic}
satisfies
\beaa
&&\lim_{t\to\infty}\left[\sup_{x\in[0,\infty)}\Big|u(t,x)-U_{KPP}\Big(x-c_{u} t+{\frac{3{d}}{c_u}}\ln t+\omega(t)\Big)\Big|+\sup_{x\in[0,\infty)}\Big|v(t,x)\Big|\right]=0,\\
\eeaa
where $\omega$ is a bounded function defined on $[0,\infty)$.
\end{thm}

Finally, we handle the case $c_{u}<c_{v}$. Then $c_{uv}<c_{u}<c_{v}$. In this case, the following result suggests that the species $u$ spreads at the slower speed $c_{uv}$ and the species $v$ spreads at the speed $c_v$ and thus a propagating terrace is formed. Though this phenomenon was proved in \cite{Carrere}, our result gives the sharp estimates for the spreading speed of the solution.

\begin{thm}\label{thm3}
Assume that {\bf(H)} and {\bf(A2)} hold, and that $c_{u}<c_{v}$. Denote $c_0=\frac{c_{uv}+c_v}{2}$.
Then the solution $(u,v)$ of \eqref{LV-sys}-\eqref{LV-ic} satisfies
\beaa
&&\lim_{t\to\infty}\left[\sup_{x\in[c_0t,\infty)}\Big|v(t,x)-V_{KPP}(x-c_{v} t+\frac{3}{c_v}\ln t+\omega(t))\Big|+\sup_{x\in[c_0t,\infty)}\Big|u(t,x)\Big|\right]=0
\eeaa
and
\beaa
&&\lim_{t\to\infty}\left[\sup_{x\in[0,c_0t)}
\Big|u(t,x)-U(x-c_{uv}t-h_1)\Big|+\sup_{x\in[0,c_0t)}\Big|v(t,x)-V(x-c_{uv}t-h_1)\Big|
\right]=0
\eeaa
for some bounded function $\omega$ on $[0,\infty)$ and
some $h_1\in\mathbb{R}$, where $(c_{uv},U,V)$ is a solution of \eqref{TW sys},
\end{thm}

Some comments on Theorem \ref{thm1}-\ref{thm3} are made in order as follows.

\begin{remark} The sign of $c_{uv}$ has been investigated in the literature. Indeed, Kan-on \cite{Kan-on95} proved that $c_{uv}$
is decreasing in $a$ and is increasing in $b$. Guo and Lin \cite{GL13} provided explicit conditions to determine the sign of $c_{uv}$; in particular, their results conclude that
\begin{itemize} \item[\rm{(i)}] When $r=d$, then $c_{uv}>0$ if $b>a>1$, $c_{uv}=0$ if $a=b>1$ and $c_{uv}<0$ if $a>b>1$.

\item[\rm{(ii)}] When $r>d$, then $c_{uv}>0$ if $a>1$ and $b\geq\big(\frac{r}{d}\big)^2a$.

\item[\rm{(iii)}] When $r<d$, then $c_{uv}<0$ if $b>1$ and $a\geq\big(\frac{d}{r}\big)^2b$.

\end{itemize}
In addition, it can be shown that if $r,\,d>0$ and $a>1$ are fixed, $c_{uv}>0$ for all large $b$. One may also see Girardin and Nadin \cite{GN15}, Rodrigo and Mimura \cite{RM00} and Ma, Huang and Ou \cite{MHO} for related discussion.
{We also refer to Girardin \cite{Girardin2019} for a recent survey on this issue.}
\end{remark}

\begin{remark} We would like to mention the following.
\begin{itemize} \item[\rm{(i)}] Similar results of Theorem~\ref{thm1}-Theorem~\ref{thm3} also hold for $x\in(-\infty,0]$ since the arguments used on the right half-line work on its left half-line in the strong-competition system.


\item[\rm{(ii)}] The techniques developed in this paper may be applicable to more general competition systems \eqref{LV-sys} as well as other parabolic systems including cooperative systems with arbitrary size.

{\item[\rm{(iii)}] Another related issue is the entire solutions (classical solutions defined globally in
time and space) for \eqref{LV-sys}. Morita and Tachibana \cite{MT2009} established the existence of two-front entire solutions which behave as two fronts approaching from both sides of $x$-axis based on the suitable  construction of super and subsolutions.
With the similar idea but more complicated construction of super and subsolutions, the existence of three- and four-front entire solutions was proved in \cite{GW19}. The super and subsolutions constructed in \cite{GW19,MT2009} are only defined for $t\leq t_0$ for some $t_0\in\mathbb{R}$, which cannot yield a sharp convergence result as $t\to\infty$. Therefore, the techniques of our work may be used to improve the results on
the asymptotic behavior of these entire solutions as $t\to\infty$ in the bistable case. We also refer to the recent work of
Lam, Salako and Wu \cite{LSW} that successfully establishes various new types of entire solutions for \eqref{LV-sys} and gains
a better understanding on the behavior of these entire solutions as $t\to\infty$.
}

{\item[\rm{(iv)}]
Theorem~\ref{thm3} shows that the system develops a propagating terrace, connecting the unstable state $(0,0)$ to
the two stable states $(1,0)$ and $(0,1)$. This can be seen as a system version of the finding of propagating terraces reported in \cite{Polacik}.}
\end{itemize}
\end{remark}

The remainder of this paper is organized as follows.
In section \ref{preliminaries}, we shall  prepare some well-known results and provide important estimates of the solution of \eqref{LV-sys}-\eqref{LV-ic} that will be used in both {\bf(A1)} and {\bf(A2)}.
Section \ref{sec:A1} is devoted to the proof of Theorem~\ref{thm1}, and Theorem~\ref{thm2} and Theorem~\ref{thm3} are
proved in Section \ref{sec:A2}.

\section{Preliminaries}\label{preliminaries}
\setcounter{equation}{0}

In this section, we prepare some preliminary results that will be used in both cases: {\bf(A1)} and {\bf(A2)}.
In the first subsection, we recall the exact exponential decays of traveling front solution of \eqref{TW sys} connecting $(0,1)$
and $(1,0)$. In the second subsection, we recall the comparison principle for {the} system \eqref{LV-sys}-\eqref{LV-ic}.
Some crucial estimates of solutions to system \eqref{LV-sys}-\eqref{LV-ic} are given in the third subsection.

\subsection{The asymptotic behavior of  bistable fronts}

The asymptotic behavior of the traveling front solution for \eqref{LV-sys} with $c=c_{uv}\neq0$ as $\xi\to\pm\infty$ is well known;
we refer to \cite{KF1996} or \cite[section 2]{MT2009}.
Here we state the results that will be used in the rest of this paper.

Let $(c,U,V)$ be a solution of {the} system \eqref{TW sys}.
To describe the asymptotic behavior of $(U,V)$ near $\xi=+\infty$, we need the following characteristic equations:
\bea
&&c\lambda+d\lambda^2+r(1-a)=0,\label{chara eq-R1}\\
&&c\lambda+\lambda^2-1=0.\label{chara eq-R2}
\eea
Let $\lambda_{1}<0$ (resp., $\lambda_{2}<0$) be the negative root of \eqref{chara eq-R1} (resp., \eqref{chara eq-R2}), i.e.,
\beaa
\lambda_{1}=\frac{-c-\sqrt{c^2+4rd(a-1)}}{2d},\quad \lambda_{2}=\frac{-c-\sqrt{c^2+4}}{2}.
\eeaa

\begin{lemma}[\cite{KF1996,MT2009}]\label{lem:AS+}
There exist two positive constants $\ell_1$ and $\ell_2$ such that
\beaa
\lim_{\xi\to+\infty}\frac{U(\xi)}{e^{{\lambda_{1}}\xi}}=\ell_1, \quad \lim_{\xi\to+\infty}\frac{1-V(\xi)}{|\xi|^{\gamma_+}e^{\Lambda_+\xi}}=\ell_2,\quad
\eeaa
where $\Lambda_+:=\max\{\lambda_{1},\lambda_{2}\}<0$ and
\beaa
\gamma_+=\begin{cases}
          0,\quad {\rm if}\    \lambda_{1}\neq \lambda_{2}, \\
          1,\quad {\rm if} \ \lambda_{1}= \lambda_{2}.
         \end{cases}
\eeaa
\end{lemma}

For the asymptotic behavior of $(U,V)$ near  $\xi=-\infty$, we need the following characteristic equations:
\bea
&&c\lambda+d\lambda^2-r=0,\label{chara eq-L1}\\
&&c\lambda+\lambda^2+1-b=0.\label{chara eq-L2}
\eea
Let $\lambda_{3}>0$ (resp., $\lambda_{4}>0$) be the positive root for \eqref{chara eq-L1} (resp., \eqref{chara eq-L2}).
Namely,
\beaa
\lambda_{3}=\frac{-c+\sqrt{c^2+4rd}}{2d},\quad \lambda_{4}=\frac{-c+\sqrt{c^2+4(b-1)}}{2}.
\eeaa

\begin{lemma}[\cite{KF1996,MT2009}]\label{lem:AS-}
There exist two positive constants $\ell_3$ and $\ell_4$ such that
\beaa
\lim_{\xi\to-\infty}\frac{1-U(\xi)}{|\xi|^{\gamma_-}e^{\Lambda_-\xi}}=\ell_3,\quad \lim_{\xi\to-\infty}\frac{V(\xi)}{e^{{\lambda_{4}}\xi}}=\ell_4,
\eeaa
where $\Lambda_-:=\min\{\lambda_{3},\lambda_{4}\}>0$ and
\beaa
\gamma_-:=\begin{cases}
          0,\quad {\rm if}\    \lambda_{3}\neq \lambda_{4}, \\
          1,\quad {\rm if} \ \lambda_{3}= \lambda_{4}.
         \end{cases}
\eeaa
\end{lemma}

\subsection{Comparison principle}

It is well known that {the} system \eqref{LV-sys}-\eqref{LV-ic} can be reduced to a {cooperative} system,
which {satisfies the comparison principle} (see, e.g., \cite{CC}).
For {the} reader's convenience, we recall the notion of super and subsolutions and the comparison principle.

Define the differential operators
\beaa
&&N_1[u,v](t,x):=u_t-d u_{xx}-ru(1-u-av),\quad N_2[u,v](t,x):=v_t-v_{xx}-v(1-v-bu).
\eeaa
We say that $(\bar{u},\underline{v})$ with $(\bar{u},\underline{v})\in [C(\overline{D})\cap C^{2,1}(D)]^2$
is a pair of supersolution of \eqref{LV-sys} in
$$D:=(\tau,T)\times (\zeta_1,\zeta_2),\ \ 0\leq \tau<T\leq \infty,\ \ {-\infty\leq \zeta_1<\zeta_2\leq+\infty}$$
if $(\bar{u},\underline{v})$ satisfies
{$N_1[\bar{u},\underline{v}]\geq0$ and $N_2[\bar{u},\underline{v}]\leq0$ in $D$.}
A pair of subsolution $(\underline{u},\bar{v})$ of \eqref{LV-sys} in $D$ can be defined analogously by reversing
all inequalities.

The following is the standard comparison principle (see, e.g., \cite{PW1984}).

\begin{lemma}[Comparison Principle]\label{CP}
Suppose that $(\bar{u},\underline{v})$ is a supersolution of \eqref{LV-sys} in $D:=(\tau,T)\times (\zeta_1,\zeta_2)$,  and
$(\underline{u},\bar{v})$ is a subsolution of \eqref{LV-sys} in $D$.
{If
\bea\label{supersol}
\begin{cases}
\bar{u}(\tau,\cdot)\geq \underline{u}(\tau,\cdot),\quad \underline{v}(\tau,\cdot)\leq \bar{v}(\tau,\cdot)\quad \mbox{in $(\zeta_1,\zeta_2)$},\\
\bar{u}(t,\zeta_i)\geq \underline{u}(t,\zeta_i),\quad \underline{v}(t,\zeta_i)\leq \bar{v}(t,\zeta_i)\quad \mbox{for $t\in(\tau,T)$ and $i=1,2$},
\end{cases}
\eea
}
then
$\bar{u}\geq\underline{u}$ and $\bar{v}\geq\underline{v}$ in $D$.
\end{lemma}
When $\zeta_1=-\infty$ or $\zeta_2=\infty$, the corresponding boundary condition (the {second} condition) in \eqref{supersol} is omitted.

\begin{remark}
The definition of super and subsolutions can be weakened slightly. For example,
when both $(\underline{u}_1,\overline{v})$ and $(\underline{u}_2,\overline{v})$ are subsolution in $D$, then
$(\max\{\underline{u}_1,\underline{u}_2\},\overline{v})$ can be referred to as a subsolution in $D$ such that
the comparison principle remains true. We refer to \cite{Girardin-Lam2018} for more discussion.
\end{remark}

\medskip

\subsection{Some crucial estimates}

In this subsection, we present several lemmas to provide crucial estimates of the solution $(u,v)$ to problem \eqref{LV-sys}-\eqref{LV-ic},
which play an important role in our analysis.
{Almost all results hold for both {\bf(A1)} and {\bf(A2)}.
Otherwise, we will emphasize in the statement of the result. }

\begin{lemma}\label{lem:simple est}
There exist $M>0$ 
such that
\bea
&&u(t,x)\leq 1+Me^{-r t}, \ \ \ \forall t\geq {0},\ x\in\mathbb{R},\label{M1-bound}\\
&&v(t,x)\leq 1+Me^{- t},\ \ \ \ \forall t\geq {0},\ x\in\mathbb{R}.\label{M2-bound}
\eea
\end{lemma}
\begin{proof}
Consider the ODE problem
 $$
 w'(t)=rw(1-w),\ \ w(0)=\|u_0\|_{L^{\infty}}:=w_0.
 $$
By an elementary calculation, we have
\beaa
w(t)=\frac{w_0}{w_0+(1-w_0)e^{-r t}},\quad t\geq0.
\eeaa
Clearly, there exists
positive constant 
$M$ such that
$w(t)\leq 1+Me^{-r t}$ for $t\geq {0}$. Then \eqref{M1-bound} follows by comparing $u(t,x)$ and $w(t)$.  Similarly,
\eqref{M2-bound} holds true.
\end{proof}

\begin{lemma}\label{lem:exp-decay}
If $c>c_u:=2\sqrt{rd}$, then 
there exist $M,\mu>0$ and $T\gg1$ such that
\beaa
u(t,x)\leq M e^{-\mu [c-{2\sqrt{rd}}\,]t},\ \ \  \forall t\geq T,\ x>ct.
\eeaa
\end{lemma}
\begin{proof}
Let {$U_{KPP}$} be the solution of
\beaa
\begin{cases}
{c_{u}} U'+dU''+r(1-U)U=0,\quad \xi\in\mathbb{R},\\
{U(-\infty)=1},\quad U(+\infty)=0,\quad U(0)=1/2,
\end{cases}
\eeaa
where ${c_{u}=2\sqrt{rd}}$.
Recall from \cite{KPP} that there exists $C>0$ such that
\bea\label{KPP-AS}
{U_{KPP}(\xi)}\sim C\xi e^{-[{c_{u}}/(2d)]\xi},\ \quad \mbox{as $\xi\to\infty$}.
\eea

Define
\beaa
\overline{u}(t,x):={M U_{KPP}(x-c_u t)}, \quad \underline{v}(t,x)=0
\eeaa
{for some large $M>1$ such that $MU_{KPP}(x)\geq u_0(x)$ for all $x\in\mathbb{R}$.}
It is easy to check that
$$N_1[\overline{u},\underline{v}](t,x)={rM(M-1)U^2_{KPP}}\geq0,\ \ N_2[\overline{u},\underline{v}](t,x)=0\ \
\mbox{in}\ [T,\infty)\times \mathbb{R}.$$ Thus, by comparison, we have
$\overline{u}(t,x)\geq u(t,x)$  in $[T,\infty)\times \mathbb{R}$, and in turn, for all {$t\geq 0$} and $x>c t$,
\beaa
u(t,x)\leq { MU_{KPP}((c-c_u)t)},
\eeaa
which together with \eqref{KPP-AS}, completes the proof.
\end{proof}

Next, we establish an exponential decay rate of $v$.

\begin{lemma}\label{lem:v-rate}
For any given $c\in(0,c_{uv})$, 
there exist positive constants {$\rho$}, $T$ and $M$ such that
\beaa
 v(t,x)\leq M e^{-{\rho} t},\quad
 \forall t\geq T,\ x\in[-ct, ct].
\eeaa
\end{lemma}
\begin{proof}
{Inspired by the proof of}  (9) in \cite{Carrere}, {one can show}
\bea\label{segregation-uv}
\lim_{t\to\infty}\left[\max_{x\in[-ct,ct]}|u(t,x)-1|+\max_{x\in[-ct,ct]}v(t,x)\right]=0.
\eea
{
Indeed, the proof of \eqref{segregation-uv} is based on a suitable construction of a subsolution by perturbing the system \eqref{TW sys}.
We have to note that the subsolution constructed in \cite{Carrere} cannot apply to our case directly because initial data
in \cite{Carrere} are different from {\bf(A1)} or {\bf(A2)}.
After a suitable modification of the subsolution used in \cite{Carrere}, one can verify \eqref{segregation-uv}.
The proof is rather long and tedious, and we put it in the Appendix.
}

By \eqref{segregation-uv},
for any given small $\epsilon>0$, there exists $T\gg1$ such that
\bea
&&0\leq v(t,x)\leq\epsilon\quad\mbox{for all $t\geq T$ and $x\in[-ct,ct]$},\label{v-ep}\\
&&u(t,x)\geq 1-\epsilon\quad\mbox{for all $t\geq T$ and $x\in[-ct,ct]$}\label{u-ep}.
\eea
By 
\eqref{u-ep}, we see from $v$ equation in \eqref{LV-sys} that
\bea\label{v-ineq-gamma}
v_t\leq v_{xx}-\gamma_{\epsilon} v\quad\mbox{for all $t\geq 0$ and $x\in[-ct,ct]$},
\eea
{where $\gamma_\epsilon:=b(1-\epsilon)-1>0$ (if necessary, we choose $\epsilon$ smaller to ensure $\gamma_\epsilon>0$).}

Given $L>0$, consider the following fixed boundary problem
\beaa
\begin{cases}
\psi_t=\psi_{xx}-\gamma_{\epsilon} \psi, \quad t>0,\ -L<x<L,\\
\psi(t,\pm L)=\epsilon,\quad t>0,\\
\psi(0,x)=\epsilon,\quad -L<x<L.
\end{cases}
\eeaa
Note that the above problem admits the unique positive steady state
\beaa
\chi(x):=
\left[\frac{e^{\sqrt{\gamma_{\epsilon}}x}+e^{-\sqrt{\gamma_{\epsilon}}x}}{e^{\sqrt{\gamma_{\epsilon}}L}+e^{-\sqrt{\gamma_{\epsilon}}L}}\right]\epsilon,\quad -L\leq x\leq L.
\eeaa

Denote $$\Psi(t,x)=\psi(t,x)-\chi(x).$$ After some simple calculations, $\Psi$ solves
\beaa
\begin{cases}
\Psi_t=\Psi_{xx}-\gamma_{\epsilon}\Psi, \quad t>0,\ -L<x<L,\\
\Psi(t,\pm L)=0,\quad t>0,\\
\Psi(0,x)=\epsilon-\chi(x)\leq \epsilon,\quad -L<x<L.
\end{cases}
\eeaa
By a simple comparison (with an obvious ODE problem), we have
\beaa
0\leq \Psi(t,x)\leq \epsilon e^{-\gamma_{\epsilon}t}\quad\mbox{for $t>0$ and $-L\leq x\leq L$},
\eeaa
which gives
\beaa
\psi(t,x)\leq \epsilon\Big(e^{-\gamma_{\epsilon} t}+\frac{e^{\sqrt{\gamma_{\epsilon}}x}+e^{-\sqrt{\gamma_{\epsilon}}x}}{e^{\sqrt{\gamma_{\epsilon}}L}+e^{-\sqrt{\gamma_{\epsilon}}L}}\Big)\quad\mbox{for $t>0$ and $-L\leq x\leq L$}.
\eeaa
{Therefore, \beaa
\lim_{t\to\infty} \psi(t,x)=\chi(x)\leq\epsilon \quad \mbox{uniformly for $x\in[-L,L]$}.
\eeaa
}
In particular, taking any $\sigma\in(0, 1/\sqrt{\gamma_{\epsilon}})$, we deduce
\beaa
\psi(t,x)
\leq \epsilon\Big(e^{-\gamma_{\epsilon}t}+\frac{2e^{\sqrt{\gamma_{\epsilon}}|x|}}{e^{\sqrt{\gamma_{\epsilon}}L}}\Big)
\leq \epsilon(e^{-\gamma_{\epsilon}t}+2e^{-\gamma_{\epsilon}\sigma L})
\eeaa
 for all $t>0$ and $x\in[-(1-\sqrt{\gamma_{\epsilon}}\sigma)L, (1-\sqrt{\gamma_{\epsilon}}\sigma)L]$.
{Therefore, one has
\beaa
\psi(t,x)\leq 3\epsilon e^{-\gamma_{\epsilon}\sigma L}\quad\mbox{for $t\geq \sigma L$ and $|x|\leq(1-\sqrt{\gamma_{\epsilon}}\sigma)L$},
\eeaa
Taking $L=c\widehat{T}$ for $\widehat{T}\geq T$,
by \eqref{v-ep} and \eqref{v-ineq-gamma}, one can apply the comparison principle to assert that
\beaa
v(t+\widehat{T},x)\leq \psi(t,x)\leq 3\epsilon e^{-\gamma_{\epsilon}\sigma c\widehat{T}},
\eeaa
for $t\geq \sigma c\widehat{T}$ and $|x|\leq(1-\sqrt{\gamma_{\epsilon}}\sigma)c\widehat{T}$.
In particular, taking $t=\sigma c\widehat{T}$, we have
\beaa
v(\sigma c\widehat{T} +\widehat{T},x)\leq 3\epsilon e^{-\gamma_{\epsilon}\sigma c\widehat{T}}\quad\mbox{for $|x|\leq(1-\sqrt{\gamma_{\epsilon}}\sigma)c\widehat{T}$}.
\eeaa

Note that $t=\sigma c\widehat{T} +\widehat{T}$ if and only if $\widehat{T}=t(\sigma c+1)^{-1}$. It follows that
\beaa
v(t,x)\leq 3\epsilon e^{-\gamma_{\epsilon}\sigma c(\sigma c+1)^{-1}t}\quad
\mbox{for $t\geq T^*$ and $|x|\leq(1-\sqrt{\gamma_{\epsilon}}\sigma)c(\sigma c+1)^{-1}t$},
\eeaa
where
$T^*=\sigma c{T} +{T}$. Since $c$ can be arbitrarily close to $c_{uv}$ and $\sigma>0$ can be arbitrarily small,
we see that the proof is complete.
}
\end{proof}

When {{\bf(A2)} holds and} $c_u>c_v$, as in {proving} (12) of \cite[p.2137]{Carrere}, one has
\bea\label{segregation-uv2}
\lim_{t\to\infty}\left[\max_{x\in[-ct,ct]}|u(t,x)-1|+\max_{x\in[-ct,ct]}v(t,x)\right]=0
\eea
for any $c\in(0,c_u)$. {
To see this,
we fix any $\hat{c}\in(c_v,c_u)$. As $\hat{c}>c_v$, by using a similar proof of Lemma~\ref{lem:exp-decay} (note that,
under {\bf(A2)}, $v_0$ has compact support), one has
$v(t,x)\leq M e^{-\nu(\hat{c})t}$ for some $\nu(\hat{c})>0$ for all $|x|\geq \hat{c}t$ and all large $t$. An argument used in
{\cite[Section 3]{Carrere}}
implies that $u(t,x)\to1$ as $t\to\infty$ uniformly for $\hat{c} t\leq |x|\leq \tilde{c} t$, where $\tilde{c}$ is any speed such that $\hat{c} <\tilde{c} <c_u$.
In particular, $u(t,-\hat{c}t)\to 1$ as $t\to\infty$.
This allows us to adopt the construction of a
subsolution $(\underline{u},\overline{v})$ used in \cite[Section 4.1]{Carrere} (with minor modifications) and compare the solution over $[T,\infty)\times [-\hat{c}t,\infty)$ for some $T\gg1$,
Consequently, we can derive
\bea\label{c-hat-speed}
\lim_{t\to\infty}\left[\max_{x\in[-\hat{c}t,0]}|u(t,x)-1|+\max_{x\in[-\hat{c}t,0]}v(t,x)\right]=0.
\eea
By symmetry, \eqref{c-hat-speed} still holds with $x\in[-\hat{c}t,0]$ replaced by $x\in[0,\hat{c}t]$.
Therefore, \eqref{segregation-uv2} holds.
}
Then, replacing \eqref{segregation-uv} by \eqref{segregation-uv2} and {following the lines} of the proof of Lemma~\ref{lem:v-rate},
one can obtain

\begin{cor}\label{cor:v-rate}
Assume that {{\bf(A2)} holds and} $c_u>c_v$.
For any given  $c\in(0,c_{u})$ and small $\epsilon>0$,
there exist
positive constants {$\rho'$}, $T$ and $M$ such that
\beaa
 v(t,x)\leq M e^{-{\rho'} t},\quad
 \forall t\geq T,\ x\in[-ct, ct].
\eeaa
\end{cor}

\begin{lemma}\label{lem:u-rate}
For any given $c\in(0,c_{uv})$, there exist positive constants $\delta$, $T$ and $M$ such that
\beaa
u(t,x)\geq 1-M e^{-\delta t},\quad \forall t\geq T,\ x\in[-ct,ct].
\eeaa
\end{lemma}
\begin{proof}
Thanks to Lemma~\ref{lem:v-rate},
there exist positive constants $T_1$, $M_1>0$ and $\delta_1>0$ such that
\beaa\label{v-est}
v(t,x)\leq M_1e^{-\delta_1 t},\quad \forall t\geq T_1,\ x\in[-c t, ct].
\eeaa
By \eqref{segregation-uv}, one can take $\eta>0$ close to $1$ and $\hat{T}\geq T_1$
such that
\beaa\label{u-eta}
u(t,x)\geq \eta,\quad \forall t\geq \hat{T},\ x\in[-c t, c t],
\eeaa
which also yields that
 $u(1-u)\geq \eta(1-u)$ for all $u\in[\eta,1]$.

To construct a subsolution of $u$-equation, we  consider
\bea\label{phi-subsys}
\begin{cases}
\phi_t=d\phi_{xx}+r[\eta(1-\phi)-aM_1 e^{-\delta_1(t+\hat{T})}\phi], \quad t>0,\ -c\hat{T}<x<c\hat{T},\\
\phi(t,\pm c\hat{T})=\eta,\quad t\geq 0,\\
\phi(0,x)=\eta,\quad -c\hat{T}\leq x\leq c\hat{T}.
\end{cases}
\eea
{It is obvious that $\phi_+\equiv1$ is a supersolution of \eqref{phi-subsys}.
By taking $\hat{T}$ large enough, $\phi_-\equiv \eta$ is a subsolution of \eqref{phi-subsys}.
Hence,}
$\eta\leq \phi\leq 1$.
It can be {further} seen that $\phi$ is a subsolution for the equation solved by $u(t+\hat{T},x)$ for $t>0$ and $-c\hat{T}\leq x\leq c\hat{T}$.

We now investigate the long-time behavior of $\phi$. For convenience, let us define
\beaa
q(t):=1+\frac{aM_1}{\eta}e^{-\delta_1(t+\hat{T})}.
\eeaa
Then, we can rewrite
$$r[\eta(1-\phi)-aM_1 e^{-\delta_1(t+\hat{T})}\phi]=r\eta-r\eta q(t)\phi.$$
Let us further define
\beaa
\Phi(t,x):=e^{Q(t)}[\phi(t,x)-\eta],\quad Q(t):=(r\eta)t-\frac{r aM_1}{\delta_1}e^{-\delta_1(t+\hat{T})}
\eeaa
such that $Q'(t)=r\eta q(t)$.
A straightforward computation {changes} $\phi$-equation into $\Phi$-equation:
\beaa
\begin{cases}
\Phi_t=d\Phi_{xx}+r\eta e^{Q(t)}[1-\eta q(t)], \quad t>0,\ -c\hat{T}<x<c\hat{T},\\
\Phi(t,\pm c\hat{T})=0,\quad t\geq 0,\\
\Phi(0,x)=0,\quad -c\hat{T}\leq x\leq c\hat{T}.
\end{cases}
\eeaa
Using the Green function of {the} heat equation, we have
\beaa
\Phi(t,x)=r\eta\int_0^t e^{Q(\tau)} [1-\eta q(\tau)]\int_{-c\hat{T}}^{c\hat{T}}\widetilde{G}(t,x;\tau,\xi)d\xi d\tau,\quad t>0,\ -c\hat{T}<x<c\hat{T},
\eeaa
where $\widetilde{G}(t,x;\tau,\xi)$ is the green function {(see, e.g., \cite[p.84]{Friedman})} defined by
\beaa\label{Green func}
\widetilde{G}(t,x;\tau,\xi)=\sum_{n\in\mathbb{Z}}(-1)^n G(t-\tau,x-\xi-2nc{\hat T}),
\eeaa
with the heat kernel $G$ given by
\beaa
G(t,x;\tau,\xi)=\frac{1}{\sqrt{4\pi d(t-\tau)}}e^{-\frac{(x-\xi)^2}{4d(t-\tau)}}.
\eeaa

In what follows, we will use an estimate given in \cite[Lemma 6.5]{DuLou}
(note that although $d=1$ therein, the same argument in \cite{DuLou} can yield the estimate for general $d$):
for any $\epsilon\in(0,1)$, there exists $T^{*}\gg1$ such that for all $\hat{T}\geq T^{*}$,
\beaa
\int_{-c \hat{T}}^{c\hat{T} }\widetilde{G}(t,x;\tau,\xi)d\xi d\tau\geq 1-\frac{4}{\sqrt{\pi}}e^{-\frac{\hat{T}}{2\sqrt{d}}}\quad\mbox{for all $(x,t)\in \hat{D}_{\epsilon}$},
\eeaa
where $\hat{D}_{\epsilon}$ is defined by
\beaa
\hat{D}_{\epsilon}:=\left\{(t,x):\  0< t\leq  \frac{\epsilon^2c^2 \hat{T}}{4\sqrt{d}},\quad |x|\leq(1-\epsilon)c \hat{T}\right\}.
\eeaa
In light of this estimate, we obtain
\beaa
\Phi(t,x)&\geq& r\eta \Big(1-\frac{4}{\sqrt{\pi}}e^{-\frac{\hat{T}}{2\sqrt{d}}}\Big)\int_0^t e^{Q(\tau)} [1-\eta q(\tau)]d\tau\\
         &\geq& r\eta \Big(1-\frac{4}{\sqrt{\pi}}e^{-\frac{\hat{T}}{2\sqrt{d}}}\Big) \Big(1-\eta -aM_1 e^{-\delta_1 \hat{T}}\Big) \int_0^t e^{Q(\tau)} d\tau
\eeaa
for all $(t,x))\in \hat{D}_{\epsilon}$.

Recalling the definition  of $\Phi$, we have $$\phi(t,x)=e^{-Q(t)}\Phi(t,x)+\eta.$$ Then
\bea\label{phi-calcu}
\phi(t,x)\geq \Big(1-\frac{4}{\sqrt{\pi}}e^{-\frac{\hat{T}}{2\sqrt{d}}}\Big) \Big( 1-\eta -aM_1 e^{-\delta_1 \hat{T}}\Big) r\eta e^{-Q(t)} \int_0^t e^{Q(\tau)}d\tau  +\eta
\eea
for all $(t,x)\in \hat{D}_{\epsilon}$.
By some simple calculations, we see that
\beaa
r\eta e^{-Q(t)} \int_0^t e^{Q(\tau)}d\tau&=&(r\eta)e^{-r\eta t+K e^{-\delta_{1}(t+\hat{T})}}
\left[\int_0^t e^{r\eta \tau-K e^{-\delta_{1}(\tau+\hat{T})}}d\tau\right]\\
&\geq& (r\eta) e^{-r\eta t} e^{K[e^{-\delta_{1}(t+\hat{T})}-e^{-\delta_{1} \hat{T}}]}{[\frac{1}{r\eta}e^{r\eta\tau}|_{\tau=0}^{\tau=t}]}\\
&=&{ e^{K[e^{-\delta_{1} \hat{T}}(e^{-\delta_{1} t}-1)]} }
    (1-e^{-r\eta t})\\
&=:& {J(t),}
\eeaa
where $K:=raM_1/\delta_1$. {Note that $J(t)\leq1$ for all $t\geq0$.}

Plugging this estimate into \eqref{phi-calcu}, we have
\beaa
\phi(t,x)&\geq& {J(t)\Big(1-\frac{4}{\sqrt{\pi}}e^{-\frac{\hat{T}}{2\sqrt{d}}}\Big) \Big( 1-\eta -aM_1 e^{-\delta_1 \hat{T}}\Big)+\eta}\\
&=& {J(t) \Big(1-\frac{4}{\sqrt{\pi}}e^{-\frac{\hat{T}}{2\sqrt{d}}}\Big) \Big( 1-aM_1 e^{-\delta_1 \hat{T}}\Big)
    +\eta\Big[1-\Big(1-\frac{4}{\sqrt{\pi}}e^{-\frac{\hat{T}}{2\sqrt{d}}}\Big)J(t)\Big] }\\
&\geq&
e^{K[e^{-\delta_1\hat{T}}(e^{-\delta_1 t}-1)]}
(1-e^{-r\eta t})\Big(1-\frac{4}{\sqrt{\pi}}e^{-\frac{\hat{T}}{2\sqrt{d}}}\Big)\Big[1-{a}M_1 e^{-\delta_1 \hat{T}}\Big]
\eeaa
for all $(t,x)\in \hat{D}_{\epsilon}$.
By the fact that $e^{x}\geq 1{+}x$ for all $x$, {and $\delta_1$ can be chosen smaller such that
  $\delta_1<1/(2\sqrt{d})$}, we then obtain
\beaa
\phi(t,x)&\geq& [1-Ke^{-\delta_{1} \hat{T}} (1-e^{-\delta_{1} t})] (1-e^{-r\eta t})\Big(1-\frac{4}{\sqrt{\pi}}e^{-\frac{\hat{T}}{2\sqrt{d}}}\Big) \Big[1-{a}M_1 e^{-\delta_1 \hat{T}}\Big]\\
         &\geq&{(1-Ke^{-\delta_{1} \hat{T}})(1-e^{-r\eta t})\Big(1-\frac{4}{\sqrt{\pi}}e^{-\delta_1\hat{T}}\Big)
         \Big(1-{a}M_1 e^{-\delta_1 \hat{T}}\Big)}\\
         &\geq& 1-{\hat{K}}e^{-\delta_{1}\hat{T}}-  e^{-r\eta t}
\eeaa
for all $(t,x)\in\hat{D}_\epsilon$, by taking $\hat T$ larger if necessary,
{where $\hat{K}$ is some large positive constant.}

Set $t=\epsilon^2 c^2 \hat{T}/(4\sqrt{d})$ and $\epsilon>0$ small enough such that
\beaa
\frac{r\eta \epsilon^2 c^2}{4\sqrt{d}}<\delta_1,
\eeaa
we obtain
\bea
\phi\Big(\frac{\epsilon^2 c^2 \hat{T}}{4\sqrt{d}},x\Big)&\geq& 1-{\hat{K}}e^{-\delta_1 \hat{T}}-  e^{-r\eta \epsilon^2 c^2 \hat{T}/(4\sqrt{d})}\label{phi-est}\\
                                   &\geq&1-({\hat{K}}+1) e^{-r\eta \epsilon^2 c^2 \hat{T}/(4\sqrt{d})}.\notag
                                   \eea
The parabolic comparison principle gives $u(t+\hat{T},x)\geq \phi(t,x)$, which together with \eqref{phi-est} implies
\beaa
u\Big(\frac{\epsilon^2 c^2 \hat{T}}{4\sqrt{d}}+\hat{T},x\Big)&\geq&1-({\hat{K}}+1) e^{-r\eta \epsilon^2 c^2 \hat{T}/(4\sqrt{d})}
\eeaa
for all $|x|\leq(1-\epsilon)c \hat{T}$.
Note that
\beaa
t=\frac{\epsilon^2 c^2 \hat{T}}{4\sqrt{d}}+\hat{T}\quad \Longleftrightarrow\quad \hat{T}=\Big(1+\frac{\epsilon^2 c^2}{4\sqrt{d}}\Big)^{-1}t.
\eeaa
This yields that
\beaa
u(t,x)\geq 1-Me^{\delta_2 t}\quad\mbox{for}\ |x|\leq  {(1-\epsilon)}c\Big(1+\frac{\epsilon^2 c^2}{4\sqrt{d}}\Big)^{-1}t,\quad t\geq T^{**},
\eeaa
where
\beaa
M={\hat{K}}+1,\ \ \delta_2:=r\eta\Big(\frac{\epsilon^2 c^2}{4\sqrt{d}}\Big)\Big(1+\frac{\epsilon^2 c^2}{4\sqrt{d}}\Big)^{-1}>0,\quad T^{**}:=T^{*}+\frac{\epsilon^2 c^2}{4\sqrt{d}}T^{*}.
\eeaa
Since $c$ can be arbitrarily close to $c_{uv}$ and  $\epsilon>0$ can be arbitrarily small, we thus {complete} the proof.
\end{proof}

{
\begin{remark}
The proof in Lemma~\ref{lem:v-rate} and Lemma~\ref{lem:u-rate} provides a method to prove
$(u,v)\to(1,0)$ exponentially over some region once we have known
the locally uniformly convergence of $(u,v)$.
In Lemma~\ref{lem:order-4} below, we will provide an independent proof for Lemma~\ref{lem:v-rate} and Lemma~\ref{lem:u-rate}
by constructing a refined subsolution.
More precisely, given $c\in(0,c_{uv})$, from  Lemma~\ref{lem:order-4}, we obtain that, for some large $T$,
\beaa
u(t,x)&\geq& 2U\Big((c-c_{uv})t-c_{uv}\hat{T}+\zeta_0-\zeta_1e^{-(\beta/2) (t+\hat{T})}\Big)-1-\hat{p}_0e^{-\beta (t+\hat{T})},\\
v(t,x)&\leq& 2(1+\hat{q}_0e^{-\beta (t+\hat{T})})V\Big((c-c_{uv})(t+\hat{T})+\zeta_0-\zeta_1e^{-(\beta/2) (t+\hat{T})}\Big),
\eeaa
for all $t\geq T$ and $|x|\leq ct$, where $\hat{q}_0>0$, $\beta>0$, $\hat{T}>0$, $\zeta_i\in\mathbb{R}$ $(i=0,1)$ can be chosen suitably.
Therefore, Lemma~\ref{lem:v-rate} and Lemma~\ref{lem:u-rate} follow immediately from Lemma~\ref{lem:AS-}.
\end{remark}
}



\section{Proof of Theorem~\ref{thm1}: scenario {\bf(A1)}}\label{sec:A1}
\setcounter{equation}{0}

This section is devoted to the proof of Theorem~\ref{thm1}.
To this aim, we shall construct suitable pairs of supersolutions and subsolutions {when {\bf(A1)} holds}.

{To illustrate our arguments, we start with a simple case;
 that is,
initial data $(u_0,v_0)$ satisfies
\bea\label{simple-ic}
\begin{cases}
0\leq u_0, v_0\leq 1,\quad u_0(-\infty)=1=v_0(+\infty),\\
u_0(x)\equiv0\quad  \mbox{for $x\geq x_u$};\quad   v_0(x)\equiv0\quad  \mbox{for $x\leq x_v$}
\end{cases}
\eea
for some $x_u,x_v\in \mathbb{R}$. The simplest example might be $u_0(x)={\bf 1}_{\{x\leq x_u\}}$ and $v_0(x)={\bf 1}_{\{x\geq x_v\}}$.
Note that \eqref{simple-ic} does not satisfy either {\bf(A1)} or {\bf(A2)}. However,
in this case, a suitable super and subsolutions is easier to construct and may provide some clues in constructing
a suitable super and subsolutions for scenario {\bf(A1)}.

Under this initial condition, together with {\bf(H)},
we will see that the species $u$ always wins the competition.
In the first subsection, we shall
construct a suitable super and subsolutions when $(u_0,v_0)$ satisfies \eqref{simple-ic}
and the convergence result will be proved (Proposition~\ref{prop1}) in \S 3.2.
The proof of Theorem~\ref{thm1} is given in \S 3.3.
}

\subsection{{A simple case: scenario \eqref{simple-ic}}}

{In this subsection, we assume that $(u_0,v_0)$ satisfies \eqref{simple-ic} and prove the following result.
\begin{prop}\label{prop1}
Assume that {\bf(H)} and \eqref{simple-ic} hold.
Then there exists a constant $\hat{h}$ such that the solution $(u,v)$ of \eqref{LV-sys}-\eqref{LV-ic}
satisfies
\bea\label{prop1-result}
\lim_{t\to\infty}\left[\sup_{x\in\mathbb{R}}\Big|u(t,x)-U(x-c_{uv}t-\hat{h})\Big|
+\sup_{x\in\mathbb{R}}\Big|v(t,x)-V(x-c_{uv}t-\hat{h})\Big|\right]=0,
\eea
where $(c_{uv},U,V)$ is a solution of \eqref{TW sys}.
\end{prop}
}

\subsubsection{The construction of a subsolution}

Denote a subsolution $(\underline{u},\overline{v})$ by
\bea\label{lower-sol}
\begin{cases}
\ \underline{u}(t,x):=\max\{U(x-c_{uv}t+\eta(t))-p(t),0\},\\
\ \overline{v}(t,x):=(1+q(t))V(x-c_{uv}t+\eta(t)),
\end{cases}
\eea
where
\bea\label{parameters}
p(t)=p_0 e^{-\alpha t},\quad q(t)=q_0  e^{-\alpha t},\quad \eta(t)=\eta_0-\eta_1 e^{-(\alpha/2)t}
\eea
for some constants $p_0>0$, $q_0>0$, $\alpha>0$ and $\eta_i\in\mathbb{R}$ $(i=0,1)$ that will be determined later.

\begin{lemma}\label{lem:lower-sol}
For any $p_0,q_0,\alpha,\eta_1>0$  satisfying
\bea\label{cond-lem1}
\alpha<\min\{r,1,(a-1)r\}, \ {p_0<\frac{q_0}{b}\Big(\frac{1-\alpha}{2}\Big)},
\eea
there exists $T^{*}\geq0$ such that
\bea\label{lower-ineq0}
N_1[\underline{u},\overline{v}]\leq0,\quad N_2[\underline{u},\overline{v}]\geq0\quad \mbox{in $[T^*,\infty)\times (-\infty,\infty)$}
\eea
for all $\eta_0\in\mathbb{R}$, where $\underline{u}$ and $\overline{v}$ are defined in \eqref{lower-sol}.
\end{lemma}
\begin{proof}
{Fix} any small $\epsilon>0$ satisfying
\bea
&&\epsilon<{\frac{(r-\alpha)p_0}{r(2p_0+aq_0)}},\label{ep-cond1}\\
&&\epsilon<\frac{(a-1)r-\alpha}{2ra(1+q_0 p_0^{-1})},\label{ep-cond2}\\
&&\epsilon<\frac{1-\alpha}{4}.\label{ep-cond3}
\eea
Since $(U,V)(-\infty)=(1,0)$ and $(U,V)(\infty)=(0,1)$,
there exists a sufficiently large constant $M$ such that
\bea
&&1>U(\xi)>1-\epsilon,\quad V(\xi)<\epsilon\quad\mbox{for all $\xi\leq-M$},\label{left side}\\
&&U(\xi)<\epsilon,\quad 1>V(\xi)>1-\epsilon\quad\mbox{for all $\xi\geq M$}.\label{right side}
\eea

For simplicity, we set $\xi=x-c_{uv} t+\eta(t)$ and write $U=U(\xi)$ (resp., $V=V(\xi)$).
Also, we assume $\underline{u}>0$ first, i.e.,
$\underline{u}(t,x)=U(\xi)-p(t)>0$.

Then, by direct computations, we get from the first equation of \eqref{TW sys} that
\bea\label{N1-ineq-lower}
&&N_1[\underline{u},\overline{v}](t,x)\\
&=&\eta'U'-c_{uv} U'-p'-dU''-r(U-p)[1-U+p-a(1+q)V]\notag\\
&=&\eta'U'+rU(1-U-aV)-p'-r(U-p)[1-U+p-a(1+q)V]\notag\\
&=&\eta'U'-p'-rU(p-aqV)+rp[1-U+p-a(1+q)V].\notag
\eea
Also, by the second equation of \eqref{TW sys}, we have
\bea\label{N2-ineq-lower}
&&N_2[\underline{u},\overline{v}](t,x)\\
&=&q'V+(1+q)(-c_{uv}+\eta')V'-(1+q)V''-(1+q)V[1-(1+q)V-b(U-p)]\notag\\
&=&q'V+(1+q)[V(1-V-bU)+\eta'V']-(1+q)V[(1-(1+q)V-bU+bp]\notag\\
&=&q'V+(1+q)\eta'V'-(1+q)V(bp-qV).\notag
\eea

Notice that if $\underline{u}=0$, then clearly $N_1[\underline{u},\overline{v}]=0$; while from \eqref{N2-ineq-lower}
we see that $\underline{u}=0$ does not affect the equality in \eqref{N2-ineq-lower}.
Hence we can 
{restrict the analysis to the case}
 $\underline{u}(t,x)=U(\xi)-p(t)$.

We now divide our discussion into three cases:
\beaa
\mbox{(i) $\xi< -M$; \quad   (ii) $|\xi|\leq M$;  \quad (iii) $\xi>M$.}
\eeaa

{\bf Case (i)}. By the fact that $\eta'>0$ (since $\alpha,\,\eta_1>0$) and $U'<0$, we have $\eta'U'<0$.
Combined with \eqref{left side} and \eqref{N1-ineq-lower} we deduce
\beaa
N_1[\underline{u},\overline{v}](t,x)
&\leq&-p'-rU(p-aqV)+rp[1-U+p]\\
&\leq&-p'-r(1-\epsilon)p+ra\epsilon q +rp(\epsilon+p)\\
&=&-p'-rp+rp^2+2r\epsilon p+r\epsilon a q\\
&=& [(\alpha -r+r p_0e^{-\alpha t})p_0+r\epsilon(2p_0+aq_0)] e^{-\alpha t}.
\eeaa
Thanks to  \eqref{ep-cond1},
we see that {there exists $T_0\gg1$ such that}
$N_1[\underline{u},\overline{v}]\leq0$ for all {$t\geq T_0$}.

On the other hand, since $V'(\cdot)/V(\cdot)\geq \kappa_0$ in $(-\infty,-M]$ for some $\kappa_0>0$ (due to Lemma~\ref{lem:AS-}), from \eqref{N2-ineq-lower} we have
\beaa
N_2[\underline{u},\overline{v}](t,x)
&\geq&\left[\frac{q'}{1+q}+\kappa_0\eta'-bp\right](1+q)V\\
&\geq& e^{-(\alpha/2)t}\Big[-\alpha q_0 e^{-(\alpha /2)t}+ \kappa_0\eta_1\frac{\alpha}{2}-b p_0 e^{-(\alpha /2)t}\Big](1+q)V.
\eeaa
Thus, one can find $T_1\gg1$ such that $N_2[\underline{u},\overline{v}](t,x)\geq0$  for all $(x,t)$ satisfying  $\xi<-M$ and $t\geq T_1$.

{\bf Case (ii)}. Since $U'<0$ in $\mathbb{R}$, we have $\max_{\xi\in[-M,M]}U'(\xi)=-\kappa_1<0$. Also, by virtue of $V\leq 1$, it is easily seen that
\beaa
N_1[\underline{u},\overline{v}](t,x)&\leq&-\eta'\kappa_1-p'-rU(p-aq)+rp(1+p)\\
&=&-\frac{\alpha}{2}\eta_1\kappa_1 e^{-(\alpha/2)t}+O(1)e^{-\alpha t}.
\eeaa
Therefore, there exists $T_2\gg1$ such that $N_1[\underline{u},\overline{v}](t,x)\leq0$  for all $(x,t)$ satisfying $|\xi|\leq-M$ and $t\geq T_2$.

Since $V'>0$ in $\mathbb{R}$, we have $\min_{\xi\in[-M,M]}V'(\xi)=\kappa_2>0$. Then, it holds
\beaa
N_2[\underline{u},\overline{v}](t,x)&\geq&q'V+\kappa_2\eta'-(1+q)Vbp=\frac{\alpha}{2}\eta_1\kappa_2 e^{-(\alpha/2)t}-O(1)e^{-\alpha t}.
\eeaa
Hence, there exists $T_3\gg1$ such that $N_2[\underline{u},\overline{v}](t,x)\geq0$  for all $(x,t)$ satisfying  $|\xi|\leq-M$ and $t\geq T_3$.

{\bf Case (iii)}. Using $\eta'U'<0$, we have
\beaa
N_1[\underline{u},\overline{v}](t,x)&\leq&-p'+raqUV+rp+rp^2-rpaV\\
&\leq&-p'+raq\epsilon+rp+rp^2-rpa(1-\epsilon)\quad \mbox{(thanks to \eqref{right side})}\\
&=&-p'-(a-1)rp+rp^2+ra(p+q)\epsilon\\
&=&\Big[\alpha  -(a-1)r + r p_0 e^{-\alpha t} +ra\Big(1+\frac{q_0}{p_0}\Big)\epsilon\Big] p_0 e^{-\alpha t}\\
&\leq&\Big[\frac{\alpha  -(a-1)r}{2} + r p_0 e^{-\alpha t}\Big] p_0 e^{-\alpha t}\quad \mbox{(using \eqref{ep-cond2})}.
\eeaa
By \eqref{cond-lem1}, there exists $T_4\gg1$ such that $N_1[\underline{u},\overline{v}](t,x)\leq0$  for all $(x,t)$ satisfying  $\xi>M$ and $t\geq T_4$.

On the other hand, by means of $\eta'V'>0$ and \eqref{right side}, we obtain
\beaa
N_2[\underline{u},\overline{v}](t,x)&\geq&q' V +(1+q)V(qV-bp)\\
&\geq&V[-q_0 \alpha +q_0(1-\epsilon)^2-(1+q_0 e^{-\alpha t})b p_0]e^{-\alpha t}\\
&=&V[q_0((1-\epsilon)^2-\alpha)-b p_0-b p_0q_0 e^{-\alpha t}]e^{-\alpha t}.
\eeaa
In view of \eqref{ep-cond3}, we deduce
\beaa
q_0((1-\epsilon)^2-\alpha)-b p_0\geq q_0(1-2\epsilon-\alpha)-b p_0\geq
 q_0 \Big(\frac{1-\alpha}{2}\Big)-b p_0 >0,
\eeaa
where the last inequality follows from \eqref{cond-lem1}.
Hence, there exists $T_5\gg1$ such that $N_2[\underline{u},\overline{v}](t,x)\geq0$ for all $(x,t)$ such that $\xi>M$ and $t\geq T_5$.

Combining the discussions in cases (i)-(iii) and taking  $T^{*}:=\max\{{T_0},T_1,T_2,T_3,T_4,T_5\}\geq0$, we have proved \eqref{lower-ineq0} for all $x\in\mathbb{R}$ and $t\geq T^*$.
This completes the proof.
\end{proof}

\medskip

Next, we shall show that the parameters in $(\underline{u},\overline{v})$ can be chosen suitably such that
it can {be compared with} the solution $(u,v)$ of \eqref{LV-sys} and {\eqref{simple-ic}}. 

\begin{lemma}\label{lem:order-1}
Let $(\underline{u},\overline{v})$ be defined in \eqref{lower-sol} and satisfy \eqref{cond-lem1}.
Then there exist small $\alpha^{*}>0$ and large $T^*>0$ and $\eta_0^{*}>0$ such that {the solution $(u,v)$
of \eqref{LV-sys} and \eqref{simple-ic} satisfies
\beaa
u(t,x)\geq\underline{u}(t+T^*,x),\quad v(t,x)\leq \overline{v}(t+T^*,x)\quad \mbox{for $t\geq0$ and $x\in\mathbb{R}$},
\eeaa
}
provided that $\alpha\in(0, \alpha^{*}]$ and $\eta_0\geq\eta_0^{*}$.
\end{lemma}
{
\begin{proof}
First, by Lemma~\ref{lem:lower-sol}, there exist $T^*\gg1$ and $\alpha^*>0$ such that
\beaa
N_1[\underline{u},\overline{v}](t+T^*,x)\leq0,\quad N_2[\underline{u},\overline{v}](t+T^*,x)\geq0
\quad \mbox{for $t\geq0$ and $x\in\mathbb{R}$},
\eeaa
as long as $\alpha\in(0, \alpha^{*}]$.

For $t=0$, we have $\underline{u}(T^*,-\infty)=1-p_0 e^{-\alpha T^*}<1$ and $\overline{v}(T^*,+\infty)=1+q_0 e^{-\alpha T^*}>1$.
Together with \eqref{simple-ic},
it is obvious that there exists $\eta_0^{*}\gg1$ such that
$u(0,x)\geq\underline{u}(T^*,x)$ and $v(0,x)\leq \overline{v}(T^*,x)$ for all $x\in\mathbb{R}$
as long as $\eta_0\geq \eta_0^{*}$. Then the desired result follows from the comparison principle.
\end{proof}
}

\medskip

\subsubsection{The construction of a supersolution}

To seek a pair of supersolution, we define
\bea\label{upper-sol}
\begin{cases}
\ \overline{u}(t,x)=(1+q(t))U(x-c_{uv}t+\eta(t)),\\
\ \underline{v}(t,x)=\max\{V(x-c_{uv}t+\eta(t))-p(t),0\},
\end{cases}
\eea
where $p$, $q$ and $\eta$ have the same form as in \eqref{parameters}.

The following lemma is  parallel to Lemma~\ref{lem:lower-sol}; we only give some sketch of the proof.
\begin{lemma}\label{lem:super-sol}
For any $p_0,q_0>0,\alpha\in(0,1)$ and $\eta_1<0$ satisfying
\bea\label{cond-lem2}
\alpha<\min\{r,1,b-1\},\ p_0< {\frac{q_0(1-\alpha)}{2a}},
\eea
there exists $T^{**}\geq0$ such that
\bea\label{upper-ineq}
N_1[\overline{u},\underline{v}]\geq0,\quad N_2[\overline{u},\underline{v}]\leq0\quad \mbox{in $[T^{**},\infty)\times (-\infty,\infty)$}
\eea
for all $\eta_0\in\mathbb{R}$, where $\overline{u}$ and $\underline{v}$ are defined in \eqref{upper-sol}.
\end{lemma}
\begin{proof} As in the proof of Lemma~\ref{lem:lower-sol}, for any sufficiently small $\epsilon>0$,
there exists a sufficiently large constant $M$ such that \eqref{left side} and \eqref{right side} hold.
Denote $x-c_{uv}t+\eta(t)$ by $\xi$ and write $U=U(\xi)$ (resp., $V=V(\xi)$).
By direct computations, we have
\bea\label{N1-ineq-upper}
&&N_1[\overline{u}, \underline{v},](t,x)\\
&=&q'U+(1+q)(-c_{uv}+\eta')U'-(1+q)U''-r(1+q)U[1-(1+q)U-a V+ap]\notag\\
&=&q'U+(1+q)\eta'U'-r(1+q)U(ap-qU)\notag
\eea
and
\beaa\label{N2-ineq-upper}
&&N_2[\overline{u},\underline{v}](t,x)\\
&=&\eta'V'-c_{uv}V'-dV''-p'-(V-p)[1-V+p-b(1+q)U]\notag\\
&=&\eta'V'-p'(t)-V(p-bqU)+p[1-V+p-b(1+q)U].\notag
\eeaa

Similar to the proof of Lemma~\ref{lem:lower-sol},
we divide our discussion into three cases:

\begin{center} (i) $\xi< -M$;\ \ \ (ii) $|\xi|\leq M$;\ \ \ (iii) $\xi>M$.
\end{center}

{\bf Case (i)}: this part can be done similarly as in Case (iii) of the proof of Lemma~\ref{lem:lower-sol}.
By \eqref{left side} and the fact that $\eta'U'>0$ (since $\eta_1<0$ and $U'<0$), from \eqref{N1-ineq-upper} it follows
\beaa
N_1[\overline{u}, \underline{v}](t,x)
&\geq& q' U+r(1+q)U(qU-ap)\\
&\geq& U[-\alpha q_0+q_0(1-\epsilon)^2-r(1+q_0e^{-\alpha t})ap_0 e^{-\alpha t}]e^{-\alpha t}.
\eeaa
Due to \eqref{cond-lem2} and the fact that $\epsilon$ can be chosen smaller than $(1-\alpha)/4$, we further have
\beaa
q_0((1-\epsilon)^2-\alpha)-a p_0\geq q_0(1-2\epsilon-\alpha)-a p_0\geq
 q_0\Big(\frac{1-\alpha}{2}\Big)-a p_0 >0.
\eeaa
Then there exists $T_1\gg1$ such that $N_1[\overline{u}, \underline{v},](t,x)\geq0$ for all $(x,t)$ fulfilling $\xi<-M$ and $t\geq T_1$.

On the other hand, in view of $\eta'U'>0$ and the behavior of $U$ and $V$ near $-\infty$, one also knows that
\beaa
N_2[\underline{u},\overline{v}](t,x)&\leq&-p'+bq\epsilon+p+p^2-bp(1-\epsilon)\\
&=&[\alpha  -(b-1) + p_0 e^{-\alpha t} +O(1)\epsilon] p_0 e^{-\alpha t}.
\eeaa
Hence, thanks to \eqref{cond-lem2} and the fact that $\epsilon$ can be chosen smaller if necessary, there exists $T_2\gg1$ such that $N_2[\underline{u},\overline{v}](t,x)\geq0$  for all $(x,t)$ satisfying  $\xi<-M$ and $t\geq T_2$.

Case (ii) and Case (iii) can be handled by the similar process as in Case (ii) and Case (i) of the proof of Lemma~\ref{lem:lower-sol}, respectively; we omit the details here.

According to the above analysis, we see that there exists $T^{**}\geq0$ such that \eqref{upper-ineq} holds,
which completes the proof.
\end{proof}

{Using the similar proof  to that of Lemma~\ref{lem:order-1}, we can obtain the following result.}

\begin{lemma}\label{lem:order-2}
Let $(\overline{u},\underline{v})$ be defined in \eqref{upper-sol} and satisfy \eqref{cond-lem2}.
Then there exist small $\alpha^{**}>0$ and large {$T^{**}>0$} and $\eta^{**}_0<0$ such that
{the solution $(u,v)$
of \eqref{LV-sys} and \eqref{simple-ic} satisfies
\beaa
u(t,x)\leq\overline{u}(t+T^{**},x),\quad v(t,x)\geq \underline{v}(t+T^{**},x)\quad \mbox{for $t\geq0$ and $x\in\mathbb{R}$},
\eeaa
}
provided that $\alpha\in(0, \alpha^{**}]$  and $\eta_0\leq \eta_0^{**}$.
\end{lemma}

Let us consider the long time behavior of  the solution of \eqref{LV-sys}. 
Set
$$\xi=x-c_{uv}t.$$
Then one can define the solution of \eqref{LV-sys} and {\eqref{simple-ic}}
as
\bea\label{hat-u-v}
(\hat{u},\hat{v})(t,\xi)=({u},{v})(t,x)=({u},{v})(t,\xi+c_{uv}t),\quad t>0,\ {\xi\in\mathbb{R}}.
\eea
Then $(\hat{u},\hat{v})$ satisfies
\bea\label{LV-sys-moving}
\begin{cases}
\hat{u}_t=d\hat{u}_{\xi\xi}+c_{uv}\hat{u}_{\xi}+r\hat{u}(1-\hat{u}-a\hat{v}),\\
\hat{v}_t=\hat{v}_{\xi\xi}+c_{uv}\hat{v}_{\xi}+\hat{v}(1-\hat{v}-b\hat{u}),\quad t>0,\ {\xi\in\mathbb{R}}.
\end{cases}
\eea

\medskip

Thanks to Lemma~\ref{lem:order-1} and Lemma~\ref{lem:order-2}, we can obtain the following result immediately.

\begin{lemma}\label{lem:squeeze-TW}
Let $(c_{uv},U,V)$ be a solution of \eqref{TW sys}. Then
there exist constants $p_0,q_0,\alpha>0$ and ${\eta}_i^{*},{\eta}_i^{**}\in \mathbb{R}$, $i=0,1$,
and $T>0$ such that
\beaa
\begin{cases}
U(\xi+\eta_0^{**}-{\eta}_1^{**} e^{-(\alpha/2)t})-p_0e^{-\alpha t}\leq \hat{u}(t,\xi)\leq (1+q_0e^{-\alpha t})U(\xi+\eta_0^{*}-{\eta}_1^{*} e^{-(\alpha/2)t}),\\
V(\xi+\eta_0^{*}-{\eta}_1^{*} e^{-(\alpha/2)t})-p_0e^{-\alpha t}\leq \hat{v}(t,\xi)\leq (1+q_0e^{-\alpha t})V(\xi+\eta_0^{**}-{\eta}_1^{**} e^{-(\alpha/2)t})
\end{cases}
\eeaa
for all $t\geq T$ and ${\xi\in\mathbb{R}}$.
\end{lemma}


{By the proof of Lemma~\ref{lem:lower-sol} and Lemma~\ref{lem:super-sol}, we conclude that
if $(\hat{u},\hat{v})$ is close to $(U,V)(\xi-\xi_0)$ for some $\xi_0$ and some time, the solution will remain close after this time.
More precisely, we have}

\begin{lemma}\label{lem:C0-stable}
Let $(c_{uv},U,V)$ be a solution of \eqref{TW sys}. Then
there exists a function $\nu(\epsilon)$ defined for small $\epsilon$ with $\nu(\epsilon)\to0$ as $\epsilon\downarrow0$ satisfying the following property:
if
{
\bea\label{ic-close}
\Big|\frac{\hat{u}(t_0,\xi)}{U(\xi-\xi_0)}-1\Big|+\Big|\frac{\hat{v}(t_0,\xi)}{V(\xi-\xi_0)}-1\Big|<\epsilon\quad \mbox{for all $\xi\in\mathbb{R}$},
\eea
}
for some $t_0,\xi_0\in\mathbb{R}$, then
{
\beaa
\Big|\frac{\hat{u}(t,\xi)}{U(\xi-\xi_0)}-1\Big|+\Big|\frac{\hat{v}(t,\xi)}{V(\xi-\xi_0)}-1\Big|<\nu(\epsilon)\quad \mbox{for all $t\geq t_0$ and $\xi\in\mathbb{R}$}.
\eeaa
}
\end{lemma}
{
\begin{proof}
From \eqref{ic-close} we see that for all $\xi\in\mathbb{R}$,
\beaa
&&(1-\varepsilon)U(\xi-\xi_0)\leq \hat{u}(t_0,\xi)\leq (1+\varepsilon)U(\xi-\xi_0),\\
&&(1-\varepsilon)V(\xi-\xi_0)\leq \hat{v}(t_0,\xi)\leq (1+\varepsilon)V(\xi-\xi_0).
\eeaa
or, equivalently, for all $x\in\mathbb{R}$,
\beaa
&&(1-\varepsilon)U(x-ct_0-\xi_0) \leq u(t_0,x) \leq (1+\varepsilon)U(x-ct_0-\xi_0),\\
&&(1-\varepsilon)V(x-ct_0-\xi_0) \leq v(t_0,x) \leq (1+\varepsilon)V(x-ct_0-\xi_0).
\eeaa
In the proof of Lemma~\ref{lem:lower-sol} and Lemma~\ref{lem:super-sol}, one may choose
suitable $p_0=O(\varepsilon)$, $q_0=O(\varepsilon)$ and $|\eta_0-\xi_0|=O(\varepsilon)$ such that
$(u,v)(t,x)$ can be compared with the super and subsolutions constructed in Lemma~\ref{lem:lower-sol} and Lemma~\ref{lem:super-sol}
from $t=t_0$.
{Note that the super and subsolutions can always be compaired with $(U,V)(\xi)$ by a translation of $O(\varepsilon)$.}
Then the desired result
follows from the comparison principle.
\end{proof}
}

{
\begin{remark}
We would like to mention that the  $C^0$-stability of the bistable wave $(U,V)$ has been established in \cite{G1982}.
The asymptotic stability of $(U,V)$ (relative to the space of bounded uniformly continuous functions) is reported in \cite{KF1996}.
Lemma~\ref{lem:C0-stable} provides another version of $C^0$-stability result for the bistable wave $(U,V)$
based on our construction of super and subsolutions.
\end{remark}
}

\medskip
\subsection{The proof of {Proposition~\ref{prop1}}}

Let  $(\hat{u},\hat{v})$ be defined in \eqref{hat-u-v} and $(c_{uv},U,V)$ be a solution of \eqref{TW sys}.
{By Lemma~\ref{lem:squeeze-TW}, it is obvious to see that Proposition~\ref{prop1} holds for $x\leq0$. It suffices to
consider $x\geq0$; namely, $\xi\geq-c_{uv}t$.}

Let $\{t_n\}$ be an arbitrary sequence such that $t_n>T$ ($T$ is defined in Lemma~\ref{lem:squeeze-TW}) for each $n$ and
$t_n\to\infty$ as $n\to\infty$. Set
\beaa
\hat{u}_n(t,\xi)=\hat{u}(t+t_n,\xi),\quad\hat{ v}_n(t,\xi)=\hat{v}(t+t_n,\xi),\quad n\in\mathbb{N}.
\eeaa
By the standard parabolic regularity theory and passing to a subsequence, we may assume that
\beaa
(\hat{u}_n,\hat{v}_n)\to (u^{\infty},v^{\infty})\quad \mbox{in $C_{loc}^{{(1+\beta)/2},1+\beta}(\mathbb{R}\times\mathbb{R})$,\ \
as $n\to\infty$},
\eeaa
where $\beta\in(0,1)$ and $(u^{\infty},v^{\infty})$ satisfies
\bea\label{limit sys}
\begin{cases}
u^{\infty}_t=du^{\infty}_{\xi\xi}+c_{uv}u^{\infty}_{\xi}+ru^{\infty}(1-u^{\infty}-av^{\infty}),\\
v^{\infty}_t=v^{\infty}_{\xi\xi}+c_{uv}v^{\infty}_{\xi}+v^{\infty}(1-v^{\infty}-bu^{\infty}),\quad t\in\mathbb{R},\ \xi\in\mathbb{R}.
\end{cases}
\eea
In addition, let us replace $t$ by $t+t_n$ in the inequalities of Lemma~\ref{lem:squeeze-TW} and take $n\to\infty$. Then we have
\bea\label{squeeze-limit}
\begin{cases}
U(\xi+\eta_0^{**})\leq u^{\infty}(t,\xi)\leq U(\xi+\eta_0^{*}),\ \ \forall t,\ \xi\in\mathbb{R},\\
V(\xi+\eta_0^{*})\leq v^{\infty}(t,\xi)\leq V(\xi+\eta_0^{**}),\ \ \forall t,\ \xi\in\mathbb{R}.
\end{cases}
\eea

Define
\beaa
&&h_1:=\inf\{h\in\mathbb{R}:\  u^{\infty}(t,\xi)\leq U(\xi-h)\ {\rm and}\ v^{\infty}(t,\xi)\geq V(\xi-h),\  \forall t,\xi\in\mathbb{R} \},\\
&&h_2:=\sup\{h\in\mathbb{R}:\  u^{\infty}(t,\xi)\geq U(\xi-h)\ {\rm and}\ v^{\infty}(t,\xi)\leq V(\xi-h),\  \forall t,\xi\in\mathbb{R}\}.
\eeaa
Notice that $h_1$ and $h_2$ are finite because of \eqref{squeeze-limit}.
Also, by continuity,
\bea
u^{\infty}(t,\xi)\leq U(\xi-h_1)\ {\rm and}\  v^{\infty}(t,\xi)\geq V(\xi-h_1),\ \ \forall t,\ \xi\in\mathbb{R},\label{h1}\\
u^{\infty}(t,\xi)\geq U(\xi-h_2)\ {\rm and}\  v^{\infty}(t,\xi)\leq V(\xi-h_2),\ \ \forall t,\ \xi\in\mathbb{R}.\label{h2}
\eea
Clearly, $h_1\geq h_2$. Below we are going to assert $h_1=h_2$. Since the proof is rather long,
we prove this assertion in the following lemma.

\begin{lemma}\label{lem: h} Let $h_1,\, h_2$ be defined as above. Then $h_1=h_2$.
\end{lemma}

\begin{proof} For contradiction we assume that $h_1>h_2$.
First of all, we claim the following
 \bea\label{notouch}
u^{\infty}(t,\xi)< U(\xi-h_1)\ {\rm and}\  v^{\infty}(t,\xi)> V(\xi-h_1), \quad \forall t,\,\xi\in\mathbb{R}.
\eea
If \eqref{notouch} is false, then there exists $t_0\in\mathbb{R}$ and $\xi_0\in\mathbb{R}$ such that
$u^{\infty}(t_0,\xi_0)=U(\xi_0-h_1)$ or $v^{\infty}(t_0,\xi_0)= V(\xi_0-h_1)$.
Observe that $(U(\xi-h_1),V(\xi-h_1))$ also satisfies \eqref{limit sys}. Using \eqref{h1} and the strong maximum principle,
we obtain
\bea\label{equal h1}
u^{\infty}(t,\xi)=U(\xi-{h_1}),\quad v^{\infty}(t,\xi)=V(\xi-{h_1})
\eea
for all $t\leq t_0$ and $\xi\in\mathbb{R}$. By the uniqueness of solutions to the corresponding Cauchy problem of \eqref{limit sys}, we then conclude that \eqref{equal h1} is valid for all $t\in \mathbb{R}$ and $\xi\in\mathbb{R}$,
contradicting the definition of $h_2$ due to $h_2<h_1$. 
Therefore, \eqref{notouch} holds.

Define
\beaa
\omega_1(\xi):=\inf_{t\in\mathbb{R}}[U(\xi-h_1)-u^{\infty}(t,\xi)],\quad  \omega_2(\xi):=\inf_{t\in\mathbb{R}}[v^{\infty}(t,\xi)-V(\xi-h_1)],\ \ \xi\in\mathbb{R}.
\eeaa
By \eqref{notouch}, we see that $\omega_i(\xi)\geq0$ for all $\xi\in\mathbb{R}$ and $i=1,2$.

In what follows, we divide our discussion into two cases:

{\bf Case 1}: there exists $z_0\in\mathbb{R}$ such that $\omega_1(z_0)=0$ or $\omega_2(z_0)=0$.

{\bf Case 2}: it holds that $\omega_i(\xi)>0$ for all $\xi\in\mathbb{R}$ and $i=1,2$.

We first consider {\bf Case 1}. Without loss of generality, we may assume that $\omega_1(z_0)=0$.
Then, there exists $\{\tau_n\}$ such that $|\tau_n|\to\infty$ and $\lim_{n\to\infty}u^{\infty}(\tau_n,z_0)=U(z_0-h_1)$.
Denote
\beaa
(\widehat{U}_n,\widehat{V}_n)(t,\xi):=(u^{\infty},v^{\infty})(t+\tau_n,\xi).
\eeaa
By standard parabolic regularity theory and passing to a subsequence we may assume that, for some $\beta\in(0,1)$,
\beaa
(\widehat{U}_n,\widehat{V}_n)\to (\widehat{U}^{\infty},\widehat{V}^{\infty})\quad \mbox{in $C_{loc}^{{(1+\beta)/2},1+\beta}(\mathbb{R}\times\mathbb{R})$,\ \, as $n\to\infty$},
\eeaa
where $(\widehat{U}^{\infty},\widehat{V}^{\infty})$ satisfies $\widehat{U}^{\infty}(0,z_0)=U(z_0-h_1)$ and
\bea\label{limit sys2}
\begin{cases}
\widehat{U}^{\infty}_t=
d\widehat{U}^{\infty}_{\xi\xi}+c_{uv}\widehat{U}^{\infty}_{\xi}+r\widehat{U}^{\infty}(1-\widehat{U}^{\infty}-a\widehat{V}^{\infty}),\ \ t,\, \xi\in\mathbb{R},\\
\widehat{V}^{\infty}_t
=\widehat{V}^{\infty}_{\xi\xi}+c_{uv}\widehat{V}^{\infty}_{\xi}+\widehat{V}^{\infty}(1-\widehat{V}^{\infty}-b\widehat{U}^{\infty}),\ \ \ t,\,\xi\in\mathbb{R}.
\end{cases}
\eea
Furthermore, from \eqref{h2} we see that
\beaa
\widehat{U}^{\infty}(t,\xi)\leq U(\xi-h_1)\ {\rm and}\  \widehat{V}^{\infty}(t,\xi)\geq V(\xi-h_1).
\eeaa

Notice that $(U(\xi-h_1), V(\xi-h_1))$ satisfies \eqref{limit sys2} and $\widehat{U}^{\infty}(0,z_0)=U(z_0-h_1)$. Thus,
the strong maximum principle and the uniqueness of solutions of the corresponding Cauchy problem
yield that
\bea\label{identity}
\widehat{U}^{\infty}(t,\xi)\equiv U(\xi-h_1)\ {\rm and}\  \widehat{V}^{\infty}(t,\xi)\equiv V(\xi-h_1) \quad \mbox{for all $t,\ \xi\in\mathbb{R}$},
\eea
which implies that
\bea\label{conv-lcoal-unif}
(\widehat{U}_n,\widehat{V}_n)(0,\xi)\to (U,V)(\xi-h_1)\quad \mbox{as $n\to\infty$ locally uniformly for $\xi\in\mathbb{R}$}.
\eea
In fact, the convergence of \eqref{conv-lcoal-unif} is uniform for $\xi\in\mathbb{R}$. Indeed,
from \eqref{h1} and \eqref{h2} and the fact that $(U,V)(-\infty)=(1,0)$ and $(U,V)(\infty)=(0,1)$, we see that for each $\epsilon>0$, there exists $M>0$ such that {for each $n\in N$},
\beaa
\|(\widehat{U}_n,\widehat{V}_n)(0,\cdot)- (U,V)(\cdot-h_1)\|_{L^{\infty}(\mathbb{R}\backslash[-M,M])}<\epsilon.
\eeaa
Together with \eqref{conv-lcoal-unif}, it follows that
$(\widehat{U}_n,\widehat{V}_n)(0,\xi)\to (U,V)(\xi-h_1)$ as $n\to\infty$ uniformly for $\xi\in\mathbb{R}$,
or equivalently,
\bea\label{conv-unif}
(u^{\infty},v^{\infty})(\tau_n,\xi)\to (U,V)(\xi-h_1),\quad \mbox{as $n\to\infty$ uniformly for $\xi\in\mathbb{R}$}.
\eea

Recall that the time sequence $\{\tau_n\}$ satisfies $|\tau_n|\to\infty$. Without loss of generality we may assume
that $\tau_n\to-\infty$ or $\tau_n\to+\infty$ (if necessary we can take a subsequence).
Suppose that $\tau_n\to-\infty$.
Then, from \eqref{conv-unif} and {the fact of the local asymptotical stability of $(U,V)(\cdot-h_1)$ (cf. \cite[Theorem 3.6]{KF1996}),}
we see that
\beaa
(u^{\infty},v^{\infty})(t,\xi)\equiv(U,V)(\xi-h_1)\quad \mbox{for all $t\in\mathbb{R}$ and $\xi\in\mathbb{R}$}.
\eeaa
which contradicts with \eqref{notouch}. Therefore, we must have $\tau_n\to+\infty$.
Then, in view of \eqref{conv-unif} and {\cite[Theorem 3.6]{KF1996}},
we have
\bea\label{close-h1}
\lim_{t\to\infty}\|(u^{\infty},v^{\infty})(t,\cdot)- (U,V)(\cdot-h_1)\|_{L^{\infty}(\mathbb{R})}=0.
\eea

We now define
\beaa
\sigma_1(\xi):=\inf_{t\in\mathbb{R}}[u^{\infty}(t,\xi)-U(\xi-h_2)],\quad  \sigma_2(\xi):=\inf_{t\in\mathbb{R}}[V(\xi-h_2)-v^{\infty}(t,\xi)],\ \ \xi\in\mathbb{R}.
\eeaa
By \eqref{h2}, we see that $\sigma_i(\xi)\geq0$ for all $\xi\in\mathbb{R}$ and $i=1,2$.
Then, we have

{\bf Claim 1:} It holds
\bea\nonumber
\sigma_i(\xi)>0\quad \mbox{for all $\xi\in\mathbb{R}$ and $i=1,2$}.
\eea
If {\bf Claim 1} is not true,
there exists $\zeta_0\in\mathbb{R}$ such that $\sigma_1(\zeta_0)=0$ or $\sigma_2(\zeta_0)=0$.
Without loss of generality, we may assume that $\sigma_1(\zeta_0)=0$.
By \eqref{close-h1} we see that there exists $\{\tilde{\tau}_n\}$
such that $\tilde{\tau}_n\to-\infty$ and $\lim_{n\to\infty}u^{\infty}(\tilde{\tau}_n,\zeta_0)=U(\zeta_0-h_2)$.

Denote
\beaa
(\widetilde{U}_n,\widetilde{V}_n)(t,\xi):=(u^{\infty},v^{\infty})(t+\tilde{\tau}_n,\xi).
\eeaa
By standard parabolic regularity theory and passing to a subsequence we may assume that, for some $\beta\in(0,1)$,
\beaa
(\widetilde{U}_n,\widetilde{V}_n)\to (\widetilde{U}^{\infty},\widetilde{V}^{\infty})\quad \mbox{in $C_{loc}^{{(1+\beta)/2},1+\beta}(\mathbb{R}\times\mathbb{R})$,\ \ as $n\to\infty$},
\eeaa
where $(\widetilde{U}^{\infty},\widetilde{V}^{\infty})$ satisfies $\widetilde{U}^{\infty}(0,\zeta_0)=U(\zeta_0-h_2)$ and
\bea\label{limit sys3}
\begin{cases}
\widetilde{U}^{\infty}_t=d\widetilde{U}^{\infty}_{\xi\xi}+c_{uv}\widetilde{U}^{\infty}_{\xi}+r\widetilde{U}^{\infty}(1-\widetilde{U}^{\infty}-a\widetilde{V}^{\infty}),
\quad \forall t,\, \xi\in\mathbb{R},\\
\widetilde{V}^{\infty}_t=\widetilde{V}^{\infty}_{\xi\xi}+c_{uv}\widetilde{V}^{\infty}_{\xi}+\widetilde{V}^{\infty}(1-\widetilde{V}^{\infty}-b\widetilde{U}^{\infty}),\quad \forall t,\, \xi\in\mathbb{R}.
\end{cases}
\eea
Then, similar to \eqref{identity}, we have
\beaa
\widetilde{U}^{\infty}(t,\xi)\equiv U(\xi-h_2)\ {\rm and}\  \widetilde{V}^{\infty}(t,\xi)\equiv V(\xi-h_2), \quad \forall t,\, \xi\in\mathbb{R}.
\eeaa
The same process as in deriving \eqref{conv-unif} gives
\bea\label{conv-unif2}
(u^{\infty},v^{\infty})(\tilde{\tau}_n,\xi)\to (U,V)(\xi-h_2),\quad \mbox{as $n\to\infty$ uniformly for $\xi\in\mathbb{R}$}.
\eea
Since $\tilde{\tau}_n\to-\infty$, it follows from \eqref{conv-unif2} and
{the fact of the local asymptotical stability of $(U,V)(\cdot-h_1)$ (cf. \cite[Theorem 3.6]{KF1996})},
 that
\beaa
(u^{\infty},v^{\infty})(t,\xi)\equiv(U,V)(\xi-h_2)\quad \mbox{for all $t,\, \xi\in\mathbb{R}$}.
\eeaa
which contradicts 
{the definition of $h_1$}
and we thus obtain {\bf Claim 1}.

Due to {\bf Claim 1}, one can use the sliding method to further assert that

{\bf Claim 2:} There exists $\epsilon>0$ sufficiently small such that
\bea\nonumber
u^{\infty}(t,\xi)\geq U(\xi-(h_2+\epsilon)),\ \ \ v^{\infty}(t,\xi)\leq V(\xi-(h_2+\epsilon)),\ \ \forall t,\,\xi\in\mathbb{R}.
\eea
Once {\bf Claim 2} is proved, we will obtain a contradiction with the definition of $h_2$.

We now verify {\bf Claim 2}.
{
Choose  $\epsilon_0>0$ sufficiently small and $\xi_0\gg1$
such that
\bea\label{e0}
\frac{a-1-\epsilon_0}{a\epsilon_0}(1-\frac{2\epsilon_0}{a})>b,
\eea
and $U(\xi-h_2)<\epsilon_0/2$ and $V(\xi-h_2)>1-\epsilon_0/(2a)$ for all $\xi\geq \xi_0$.
Therefore, one can take $\epsilon\in(0,\epsilon_0)$ such that
\bea\label{bistability}
U(\xi-(h_2+\epsilon))<\epsilon_0,\quad V(\xi-(h_2+\epsilon))>1-\frac{\epsilon_0}{a}\quad \mbox{for all $\xi\geq \xi_0$}.
\eea
}

Inspired by \cite{DMZ2015}, we consider the following auxiliary system:
\bea\label{aux-sys}
\begin{cases}
P_t=dP_{\xi\xi}+c_{uv}P_{\xi}+f(P,Q),\quad t>0,\ \xi\geq \xi_0,\\
Q_t=Q_{\xi\xi}+c_{uv}Q_{\xi}+g(P,Q),\quad t>0,\ \xi\geq \xi_0,\\
P(t,\xi_0)=U(\xi_0-(h_2+\epsilon)),\quad Q(t,\xi_0)=V(\xi_0-(h_2+\epsilon)),\quad t>0,\\
P(0,\xi)=0,\quad Q(0,\xi)=1,\quad \xi\geq \xi_0,
\end{cases}
\eea
where
$$f(P,Q):=rP(1-P-aQ),\ \ \ g(P,Q):=Q(1-Q-bP).$$

Note that the initial function ${(P,Q)(0,\cdot)}=(0,1)$ forms a pair of subsolution of the corresponding stationary problem of \eqref{aux-sys}. Hence, {from the theory of monotone systems (cf. \cite[Ch1]{CC}), we see that}
$P(t,\cdot)$ is increasing in $t$ and $Q(t,\cdot)$ is decreasing in $t$. Also,
because $(U(\xi-(h_2+\epsilon)),V(\xi-(h_2+\epsilon)))$ satisfies the first two equations and the boundary condition of \eqref{aux-sys},
one can apply the comparison principle to deduce that
\bea\label{CP-result}
0\leq P(t,\xi) \leq U(\xi-(h_2+\epsilon)),\quad V(\xi-(h_2+\epsilon))\leq Q(t,\xi)\leq 1
\eea
for all $t>0$ and $\xi\geq \xi_0$. {Note that
although the compatibility condition does not hold for \eqref{aux-sys},
we can use a well known approximation argument to obtain \eqref{CP-result}. More precisely,
one may consider a suitable sequence of smooth and uniformly bounded approximating initial data satisfying the boundary conditions
which converges to $(P,Q)(0,\cdot)$ in $L^2$ norm. By applying a standard comparison principle,
\eqref{CP-result} holds by replacing $(P,Q)$ with the corresponding solution
with such smooth initial data. Then, \eqref{CP-result} follows by a standard compactness argument.
}

Define the limit functions
\beaa
P^*(\xi):=\lim_{t\to\infty}P(t,\xi),\quad Q^*(\xi):=\lim_{t\to\infty}Q(t,\xi),\quad \xi>\xi_0.
\eeaa
Then, one has
\bea\label{order-PQUV}
P^*(\xi) \leq U(\xi-(h_2+\epsilon)),\quad V(\xi-(h_2+\epsilon))\leq Q^*(\xi),\quad \xi>\xi_0.
\eea
Furthermore, $(P^*,Q^*)$ satisfies
\beaa\label{aux-lim-sys}
\begin{cases}
0=dP^{*}_{\xi\xi}+c_{uv}P^*_{\xi}+f(P^*,Q^*),\quad \xi\geq \xi_0,\\
0=Q^{*}_{\xi\xi}+c_{uv}Q^*_{\xi}+g(P^*,Q^*),\quad \xi\geq \xi_0,\\
P^*(\xi_0)=U(\xi_0-(h_2+\epsilon)),\quad Q^*(\xi_0)=V(\xi_0-(h_2+\epsilon)),\\
P^*(\infty)=0,\quad Q^*(\infty)=1.
\end{cases}
\eeaa

In the sequel, we are going to conclude

{\bf Claim 3:} It holds
\bea\nonumber
P^*(\xi)= U(\xi-(h_2+\epsilon)),\quad Q^*(\xi)= V(\xi-(h_2+\epsilon)),\quad  \xi\geq\xi_0.
\eea
To verify {\bf Claim 3}, we introduce
\beaa
Z_1(\xi):=U(\xi-(h_2+\epsilon))-P^*(\xi),\quad Z_2(\xi):=Q^*(\xi)-V(\xi-(h_2+\epsilon)).
\eeaa
From \eqref{order-PQUV} it follows that
\bea\label{Zi}
Z_i(\xi_0)=0,\quad Z_i(\xi)\geq0\quad \mbox{for all $\xi\geq \xi_0$ and $i=1,2$}.
\eea
For convenience, we write $U_\epsilon(\xi)=U(\xi-(h_2+\epsilon))$ and $V_\epsilon(\xi)=V(\xi-(h_2+\epsilon))$.
By direct computations, we have
\beaa
dZ_1^{''}+c_{uv}Z_1'&=&-rU_\epsilon(1-U_\epsilon-aV_\epsilon)+rP^*(1-P^*-aQ^*)\\
&=&r[(aV_\epsilon+P^*+U_\epsilon-1)Z_1-aP^*Z_2],\quad \xi\geq \xi_0,\\
Z_2^{''}+c_{uv}Z_2'&=&-Q^*(1-Q^*-bP^*)+V_\epsilon(1-V_\epsilon-bU_\epsilon)\\
&=&(bP^*+Q^*+V_\epsilon-1)Z_2-bV_\epsilon Z_1\\
&\geq &(2V_\epsilon-1)Z_2-bV_\epsilon Z_1,\quad \xi\geq \xi_0\quad \mbox{(due to \eqref{order-PQUV})}.
\eeaa
Since $Z_i(\xi_0)=0\leq Z_i(\xi)$ for $\xi\geq \xi_0$ and $Z_i(+\infty)=0$, one can define
\beaa
Z_i(\zeta_i)=\max_{\xi\in[\xi_0,\infty)}Z_i(\xi)\geq0,\quad i=1,2.
\eeaa
Then, {\bf Claim 3} is equivalent to
\bea\label{Zi-const}
Z_i(\zeta_i)=0 \quad \mbox{for $i=1,2$}.
\eea
Suppose that $Z_1(\zeta_1)>0$.

We then have to distinguish two cases:
\beaa
{\rm(i)}\ (a-1-{\epsilon_0})Z_1(\zeta_1)>a{\epsilon_0} Z_2(\zeta_2);\quad {\rm(ii)}\ (a-1-{\epsilon_0})Z_1(\zeta_1)\leq a{\epsilon_0} Z_2(\zeta_2).
\eeaa
When case (i) happens, one can use the equation of $Z_1$, {\eqref{bistability}} and the fact that $P^*\leq U_{\epsilon}$ to deduce
\beaa
0&\geq& dZ_1^{''}(\zeta_1)+c_{uv}Z_1'(\zeta_1)>r\big[(a-{\epsilon_0}-1)Z_1(\zeta_1)-a{\epsilon_0} Z_2(\zeta_1)\big]\\
&\geq&r\big[(a-{\epsilon_0}-1)Z_1(\zeta_1)-a{\epsilon_0} Z_2(\zeta_2)\big]>0,
\eeaa
which reaches a  contradiction and (i) thus cannot occur.

On the other hand, if case (ii) happens, one can use the equation of $Z_2$, {\eqref{bistability} and \eqref{e0}} to deduce
\beaa
0&\geq& Z_2^{''}(\zeta_2)+c_{uv}Z_2'(\zeta_2)>\Big[2\Big(1-\frac{{\epsilon_0}}{a}\Big)-1\Big]Z_2(\zeta_2)-b Z_1(\zeta_2)\\
&\geq&\frac{a-1-{\epsilon_0}}{a{\epsilon_0}}\Big[1-\frac{2{\epsilon_0}}{a}\Big]Z_1(\zeta_1)-b Z_1(\zeta_1)>0.
\eeaa
Again, we arrive at a contradiction.
Therefore, $Z_1(\zeta_1)=0$, or equivalently, $Z_1(\xi)=0$ for all $\xi\geq\xi_0$. Together with \eqref{Zi} and the equation of $Z_2$,
we have
\beaa
Z_2^{''}+c_{uv}Z_2'-(bP^*+Q^*+V_\epsilon-1)Z_2=0\quad \mbox{for $\xi\geq\xi_0$}; \quad Z_2(\xi_0)=0\leq Z_2(\xi)\quad \mbox{for $\xi\geq\xi_0$}.
\eeaa
As $Z_2(\infty)=0$,  the strong maximum principle implies that $Z_2(\xi)=0$ for all $\xi\geq\xi_0$. Thus, we have proved \eqref{Zi-const} and then {\bf Claim 3} holds.

We now complete the proof of {\bf Claim 2}. Because of {\bf Claim 1}, one can fix $\epsilon>0$ sufficiently small such that
\beaa
u^{\infty}(t,\xi_0)\geq U(\xi_0-(h_2+\epsilon)),\quad v^{\infty}(t,\xi_0)\leq V(\xi_0-(h_2+\epsilon))\quad \mbox{for all $t\in\mathbb{R}$}.
\eeaa
Also, notice that $u^{\infty}(t,\xi_0)\geq0=P(0,\xi)$ and $v^{\infty}(t,\xi_0)\leq1=Q(0,\xi)$ for all $\xi\geq\xi_0$. Using the comparison principle, we obtain
\beaa
u^{\infty}(s+t,\xi)\geq P(t,\xi),\quad v^{\infty}(s+t,\xi)\leq Q(t,\xi)\quad \mbox{for all $t>0$, $s\in\mathbb{R}$ and $\xi\geq\xi_0$},
\eeaa
which is equivalent to
\beaa
u^{\infty}(t,\xi)\geq P(t-s,\xi),\quad v^{\infty}(t,\xi)\leq Q(t-s,\xi)\quad \mbox{for all $t>s$, $s\in\mathbb{R}$ and $\xi\geq\xi_0$}.
\eeaa
 By taking $s\to-\infty$ and using {\bf Claim 3}, we have
\beaa
u^{\infty}(t,\xi)\geq P^*(\xi)=U(\xi-(h_2+\epsilon)),\quad v^{\infty}(t,\xi)\leq Q^*(\xi)=V(\xi-(h_2+\epsilon))
\eeaa
for all $t\in\mathbb{R}$ and $\xi\geq \xi_0$.

By a similar process used as above, we can conclude that there exists $\xi_1\gg1$ such that
\bea\label{sliding-L}
u^{\infty}(t,\xi)\geq U(\xi-(h_2+\epsilon)),\quad v^{\infty}(t,\xi)\leq V(\xi-(h_2+\epsilon))
\eea
for all $t\in\mathbb{R}$ and $\xi\leq -\xi_1$ by taking $\epsilon>0$ smaller if necessary.

Notice that by the continuity, \eqref{sliding-L} still holds for all $t\in\mathbb{R}$ and $\xi\in[-\xi_1,\xi_0]$ by choosing $\epsilon>0$ further smaller if necessary. Therefore, we have proved {\bf Claim 2}. However, this contradicts the definition of $h_2$. Hence, we must have $h_1=h_2$ when {\bf Case 1} occurs.

We now treat {\bf Case 2}. In this case, one can apply the sliding method used above to show that
\beaa
u^{\infty}(t,\xi)\leq U(\xi-(h_1-\epsilon)),\quad v^{\infty}(t,\xi)\geq V(\xi-(h_1-\epsilon)),\ \ \forall t,\,\xi\in\mathbb{R}
\eeaa
for some small $\epsilon>0$. This contradicts the definition of $h_1$, which means that $h_1>h_2$ is impossible.
Hence, it is necessary that $h_1=h_2$ when {\bf Case 2} occurs. The proof is thus complete.
\end{proof}

With the aid of Lemma \ref{lem: h}, we are now ready to present

\begin{proof}[Proof of {Proposition~\ref{prop1}}] Lemma \ref{lem: h} tells us that
\bea\nonumber
u^{\infty}(t,\xi)=U(\xi-\hat{h}),\quad v^{\infty}(t,\xi)=V(\xi-\hat{h})\quad \mbox{for all $t,\,\xi\in\mathbb{R}$}
\eea
with $\hat{h}=h_1=h_2$. It then follows that
\beaa
\lim_{n\to\infty}(\hat{u},\hat{v})(t+t_n,\xi)= (U,V)(\xi-\hat{h})\quad \mbox{in $C_{loc}^{{(1+\beta)/2},1+\beta}(\mathbb{R}\times\mathbb{R})$}.
\eeaa
Since the time sequence $\{t_n\}$ can be chosen arbitrarily, we have
\beaa
\lim_{t\to\infty}(\hat{u},\hat{v})(t,\xi)=(U,V)(\xi-\hat{h})\quad \mbox{uniformly for $\xi$ in any compact subset of $\mathbb{R}$}.
\eeaa
By \eqref{hat-u-v}, we thus obtain
\bea\label{conv}
\lim_{t\to\infty}({u},{v})(t,x)=(U,V)(x-c_{uv}t+\hat{h})\quad \mbox{locally uniformly in $x-c_{uv}t$ with $x\geq0$}.
\eea
Moreover, from \eqref{h1} and \eqref{h2} and the fact that $(U,V)(-\infty)=(1,0)$ and $(U,V)(\infty)=(0,1)$, it is clear to see that for each $\epsilon>0$, there exists
$N'>0$ and $M'>0$ such that $t\geq N'$ implies that
\beaa
|(u,v)(t,x)- (U,V)(x-c_{uv}t-\hat{h})|<\epsilon\quad \mbox{{if $0\leq x\leq c_{uv}t-M'$ or $x\geq c_{uv}t+M'$}},
\eeaa
which, combined with \eqref{conv}, yields {\eqref{prop1-result}}. This completes the proof.
\end{proof}

\medskip

{
\subsection{The proof of Theorem~\ref{thm1}}

The construction of super and subsolutions in the previous subsection motivates us to deal with the case that
initial data $(u_0,v_0)$ satisfies {\bf(A1)}. Here we shall construct a new type of super-subsolutions when {\bf(A1)} holds; but
the process becomes more complicated.
Then Theorem~\ref{thm1} follows using an argument similar to that of Proposition~\ref{prop1}.
}

{
\subsubsection{The construction of a supersolution}
Denote a  supersolution $(\overline{u},\underline{v})$ by
\bea\label{supersol-2}
\begin{cases}
\overline{u}(t,x)=U(x-c_{uv}t+\zeta(t))+U(-x-c_{uv}t+\zeta(t))-1+\hat{p}(t),\ t\geq0,\ x\in\mathbb{R},\\
\underline{v}(t,x)=(1-\hat{q}(t))\Big[V(x-c_{uv}t+\zeta(t))+V(-x-c_{uv}t+\zeta(t))\Big],\ t\geq0,\ x\in\mathbb{R},
\end{cases}
\eea
where
\bea\label{parameters2}
\hat{p}(t)=\hat{p}_0e^{-\beta t},\ \,\hat{q}(t)=\hat{q}_0e^{-\beta t}, \ \,\zeta(t)=\zeta_0-\zeta_1 e^{-(\beta/2)t}
\eea
for some $\hat{p}_0,\hat{q}_0,\beta>0$ and $\zeta_i\in\mathbb{R}$ ($i=0,1$) that will be determined later.
The form of $\overline{u}$ here is inspired by \cite{FM1977}.
}

{
\begin{lemma}\label{lem:super-sol2}
Let $\overline{u}$ and $\underline{v}$ be defined in \eqref{supersol-2}.
For any $\zeta_1<0$ and
$\hat{p}_0,\hat{q}_0>0$ with $\hat{q}_0>2b \hat{p}_0$,
there exists $T^{*}\geq0$ such that
\bea\label{upper-ineq*}
N_1[\overline{u},\underline{v}]\geq0,\quad N_2[\overline{u},\underline{v}]\leq0\quad \mbox{in $[T^{*},\infty)\times (-\infty,\infty)$,}
\eea
provided $\beta>0$ is sufficiently small and $\zeta_0$ is sufficiently close to $-\infty$.
\end{lemma}
}
\begin{proof}
For notational convenience, we denote
$$\xi_{\pm}=\pm x-c_{uv}t+\zeta(t),\ \ \ (U_{\pm},V_{\pm})=(U(\xi_{\pm}),V(\xi_{\pm})).$$
Then after some direct computation, we obtain
\beaa
N_1[\overline{u},\underline{v}]&=&(-c_{uv}+\zeta'(t))(U'_{+}+U'_{-})+\hat{p}'(t)\\
&&-d(U''_{+}+U''_{-})-f(U_{+}+U_{-}-1+\hat{p},(1-\hat{q})(V_{+}+V_{-})),
\eeaa
where $f(u,v):=ru(1-u-av)$. Since $-c_{uv}U'_{\pm}-dU''_{\pm}=f(U_{\pm},V_{\pm})$, we thus have
\bea
N_1[\overline{u},\underline{v}]&=&\zeta'(t)(U'_{+}+U'_{-})+\hat{p}'+f(U_+,V_+)+f(U_-,V_-)\label{N1-sec4}\\
&&\quad-f(U_{+}+U_{-}-1+\hat{p},(1-\hat{q})(V_{+}+V_{-}))\notag\\
&=&\zeta'(t)(U'_{+}+U'_{-})+\hat{p}'+f(U_+,V_+)+f(U_-,V_-)\notag\\
&&\quad-f(U_{+}+U_{-}-1+\hat{p},V_+)+f(U_{+}+U_{-}-1+\hat{p},V_+)\notag\\
&&\quad-f(U_{+}+U_{-}-1+\hat{p},(1-\hat{q})(V_{+}+V_{-})).\notag
\eea
Similarly, making use of  $-c_{uv}V'_{\pm}-V''_{\pm}=g(U_{\pm},V_{\pm})$ we obtain
\bea
\ \ \ N_2[\overline{u},\underline{v}]&=&-\hat{q}'(t)(V_{+}+V_{-})+(1-\hat{q})\zeta'(t)(V'_{+}+V'_{-})
                                    +(1-\hat{q})[g(U_+,V_+)+g(U_-,V_-)]\label{N2-sec4}\\
                                   &&\qquad-g(U_{+}+U_{-}-1+\hat{p},(1-\hat{q})(V_{+}+V_{-}))\notag\\
                               &=&-\hat{q}'(t)(V_{+}+V_{-})+(1-\hat{q})\zeta'(t)(V'_{+}+V'_{-})\notag\\
                                   &&\qquad-(1-\hat{q})V_+[-V_-+\hat{q}(V_++V_-)-b(U_--1+\hat{p})]\notag\\
                                   &&\qquad -(1-\hat{q})V_-[-V_++\hat{q}(V_++V_-)-b(U_+-1+\hat{p})],\notag
\eea
where $g(u,v)=v(1-v-bu)$.

We shall show that $N_1[\overline{u},\underline{v}]\geq0$ and $N_2[\overline{u},\underline{v}]\leq0$ for $x\in\mathbb{R}$ and sufficiently large $t$. Here we only consider the range $x\geq0$
 since a similar process can be used for the case $x<0$.
First, we take $\zeta_1<0$ such that $\zeta'<0$. Since $x\geq0$, $U'<0$ and $\zeta'<0$,
we have
\beaa
1-U_-=1-U(-x-c_{uv}t+\zeta(t))\leq 1-U(-c_{uv}t+{\zeta_0}).
\eeaa
We also require $\zeta_0<0$. Then, by Lemma~\ref{lem:AS-}, there exist two constants $\lambda_u>0$ and $K_1>0$ independent of
${\Lambda:=(\hat{p}_0,\hat{q}_0,\beta,\zeta_0)}$,
such that
\bea\label{est:1-U-}
1-U_-\leq K_1e^{-\lambda_{u}(c_{uv}t-{\zeta_0})}\quad \mbox{for all $x\geq0$ and $t\geq0$.}
\eea
Without loss of generality, we may assume that $U_- -1 +\hat{p}>0$
since we may choose $\beta<-\lambda_{u}c_{uv}$ and $-\zeta_0$ sufficiently large.

Similarly, thanks to Lemma~\ref{lem:AS-} we may find two constants $\lambda_v>0$ and $K_2>0$ (independent of
$\Lambda$) such that
\bea\label{est:V-}
V_-\leq K_2e^{-\lambda_{v}(c_{uv}t-{\zeta_0})}\quad \mbox{for all $x\geq0$ and $t\geq0$.}
\eea

To derive the differential inequalities, we divide the discussion into three cases.

{\bf Case 1:} $0\leq U_+\leq \delta$ and $1-\delta\leq V_+\leq 1$ for some small $\delta>0$.
Since $\delta$ is sufficiently small, over the range $0\leq u\leq \delta$ and $1-\delta\leq v\leq1$,
there exists $m_1>0$ such that
$(\partial f/\partial u)(u,v)=r(1-2u-av)<-m_1$ ($a>1$ is also used). Thus, the mean value theorem gives
\bea\label{f1}
f(U_+,V_+)-f(U_{+}+U_{-}-1+\hat{p},V_+)\geq m_1(U_- -1+\hat{p})
\eea
for $0\leq U_+\leq \delta$ and $1-\delta\leq V_+\leq 1$. Also, by some simple computations,
\bea\label{f2}
&&f(U_{+}+U_{-}-1+\hat{p},V_+)-f(U_{+}+U_{-}-1+\hat{p},(1-\hat{q})(V_++V_-))\\
&&=ra(U_{+}+U_{-}-1+\hat{p})[-V_{+}+(1-\hat{q})(V_{+}+V_{-})]\notag\\
&&\geq-ra\hat{q}(U_{+}+U_{-}-1+\hat{p})(V_{+}+V_{-}).\notag
\eea
Due to the range of $U_+$ and $V_+$ in Case 1, we deduce from \eqref{f2} that
\bea\label{f2-2}
f(U_{+}+U_{-}-1+\hat{p},V_+)-f(U_{+}+U_{-}-1+\hat{p},(1-\hat{q})(V_++V_-))
\geq-2ra\hat{q}(\delta+\hat{p}).
\eea
Obviously, it holds
\bea\label{f3}
f(U_-,V_-)\geq -raU_-V_-\geq -raV_-.
\eea
As a consequence, by \eqref{f1}, \eqref{f2-2}, \eqref{f3} and the fact $U'_{\pm}\zeta'>0$,  we see from \eqref{N1-sec4}
that
\beaa
N_1[\overline{u},\underline{v}]\geq \hat{p}'+m_1(U_- -1+\hat{p})-2ra\hat{q}(\delta+\hat{p})-raV_-
\eeaa
for $0\leq U_+\leq \delta$ and $1-\delta\leq V_+\leq 1$.
In view of \eqref{est:1-U-} and \eqref{est:V-}, we obtain
\beaa
N_1[\overline{u},\underline{v}]\geq (-\beta \hat{p}_0 +m_1\hat{p}_0-2raq_0(\delta+\hat{p}_0e^{-\beta t})) e^{-\beta t}
-m_1 K_1e^{-\lambda_{u}(c_{uv}t-{\zeta_0})}-raK_2e^{-\lambda_{v}(c_{uv}t-{\zeta_0})}.
\eeaa
Hence, there exists $T_1\gg1$ such that $N_1[\overline{u},\underline{v}]\geq0$ for $x\geq0$ and $t\geq T_1$ within the range in Case 1, provided $\beta>0$ and $\delta>0$ are sufficiently small.

We next consider the inequality of $N_2[\overline{u},\underline{v}]$. Since $\zeta'V'_{\pm}<0$, from  \eqref{N2-sec4} it follows that,
{for all large $t$ such that $1-\hat{q}>0$, we have}
\beaa
N_2[\overline{u},\underline{v}]&\leq&  -\hat{q}'(V_{+}+V_{-})-(1-\hat{q})V_+[-V_-+\hat{q}(V_++V_-)-b(U_--1+\hat{p})]\\
&&\quad -(1-\hat{q})V_-[-V_++\hat{q}(V_++V_-)-b(U_+-1+\hat{p})]\\
&\leq& -\hat{q}'(V_{+}+V_{-})+2(1-\hat{q})V_+V_- - (1-\hat{q})\hat{q}(V_+ +V_-)^2+b(1-\hat{q})\hat{p}(V_++V_-)\\
&\leq& -2\hat{q}'+2(1-\hat{q})V_- -(1-\hat{q})\hat{q}(1-\delta)^2+2b(1-\hat{q})\hat{p},
\eeaa
where we have used $1-\delta \leq V_{+}\leq1$ and $0\leq V_{-}\leq1$. This, together with \eqref{est:V-}, yields
\beaa
N_2[\overline{u},\underline{v}]\leq 2\beta \hat{q}_0 e^{-\beta t}{+2(1-\hat{q})K_2e^{-\lambda_{v}(c_{uv}t-{\zeta_0})}}- (1-\hat{q})e^{-\beta t}\Big[\hat{q}_0 (1-\delta)^2-2b \hat{p}_0\Big].
\eeaa
Then one can find $T_2\gg1$ such that $N_2[\overline{u},\underline{v}]\leq0$ for $x\geq0$ and $t\geq T_2$ within the range in Case 1,
provided $\beta>0$ small enough and $\hat{q}_0 (1-\delta)^2>2b \hat{p}_0$ {by choosing $\delta$ smaller if necessary}.

{\bf Case 2:} $1-\delta\leq U_+\leq 1$ and $0\leq V_+\leq \delta$ for some small $\delta>0$. In this case,
there exists $m_2>0$ such that
$(\partial f/\partial u)(u,v)=r(1-2u-av)<-m_2$ for $1-\delta\leq u\leq 1$ and $0\leq v\leq\delta$.
This allows us to apply the same argument in Case 1 to deduce that for some large $T_3>0$, $N_1[\overline{u},\underline{v}]\geq0$ for $t\geq T_3$. The details are omitted here.

To verify $N_2[\overline{u},\underline{v}]\leq0$, we first observe that $V_-\leq V_+\leq\delta$ when $x\geq0$. Thus,
one can find $\kappa>0$ such that $V'_{\pm}\geq \kappa V_{\pm}$. Recall that $\zeta'<0$. Then we have
\beaa
(1-\hat{q})\zeta'(V'_++V'_-)\leq \kappa(1-\hat{q})\zeta'(V_++V_-),
\eeaa
{as long as $1-\hat{q}>0$.}
From  \eqref{N2-sec4} (also see the computation of $N_2[\overline{u},\underline{v}]$ in Case 1) we have
\beaa
N_2[\overline{u},\underline{v}]&\leq&
-\hat{q}'(V_{+}+V_{-})+\kappa(1-\hat{q})\zeta'(V_{+}+V_{-})+2(1-\hat{q})V_+V_- \\
 &&\quad- (1-\hat{q})\hat{q}(V_+ +V_-)^2+b(1-\hat{q})\hat{p}(V_++V_-)
\\
&\leq& (V_{+}+V_{-})\Big[-\hat{q}'+\kappa(1-\hat{q})\zeta'+2(1-\hat{q})\frac{V_+V_-}{V_+ +V_-}+b(1-\hat{q})\hat{p}\Big]\\
&\leq& (V_{+}+V_{-})\Big[\beta \hat{q}_0 e^{-\beta t} -\kappa(1-\hat{q}_0 e^{-\beta t})\Big(\frac{\beta|\zeta_1|}{2} e^{-(\beta/2) t}\Big)\\
&&\quad +2(1-\hat{q}_0 e^{-\beta t})K_2e^{-\lambda_{v}(c_{uv}t-{\zeta_0})} +b(1-\hat{q}_0 e^{-\beta t})\hat{p}_0e^{-\beta t}\Big],
\eeaa
{provided $1-\hat{q}>0$.}
Therefore, it is easily seen that, for some large $T_3>0$, $N_2[\overline{u},\underline{v}]\leq0$ for $t\geq T_3$ for all sufficiently small $\beta>0$.

{ When $(U_+,V_+)$ does not satisfy Case 1 and Case 2, we are led to consider:
}

{\bf Case 3:} {the middle part: $\delta_1\leq U_+,V_+\leq1-\delta_2$ for some small $\delta_i>0$ for $i=1,2$.}
In this range, there exists $\kappa_1>0$ such that
$U'_+\leq -\kappa_1$, which together with $U'_-<0$ and $\zeta'<0$ implies that $\zeta'(U'_+ + U'_-)\geq -\zeta' \kappa_1$.
For convenience, we use $C$ as a positive constant independent of $\Lambda:=(\hat{p}_0,\hat{q}_0,\beta,\zeta_0)$ and {$(\delta_1,\delta_2)$}, which may vary from
inequality to inequality.
By the Lipschitz continuity of $f$, there exists $C>0$
\beaa
|f(U_+,V_+)-f(U_{+}+U_{-}-1+\hat{p},V_+)|\leq C(\hat{p}-1+U_-).
\eeaa
Moreover, as seen in the calculations of \eqref{f2} and \eqref{f3},
\beaa
&&f(U_{+}+U_{-}-1+\hat{p},V_+)-f(U_{+}+U_{-}-1+\hat{p},(1-\hat{q})(V_++V_-))\geq -C(1+\hat{p}_0)\hat{q},\\
&&f(U_-,V_-)\geq -CV_-.
\eeaa
Therefore, using \eqref{est:V-}, from \eqref{N1-sec4} we get
\beaa
N_1[\overline{u},\underline{v}]&\geq&-\kappa_1\zeta'+\hat{p}'-C[\hat{p}-1+U_- +(1+\hat{p}_0)\hat{q}+V_-]\\
&\geq&\kappa_1\frac{\beta}{2}|\zeta_1|e^{-(\beta/2) t}-\beta \hat{p}_0e^{-\beta t}
-C[\hat{p}_0e^{-\beta t}+(1+\hat{p}_0)\hat{q}_0e^{-\beta t}+e^{-\lambda_{v}(c_{uv}t-{\zeta_0})})].
\eeaa
Then there exists $T_4\gg1$ such that $N_1[\overline{u},\underline{v}]\geq0$ for all $t\geq T_4$, provided $\beta>0$ is sufficiently small.

On the other hand, in this range there exists $\kappa_2>0$ such that $V'_+\geq \kappa_2$, which
 together with $V'_->0$ and $\zeta'<0$ imply that $\zeta'(V'_+ + V'_-)\leq \zeta' \kappa_2$.
Thanks to \eqref{est:V-}, we see from \eqref{N2-sec4} that
\beaa
N_2[\overline{u},\underline{v}]&\leq&
-\hat{q}'(V_+ +V_-)+\kappa_2(1-\hat{q})\zeta'+2(1-\hat{q})V_+V_- +b(1-\hat{q})\hat{p}(V_++V_-)\\
&\leq& 2\beta \hat{q}_0 e^{-\beta t}-\kappa_2(1-\hat{q}_0e^{-\beta t}) \frac{\beta}{2}|\zeta_1|e^{-(\beta/2) t}+C[e^{-\lambda_{v}(c_{uv}t-{\zeta_0})}+\hat{p}_0e^{-\beta t}],
\eeaa
{provided $1-\hat{q}>0$.}
Then there exists $T_5\gg1$ such that $N_2[\overline{u},\underline{v}]\leq0$ for all $t\geq T_5$, provided $\beta>0$ is sufficiently small.

Combining the discussion in Cases 1-3 and taking $T^*=\max\{T_1,T_2,T_3,T_4,T_5\}$,
indeed we have shown that there exists some small $\beta^*>0$ such that
\beaa
N_1[\overline{u},\underline{v}]\geq0,\quad N_2[\overline{u},\underline{v}]\leq0\quad\mbox{for $x\in\mathbb{R}$ and $t\geq T^*$}
\eeaa
provided $\beta\in(0, \beta^*)$, {$-\zeta_0\gg1$} and $\hat{q}_0 (1-\delta)^2>2b \hat{p}_0$. {This completes the proof.}
\end{proof}

\begin{lemma}\label{lem:v-to-1}
For each $c>c_u:=2\sqrt{rd}$, {$v(t,x)$} converges to $1$ uniformly for {$|x|\geq ct$} as $t\to\infty$.
\end{lemma}
\begin{proof}
The argument is similar to that of \cite[Lemma 2]{Carrere} with minor modifications; we omit the details here.
\end{proof}

{
\begin{lemma}\label{lem:order-3}
Let $(\overline{u},\underline{v})$ be defined in \eqref{supersol-2}.
Then there exist $\beta, \hat{p}_0, \hat{q}_0>0$, $\zeta_1<0$, $\widetilde{T}>0$,
$T^{*}>0$ and $\zeta^{*}_0<0$ such that
the solution $(u,v)$
of \eqref{LV-sys} and  \eqref{LV-ic} with {\bf(A1)} satisfies
\beaa
u(t+\widetilde{T},x)\leq\overline{u}(t,x),\quad v(t+\widetilde{T},x)\geq \underline{v}(t,x)\quad \mbox{in $[T^*,\infty)\times(-\infty,\infty)$},
\eeaa
for any $\zeta_0\leq \zeta_0^{*}$.
\end{lemma}
}
{
\begin{proof}
By Lemma~\ref{lem:super-sol2}, one can choose suitable $\beta, \hat{p}_0, \hat{q}_0>0$, $\zeta_0\in \mathbb{R}$, $\zeta_1<0$ and
$T_0\gg1$ such that \eqref{upper-ineq*} holds for $t\geq T_0$. Furthermore, from the proof of Lemma~\ref{lem:super-sol2} we see that
$T_0$ can be chosen independently for all large negative $\zeta_0$ (decreasing $\zeta_0$ strengthens the differential inequalities therein).

In view of Lemma~\ref{lem:simple est},
one can take $T_1\gg1$ such that
\bea\label{u-ub}
u(t,x)\leq 1+ Me^{-rt}\quad\mbox{ for all $t\geq T_1$ and $x\in\mathbb{R}$}.
\eea

We now fix any $L>0$. Thanks to Lemma~\ref{lem:AS-}, for $x\in[-L,L]$ and $t>0$,
\beaa
\overline{u}(t,x)&\geq& U(L-c_{uv}t+\zeta(t))+U(L-c_{uv}t+\zeta(t))-1+\hat{p_0}e^{-\beta t}\\
&\geq& 2(1-Ke^{-\lambda_u(c_{uv}t-\zeta_0)})-1+\hat{p_0}e^{-\beta t},\\
&=& 1-2Ke^{-\lambda_u(c_{uv}t-\zeta_0)}+\hat{p_0}e^{-\beta t}
\eeaa
for some $K>0$ (independent of all negative $\zeta_0$) and $\lambda_u>0$.
Then there exists $T_2\gg1$ such that
\bea\label{u-middle}
\overline{u}(t,x)\geq 1+M e^{-rt}\quad \mbox{for all $t\geq T_2$, $x\in[-L,L]$ and all large negative  $\zeta_0$},
\eea
if necessary we choose $\beta$ smaller such that $\beta<\min\{r,\lambda_uc_{uv}\}$.

On the other hand, since $\beta<\lambda_uc_{uv}$, there exist $T_3\gg1$ and $K'>0$ such that for all $t\geq T_3$,  $x\geq L$ and all large negative $\zeta_0$,
\bea\label{u-infty part}
\overline{u}(t,x)\geq U(x-c_{uv}t+\zeta(t))-K'e^{-\lambda_uc_{uv}t}+\hat{p_0}e^{-\beta t}
\geq U(x-c_{uv}t+\zeta(t)).
\eea
Similarly, we can find $T_4\gg1$ such that for all $t\geq T_4$,  $x\leq -L$ and all large negative  $\zeta_0$,
\bea\label{u-infty part2}
\overline{u}(t,x)\geq U(-x-c_{uv}t+\zeta(t)).
\eea

Set $T^{*}=\max\{T_0,T_1,T_2,T_3,T_4\}$. We shall prove that for some $\widetilde{T}>0$,
\bea\label{cp-ic}
u(T^*+\widetilde{T},x)\leq\overline{u}(T^*,x),\quad v(T^*+\widetilde{T},x)\geq \underline{v}(T^*,x)\quad \mbox{for all $x\in\mathbb{R}$.}
\eea
To see this, we first note that $\underline{v}(T^*,+\infty)= 1-\hat{q}_0e^{-\beta T^*}$, by Lemma~\ref{lem:v-to-1},
we can find $\widetilde{T}\gg1$ such that $v(T^*+\widetilde{T},x)>1-\hat{q}_0e^{-\beta T^*}$ for all $|x|\gg1$.
Together with the fact that $V(-\infty)=0$,
there exists $\tilde{\zeta}\gg1$ such that
$\underline{v}(T^*,x)\leq v(T^*+\widetilde{T},x)$ for all $x\in\mathbb{R}$ if $\zeta_0\leq -\tilde{\zeta}$.
On the other hand,
using \eqref{u-ub}, \eqref{u-middle}, we have
\bea\label{u-middle2}
u(T^*+\widetilde{T},x)\leq 1+ Me^{-r(T^*+\widetilde{T})}\leq 1+ Me^{-rT^*}\leq\overline{u}(T^*,x)
\eea
for all $x\in[-L,L]$ and all large negative $\zeta_0$.

Using \eqref{u-infty part} and \eqref{u-infty part2} and Lemma~\ref{lem:AS+}, one has
\beaa
\overline{u}(T^*,x)\geq C e^{-\lambda_1 x}\quad \mbox{for all $|x|\geq L$},
\eeaa
where $C$ and $\lambda_1$ are two positive constants.
Together with the fact that $u(T^*+\widetilde{T},x)=O(e^{-x^2/(8d(T^*+\widetilde{T}))})$, 
we see that
\bea\label{u-outside}
u(T^*+\widetilde{T},x) \leq \overline{u}(T^*,x)\quad \mbox{for all $|x|\gg1$}.
\eea
Therefore, combining \eqref{u-middle2} and \eqref{u-outside}, one
can find $\hat{\zeta}\gg1$ such that
$u(T^*+\widetilde{T},x)\leq\overline{u}(T^*,x)$ for all $x\in\mathbb{R}$ if $\zeta_0\leq -\hat{\zeta}$.

From the above discussion, we see that
\eqref{cp-ic} holds provided $\zeta_0\leq \zeta_0^*:=-\max\{\hat{\zeta},\tilde{\zeta}\}$.
Therefore, Lemma~\ref{lem:order-3} follows from the comparison principle.
\end{proof}
}

{
\subsubsection{The construction of a subsolution}
Denote a subsolution $(\underline{u},\overline{v})$ by
\bea\label{subsol-2}
\begin{cases}
\ \underline{u}(t,x):=\max\{U(x-c_{uv}t+\zeta(t))+U(-x-c_{uv}t+\zeta(t))-1-p(t),0\},\\
\ \overline{v}(t,x):=(1+q(t))\Big[V(x-c_{uv}t+\zeta(t))+V(-x-c_{uv}t+\zeta(t))\Big],
\end{cases}
\eea
where $\zeta(t)$, $\hat{p}(t)$ and $\hat{q}(t)$ have the same form as in \eqref{parameters2}.
}

{
\begin{lemma}\label{lem:subsol-2}
Let $\underline{u}$ and $\overline{v}$ be defined in \eqref{subsol-2}.
For any $\hat{p}_0,\hat{q}_0>0$ with $\hat{q}_0>2b(1+\hat{q}_0)\hat{p}_0$ and $\zeta_1>0$,
there exists $T^{**}\geq0$ such that
\bea\label{upper-ineq**}
N_1[\underline{u},\overline{v}]\geq0,\quad N_2[\underline{u},\overline{v}]\leq0\quad \mbox{in $[T^{**},\infty)\times (-\infty,\infty)$}
\eea
for all $\zeta_0\leq0$, provided $\beta>0$ is sufficiently small.
\end{lemma}
}
{\begin{proof} Since the proof is similar to that of Lemma~\ref{lem:super-sol2},
we give some sketch of the proof, but give the details for the different parts.

For notational convenience, we also denote
$\xi_{\pm}=\pm x-c_{uv}t+\zeta(t)$ and $(U_{\pm},V_{\pm})=(U(\xi_{\pm}),V(\xi_{\pm}))$.
We first consider $\underline{u}>0$. After some direct computation (similar to that of Lemma~\ref{lem:super-sol2}, but replace $\hat{p}$ by $-\hat{p}$ and
$-\hat{q}$ by $\hat{q}$), we obtain
\bea\label{N1-sub-super}
N_1[\underline{u},\overline{v}]
&=&\zeta'(t)(U'_{+}+U'_{-})-\hat{p}'+f(U_+,V_+)+f(U_-,V_-)\\
&&\quad-f(U_{+}+U_{-}-1-\hat{p},V_+)+f(U_{+}+U_{-}-1-\hat{p},V_+)\notag\\
&&\quad-f(U_{+}+U_{-}-1-\hat{p},(1+\hat{q})(V_{+}+V_{-})),\notag
\eea
where $f(u,v)=ru(1-u-av)$.
Similarly,  we have
\bea\label{N2-sub-super}
\ \ \ N_2[\underline{u},\overline{v}]
                               &=&\hat{q}'(t)(V_{+}+V_{-})+(1+\hat{q})\zeta'(t)(V'_{+}+V'_{-})\\
                                   &&\qquad+(1+\hat{q})V_+[V_- +\hat{q}(V_+ +V_-)+b(U_--1-\hat{p})]\notag\\
                                   &&\qquad +(1+\hat{q})V_-[V_++\hat{q}(V_++V_-)+b(U_+-1-\hat{p})],\notag
\eea
where $g(u,v)=v(1-v-bu)$.

We shall show that $N_1[\underline{u},\overline{v}]\geq0$ and $N_2[\underline{u},\overline{v}]\leq0$ for all $x\in\mathbb{R}$ and $t\gg1$.
By symmetry, we only consider the range $x\geq0$.
Let us fix $\zeta_1>0$ such that $\zeta'>0$. Since $x\geq0$, $U'<0$ and $\zeta'<0$,
it follows from Lemma~\ref{lem:AS-} that
\bea\label{key est:1-U-}
1-U_-\leq 1-U(-c_{uv}t+\zeta_0)\leq K_1e^{-\lambda_{u}(c_{uv}t-\zeta_0)}\quad \mbox{for all $x\geq0$ and $t\geq0$},
\eea
where $\lambda_u$ and $K_1$ are positive constants independent of all $\zeta_0\leq0$.
Similarly, there exist two constants $\lambda_v>0$ and $K_2>0$ (independent of all
$\zeta_0\leq0$) such that
\bea\label{key est:V-}
V_-\leq K_2e^{-\lambda_{v}(c_{uv}t-\zeta_0)}\quad \mbox{for all $x\geq0$ and $t\geq0$.}
\eea

As in the proof of Lemma~\ref{lem:super-sol2}, we divide the discussion into three cases:

{\bf Case 1:} $0\leq U_+\leq \delta$ and $1-\delta\leq V_+\leq 1$ for some small $\delta>0$.
Since $\delta$ is sufficiently small, as in the proof of Case (i) in Lemma~\ref{lem:super-sol2},
there exists $m_1>1$ such that
\bea\label{f1'}
f(U_+,V_+)-f(U_{+}+U_{-}-1-\hat{p},V_+)\leq -m_1(\hat{p}+1-U_-)\leq -m_1\hat{p}
\eea
for $0\leq U_+\leq \delta$ and $1-\delta\leq V_+\leq 1$. Also, by some simple computations, it holds that
\beaa
&&f(U_{+}+U_{-}-1-\hat{p},V_+)-f(U_{+}+U_{-}-1-\hat{p},(1+\hat{q})(V_++V_-))\\
&&=ra(U_{+}+U_{-}-1-\hat{p})[\hat{q}V_{+}+(1+\hat{q})V_{-}]\notag\\
&&\leq raU_{+}[\hat{q}V_{+}+(1+\hat{q})V_{-}].\notag
\eeaa
With the range of $U_+$ and $V_+$ in Case 1, by \eqref{key est:V-}, we further have
\bea\label{f2'}
&&f(U_{+}+U_{-}-1-\hat{p},V_+)-f(U_{+}+U_{-}-1-\hat{p},(1+\hat{q})(V_++V_-))\\
&&\leq ra\delta \hat{q} +ra\delta(1+\hat{q})K_2e^{-\lambda_{v}(c_{uv}t-\zeta_0)}.\notag
\eea
Also, in view of \eqref{key est:1-U-}, one has
\bea\label{f3'}
f(U_-,V_-)\leq rU_-(1-U_-)\leq r  K_1e^{-\lambda_{u}(c_{uv}t-\zeta_0)}.
\eea
Combining \eqref{f1'}, \eqref{f2'}, \eqref{f3'} and the fact $U'_{\pm}\zeta'>0$,  we see from \eqref{N1-sub-super}
that
\beaa
N_1[\underline{u},\overline{v}]&\leq& (\beta\hat{p}_0-m_1\hat{p}_0+ra\delta\hat{q}_0)e^{-\beta t}\\ &&+\delta(1+\hat{q}_0)K_2e^{-\lambda_{v}(c_{uv}t-\zeta_0)}+r K_1e^{-\lambda_{u}(c_{uv}t-\zeta_0)}
\eeaa
for $0\leq U_+\leq \delta$ and $1-\delta\leq V_+\leq 1$.
Therefore, there exists $T_1\gg1$ such that $N_1[\underline{u},\overline{v}]\leq0$ for $x\geq0$ and $t\geq T_1$ within the range in Case 1, provided $\beta>0$ and $\delta>0$ are sufficiently small.

We next deal with the inequality of $N_2[\underline{u},\overline{v}]$. Since $\zeta'V'_{\pm}>0$, from  \eqref{N2-sub-super} one has
\beaa
N_2[\underline{u},\overline{v}]&\geq&  \hat{q}'(V_{+}+V_{-})+(1+\hat{q})\hat{q}(V_+ +V_-)^2
-b(1+\hat{q})V_+[(1-U_-)+\hat{p}]\\
&&-b(1+\hat{q})(1+\hat{p})V_-\\
&\geq& 2\hat{q}'+(1+\hat{q})\hat{q}(1-\delta)^2-b(1+\hat{q}_0)[K_1e^{-\lambda_{u}(c_{uv}t-\zeta_0)}+\hat{p}+(1+\hat{p})K_2e^{-\lambda_{v}(c_{uv}t-\zeta_0)}]\\
&=& e^{-\beta t}[\hat{q}_0(1-\delta)^2-2\beta \hat{q}_0-b(1+\hat{q}_0)\hat{p}_0]\\
&&-b(1+\hat{q}_0)[K_1e^{-\lambda_{u}(c_{uv}t-\zeta_0)}+(1+\hat{p}_0)K_2e^{-\lambda_{v}(c_{uv}t-\zeta_0)}]
\eeaa
where we have used $1-\delta \leq V_{+}\leq1$, $0\leq V_{-}\leq1$, \eqref{key est:1-U-} and \eqref{key est:V-}.
Then one can find $T_2\gg1$ such that $N_2[\underline{u},\overline{v}]\geq0$ for $x\geq0$ and $t\geq T_2$ within the range in Case 1,
provided $\delta>0$ is chosen small enough and
\beaa
0<\beta<\min\{\lambda_{u}c_{uv},\lambda_{v}c_{uv}, 1/4\},\quad \hat{q}_0>2\beta\hat{q}_0+b(1+\hat{q}_0)\hat{p}_0.
\eeaa

{\bf Case 2:} $1-\delta\leq U_+\leq 1$ and $0\leq V_+\leq \delta$ for some small $\delta>0$. Since
there exists $m_2>0$ such that $(\partial f/\partial u)(u,v)=r(1-2u-av)<-m_2$ for $1-\delta\leq u\leq 1$ and $0\leq v\leq\delta$,
one can apply the same argument in Case 1 to deduce that for some large $T_3>0$, $N_1[\underline{u},\overline{v}]\leq0$ for $t\geq T_3$
as long as  $\beta>0$ and $\delta>0$ are small enough. We omit the details here.

We now show $N_2[\underline{u},\overline{v}]\leq0$. Observe that $V_-\leq V_+\leq\delta$ when $x\geq0$ and recall that $\zeta'>0$. Thus,
one can find $\kappa>0$ such that $V'_{\pm}\geq \kappa V_{\pm}$ and so
\beaa
(1+\hat{q})\zeta'(V'_++V'_-)\geq \kappa\zeta'(V_++V_-).
\eeaa
From  \eqref{N2-sub-super} and using estimates \eqref{key est:1-U-} and \eqref{key est:V-} we have
\bea\label{N2-case2}
&&N_2[\underline{u},\overline{v}]\\
&&\geq \hat{q}'(V_{+}+V_{-})+\kappa\zeta'(V_{+}+V_{-})+b(1+\hat{q})[V_+(U_- -1-\hat{p})+V_-(U_+ -1-\hat{p})]\notag
\\
&&\geq V_+[\hat{q}'+\kappa\zeta'-b(1+\hat{q}_0)K_1e^{-\lambda_{u}(c_{uv}t-\zeta_0)}-b(1+\hat{q}_0)\hat{p}]
\notag\\
&&\qquad+V_-[\hat{q}'+\kappa\zeta'-b(1+\hat{q}_0)\hat{p}]-b(1+\hat{q}_0)V_-(1-U_+)\notag\\
&&= V_+e^{-(\beta/2)t}\Big[-\beta\hat{q}_0e^{-(\beta/2)t}+\kappa\frac{\beta}{2}\zeta_1-K'e^{-\lambda_{u}(c_{uv}t-\zeta_0)+(\beta/2)t}-b(1+\hat{q}_0)\hat{p}_0 e^{-(\beta/2)t}\Big]\notag\\
&&\qquad+V_-e^{-(\beta/2)t}\Big[-\beta\hat{q}_0e^{-(\beta/2)t}+\kappa\frac{\beta}{2}\zeta_1-b(1+\hat{q}_0)\hat{p}_0 e^{-(\beta/2)t}\Big]-b(1+\hat{q}_0)V_-(1-U_+)\notag
\eea
where $K':=b(1+\hat{q}_0)K_1$.

To obtain $N_2[\underline{u},\overline{v}]\leq0$, we need to estimate the last term $b(1+\hat{q}_0)V_-(1-U_+)$.
For this, we observe that over this range, $1-\delta\leq U_+\leq1$, we must have $x-c_{uv}t+\zeta(t)\leq -N_{\delta}$ for some $N_\delta>0$
since $U'<0$ and $U(+\infty)=0$. This means that $x-c_{uv}t=O(1)$. If $x\in[0,c_{uv}t/2]$, we apply Lemma~\ref{lem:AS-} to ensure that
\bea\label{key-case2-1}
b(1+\hat{q}_0)V_-(1-U_+)\leq C_1V_-e^{\lambda_{u}(c_{uv}t/2-c_{uv}t+\zeta_0)}= C_1 V_-e^{-\lambda_{u}(c_{uv}t/2-\zeta_0)}
\eea
for some $C_1>0$ and $\lambda_u>0$. If $c_{uv}t/2\leq x\leq c_{uv}t+O(1)$, by Lemma~\ref{lem:AS-} again, we obtain
\beaa
\qquad V_-(1-U_+)\leq \delta V_-=\delta \frac{V_-}{V_+}V_+
\leq \delta\frac{V(-c_{uv}t/2-c_{uv}t+\zeta(t))}{V( c_{uv}t/2-c_{uv}t+\zeta(t))}V_+\leq C_2e^{-c_{uv}\lambda_{4}t}V_+,
\eeaa
where $\lambda_4>0$ given in Lemma~\ref{lem:AS-} and $C_2$ is a positive constant. Therefore, if $c_{uv}t/2\leq x\leq c_{uv}t+O(1)$,
we have
\bea\label{key-case2-2}
b(1+\hat{q}_0)V_-(1-U_+)\leq C_3e^{-c_{uv}\lambda_{4}t}V_+
\eea
for some $C_3>0$.

Combining \eqref{N2-case2}, \eqref{key-case2-1} and \eqref{key-case2-2}, we obtain that
for some large $T_4>0$, $N_2[\underline{u},\overline{v}]\leq0$ for $t\geq T_4$ as long as $\beta>0$ is chosen sufficiently small.

{\bf Case 3:} {the middle part: $\delta_1\leq U_+,V_+\leq1-\delta_2$ for some small $\delta_i>0$ for $i=1,2$.}
In this range, one can follow the similar process to that of Case 3 in Lemma~\ref{lem:super-sol2} to ensure that
there exists $T_5\gg1$ such that $N_1[\underline{u},\overline{v}]\leq0$ and $N_2[\underline{u},\overline{v}]\geq0$
for all $t\geq T_5$, provided $\beta>0$ is chosen sufficiently small. We omit the details here.

From the above discussion, taking $T^{**}=\max\{T_1,T_2,T_3,T_4,T_5\}$, we have shown that
there exists some small $\beta^{**}>0$ such that
$N_1[\underline{u},\overline{v}]\leq0$ and $N_2[\underline{u},\overline{v}]\geq0$ for $x\in\mathbb{R}$ and $t\geq T^{**}$ if $\underline{u}>0$,
provided $\beta\in(0, \beta^{**})$ and $\hat{q}_0>2b(1+\hat{q}_0)\hat{p}_0$. When $\underline{u}=0$, it is not hard to see that the above conclusion
still holds and we omit the details here.
This completes the proof.
\end{proof}
}

{
\begin{lemma}\label{lem:order-4}
Let $\underline{u}$ and $\overline{v}$ be defined in \eqref{subsol-2}.
Then there exist $\beta, \hat{p}_0, \hat{q}_0>0$, $\zeta_0\in\mathbb{R}$, $\zeta_1<0$ large $T^{**},\widehat{T}>0$ such that the solution $(u,v)$
of \eqref{LV-sys} and \eqref{LV-ic} {\rm(}with {\bf(A1)} or {\bf(A2)}{\rm)}  satisfies
\beaa
u(t+\widehat{T},x)\geq\underline{u}(t,x),\quad v(t+\widehat{T},x)\leq \overline{v}(t,x)\quad \mbox{for $t\geq T^{**}$ and $x\in\mathbb{R}$}.
\eeaa
\end{lemma}
}
{\begin{proof}
By Lemma~\ref{lem:subsol-2}, one can choose suitable $\beta, \hat{p}_0, \hat{q}_0>0$, $\zeta_0\in\mathbb{R}$, $\zeta_1<0$ and $T^{**}>0$ such that $N_1[\underline{u},\overline{v}]\leq0$ and $N_2[\underline{u},\overline{v}]\geq0$ for $x\in\mathbb{R}$ and $t\geq T^{**}$. Note that $\overline{v}(T^{**},+\infty)=1+\hat{q}_0e^{-\beta T^{**}} $ and $\underline{u}(T^{**},x)=0$ for all $|x|\gg1$ (this is because $\beta>0$ is chosen sufficiently small).
One can take large $L>0$ such that
\bea\label{super-sub-ic1}
\overline{v}(T^{**},x)\geq 1+\frac{\hat{q}_0}{2}e^{-\beta T^{**}} ,\quad \underline{u}(T^{**},x)=0\quad \mbox{for all $|x|\geq L$}.
\eea
From the definition of $(\underline{u},\overline{v})$, it is obvious that
\bea\label{super-sub-ic2}
\max_{x\in[-L,L]}\underline{u}(T^{**},x)\leq 1-\hat{p}(T^{**}),\quad A:=\min_{x\in[-L,L]}\overline{v}(T^{**},x)>0.
\eea

Let us fix such $T^{**}$ and $L$.
In view of Lemma~\ref{lem:simple est} (if necessary we take a smaller $\beta$) and \eqref{super-sub-ic1},
we can find $T_1>0$ such that
\bea\label{super-sub-ic3}
v(t+T^{**},x)\leq 1+(\hat{q}_0/2)e^{-\beta T^{**}}\leq \overline{v}(T^{**},x)
\eea
for all $t\geq T_1$ and $|x|\geq L$.

Inside the region $[-L,L]$,  {\bf (H3)} and \eqref{super-sub-ic2} yield that for some large $T_2>0$,
\bea\label{super-sub-ic4}
u(t+T^{**},x)\geq 1-\hat{p}(T^{**})\geq \underline{u}(T^{**},x),\quad  v(t+T^{**},x)\leq A\leq \overline{v}(T^{**},x)
\eea
for all $t\geq T_2$ and $x\in[-L,L]$.

Let us set $\widehat{T}:=\max\{T_1,T_2\}$. Combining \eqref{super-sub-ic1}, \eqref{super-sub-ic3} and \eqref{super-sub-ic4}, we have
\beaa
u(\widehat{T}+T^{**},x)\geq \underline{u}(T^{**},x),\quad v(\widehat{T}+T^{**},x)\leq \overline{v}(T^{**},x) \quad \mbox{for all  $x\in\mathbb{R}$}.
\eeaa
Therefore, one can compare $(u,v)(t+\widehat{T},\cdot)$ with $(\underline{u},\overline{v})(t,\cdot)$ from $t=T^{**}$ and thus the proof of Lemma~\ref{lem:order-4} is complete.
\end{proof}
}

{
\begin{remark}\label{rk-H3}
From the proof of Lemma~\ref{lem:order-4}, we see that
if $(u_0,v_0)$ satisfies
\beaa
u_0(x)\geq \underline{u}(T^{**},x),\quad v_0(x)\leq \overline{v}(T^{**},x),\quad \forall\ x\in\mathbb{R},
\eeaa
then $u(t,x)\geq \underline{u}(t+T^{**},x)$ and $v(t,x)\leq \overline{v}(t+T^{**},x)$ for $t\geq0$ and $x\in\mathbb{R}$,
which establish the successful invasion of $u$.
Roughly speaking, $u_0$ needs to have sufficiently large support, and the amplitude of
$u_0$ require to be large in the support; while the amplitude of $v_0$ cannot be too large.
This provides some initial functions which satisfy {\bf(A1)} or {\bf(A2)} such that
{\bf(H3)} holds.
\end{remark}
}

\medskip

{
We are now able to prove Theorem~\ref{thm1} by applying the same argument used in the previous section.
}

{
\begin{proof}[Proof of Theorem~\ref{thm1}]
Set $\xi=x-c_{uv}t$ with $x\geq0$. Define the solution of \eqref{LV-sys} and {\eqref{LV-ic}}
as
\beaa
(\hat{u},\hat{v})(t,\xi)=({u},{v})(t,x)=({u},{v})(t,\xi+c_{uv}t),\quad t>0,\ \xi\geq -c_{uv}t.
\eeaa
Then $(\hat{u},\hat{v})$ satisfies the system \eqref{LV-sys-moving}.
In view of Lemma~\ref{lem:order-3} and Lemma~\ref{lem:order-4}, one has the following result:
for some suitable $\zeta^{*}_i$, $\zeta^{**}_i$ $(i=0,1)$, $\hat{p}^{*}_0$, $\hat{p}^{**}_0$, $\hat{q}^{*}_0$, $\hat{q}^{**}_0$,
$\beta$ and large $T>0$,
\beaa
&&U(\xi+\zeta^{*}_0-{\zeta}^{*}_1 e^{-(\beta/2)t})+U(-c_{uv}t+\zeta^{*}_0-{\zeta}^{*}_1 e^{-(\beta/2)t})-1-p^{*}_0e^{-\beta t}\\
&&\qquad\leq \hat{u}(t,\xi) \leq U(\xi+\zeta^{**}_0-{\zeta}^{**}_1 e^{-(\beta/2)t})+p^{**}_0e^{-\beta t},
\eeaa
and
\beaa
&&(1-q^{**}_0e^{-\beta t})V(\xi+\zeta_0^{**}-{\zeta}_1^{**} e^{-(\beta/2)t})\\
&&\qquad\leq \hat{v}(t,\xi)\leq (1+q^{*}_0e^{-\beta t})[V(\xi+\zeta_0^{*}-{\zeta}_1^{*} e^{-(\beta/2)t})+V(-c_{uv}t+\zeta_0^{**}-{\zeta}_1^{**} e^{-(\beta/2)t})]
\eeaa
for all $t\geq T$ and $\xi\geq -c_{uv}t$.
With the above inequalities, one can follow the same line as that in the proof of Proposition~\ref{prop1} to conclude Theorem~\ref{thm1}.
\end{proof}
}

\section{Proof of Theorems \ref{thm2} and \ref{thm3}: scenario {\bf(A2)}}\label{sec:A2}
\setcounter{equation}{0}

In this section, we prove Theorems \ref{thm2} and \ref{thm3}; unless otherwise specified, it is assumed that $(u_0,v_0)$ satisfies {\bf(A2)} throughout this section.

\subsection{The proof of Theorem~\ref{thm2}}

Let $(u,v)$ be the solution of \eqref{LV-sys}-\eqref{LV-ic}. Given $m\in(0,1)$, we define
$E_m(t)$ as the set of points in $(0,\infty)$ such that $u(t,\cdot)=m$. Namely,
\beaa
E_m(t)=\{x>0:\ \,u(t,x)=m\}.
\eeaa

\begin{lemma}\label{lem:upper-E}
For any $m\in(0,1)$, there exist $M>0$ and $T>0$ such that
\beaa
\max E_m(t)\leq c_u t-{\frac{3{d}}{c_u}}\ln t+M,\ \ \forall t\geq T.
\eeaa
\end{lemma}
\begin{proof}
Let $\bar{u}$ be the solution of the problem
$$\bar{u}_t=d \bar{u}_{xx}+r(1-\bar{u})\bar{u}\ \ \mbox{in}\ (0,\infty)\times\mathbb{R}; \ \ \bar{u}(0,\cdot)=u_0.$$
From \cite{Bramson} or \cite[{Theorem} 1.1]{HNRR}, we see that there exist $M>0$ and $T>0$ such that
{ \beaa
\bar{E}_m(t)\subset\left[c_u t-{\frac{3}{2\lambda_u}}\ln t-M,\ c_u t-{\frac{3}{2\lambda_u}}\ln t+M\right]\quad \mbox{for all $t\geq T$},
\eeaa
where $\bar{E}_m(t)=\{x>0:\ \,\bar{u}(t,x)=m\}$  and $\lambda_u$ is the (double) root of the
characteristic equation
$d\lambda^2-c_u\lambda+r=0$. That is, $\lambda_u=c_u/2d$. Therefore, we have}
\bea\label{E-bar}
\bar{E}_m(t)\subset\left[c_u t-\frac{3{d}}{c_u}\ln t-M,\ c_u t-\frac{3{d}}{c_u}\ln t+M\right]\quad \mbox{for all $t\geq T$},
\eea

Since $\bar{u}_t\geq d \bar{u}_{xx}+r(1-\bar{u}-av)\bar{u}$ in $(0,\infty)\times\mathbb{R}$, one can apply the comparison principle to deduce $\bar{u}\geq u$, which implies that $\max E_m(t)\leq \max \bar{E}_m(t)$ for $t\geq T$. Using \eqref{E-bar}, we thus complete the proof.
\end{proof}

\begin{lemma}\label{lem:v-vanishing}
Assume that $c_u>c_v$. Then there exist $C,\mu, T>0$ such that
\beaa
\sup_{x\in\mathbb{R}^+}v(x,t)\leq C e^{-\mu t},\quad \forall t\geq T.
\eeaa
\end{lemma}
\begin{proof}
Since $v_0$ is of compact support,
by the proof of Lemma~\ref{lem:exp-decay} (just exchanging the role of $u$ and $v$), we have the following result:
if $c>c_v:=2$, then there exist $M,{\mu'}>0$ and $T\gg1$ such that
\beaa
v(t,x)\leq M e^{-{\mu'} ({c-2}\,)t}\quad\mbox{for all $t\geq T$ and $x>ct$}.
\eeaa
Together with Corollary~\ref{cor:v-rate}, we thus complete the proof.
\end{proof}

We next derive a lower estimate of $\min E_m(t)$. For our purpose, consider
\bea\label{subsol-eq}
\underline{u}_t=d \underline{u}_{xx}+\underline{u}(r-r\underline{u}-C_0 e^{-\mu t})\quad \mbox{in $(0,\infty)\times\mathbb{R}$},
\eea
where $C_0:=raC$, where $\mu,C>0$ is defined in Lemma~\ref{lem:v-vanishing}.
We shall apply the method developed by  Hamel, Nolen, Roquejoffre, and Ryzhik \cite{HNRR} to estimate $\min E_m(t)$. To do so, we consider the linearized equation of \eqref{subsol-eq}
with the Dirichlet boundary condition along a suitable curve $x=X(t)$. Namely,
\bea\label{linearized-subeq}
w_t=dw_{xx}+w(r-C_0 e^{-\mu t})\quad \mbox{in $(0,\infty)\times(X(t),\infty)$ with $w(t,X(t))=0$},
\eea
where  $w(0,\cdot)=w_0\geq {(\not\equiv)0}$ in $(0,\infty)$ and is of compact support.

Motivated by \cite{HNRR}, we define
\beaa
X(t):=c_u t-\frac{3d}{c_u}\ln(t+t_0),\quad z(t,x')=w(t,x),\quad x'=x-X(t),
\eeaa
where $t_0>0$ will be determined later. After some simple calculations and dropping the prime sign, \eqref{linearized-subeq} becomes
\bea\label{z-eq}
z_t=d z_{xx}+\left[c_u-\frac{3d}{c_u(t+t_0)}\right]z_x+(r-C_0 e^{-\mu t})z\quad \mbox{in $(0,\infty)\times(0,\infty)$},
\eea
where $z(t,0)=0$ and $z(0,\cdot)=z_0\geq {(\not\equiv)0}$ in $(0,\infty)$ and is of compact support.

\medskip

We shall prove that $z(t,x)$  has both positive upper and lower bounds over $[1,\infty)\times[a,b]$ for any given $0<a<b<\infty$ using the argument of
\cite[Lemma 2.1]{HNRR}. To {this}  end, we need the following lemma given in \cite{HNRR}.

\begin{lemma}\label{lem:HNRR}{\rm(}\cite[Lemma 2.2]{HNRR}{\rm)}
Suppose that $p(t,y)$ satisfies
\beaa
p_\tau+\mathcal{L} p=-\varepsilon e^{-\tau/2}p_{y},\quad \tau>0,\ y>0;\quad p(\tau,0)=0,
\eeaa
where $$\mathcal{L}p:=- p_{yy}-y p_y/2-p.$$
Then there exists $\varepsilon_0>0$ such that for any compact set $K$ of $\mathbb{R}_+$, there exists $C_K>0$ such that
for $0<\varepsilon<\varepsilon_0$,
\beaa
p(\tau,y)=y\left[\frac{e^{-y^2/4}}{2\sqrt{\pi}} \Big(\int_0^\infty \xi p(0,\xi)d\xi+O(\varepsilon)\Big)+e^{-\tau/2}\tilde{p}(\tau,y)\right],
\eeaa
where $|\tilde{p}(\tau,y)|\leq C_K e^{-y^2/8}$ for all $\tau>0$ and $y\in K$; and $O(\varepsilon)$ denote a function of $(\tau,y)$ for $\tau>0$ and $y\in K$.
\end{lemma}

Due to Lemma~\ref{lem:HNRR}, we have the following estimate for $z$.

\begin{lemma}\label{lem:z}
Let $z$ satisfy \eqref{z-eq}. Then
there exists $t_0>0$ depending on $z_0$ such that for any $0<a<b<\infty$, it holds
\beaa
0<\inf_{t\geq1,a\leq x\leq b}z(t,x)\leq \sup_{t\geq1,a\leq x\leq b}z(t,x)<\infty.
\eeaa
\end{lemma}
\begin{proof} Our proof is based on \cite[Lemma 2.1]{HNRR}.
Define  $$q(t,x)=e^{\frac{c_u }{2d}x} z(t,x).$$ Then, $q$ satisfies
\beaa
q_t=d q_{xx}-\frac{3d}{c_u(t+t_0)}q_x{+}\frac{3}{2(t+t_0)}q-C_0 e^{-\mu t}q \quad \mbox{in $(0,\infty)\times(0,\infty)$}
\eeaa
with $q(t,0)=0$.
Using the self-similar variables
\beaa
\tau=\ln(t+t_0)-\ln t_0,\quad y=\frac{x}{[d(t+t_0)]^{1/2}},
\eeaa
and setting $Q(\tau,y):=q(t,x)$,
direct computations yield that
\beaa
Q_{\tau}-\mathcal{L}Q=-\varepsilon{e^{-\tau/2}} Q_y+\Big[\frac{1}{2}-C_0t_0e^{\tau-\mu t_0(e^{\tau}-1)}\Big]Q \quad \mbox{in $(0,\infty)\times(0,\infty)$}
\eeaa
with $Q(\tau,0)=0$, where $\mathcal{L}$ is defined in Lemma~\ref{lem:HNRR} and
\beaa
\varepsilon:={\frac{3\sqrt{d}}{c_u\sqrt{t_0}}}.
\eeaa

Define $J(\tau):=C_0t_0e^{\tau-\mu t_0(e^{\tau}-1)}$ and
\beaa
I(\tau):=\exp\Big[\int_0^{\tau} \Big(\frac{1}{2}-J(s)\Big)ds\Big].
\eeaa
Then, by Lemma~\ref{lem:HNRR}, we have
\beaa
Q(\tau,y)&=&{I(\tau)}y\Big[\frac{e^{-y^2/4}}{2\sqrt{\pi}} \Big(\int_0^\infty \xi Q(0,\xi)d\xi+O(\varepsilon)\Big)+e^{-\tau/2}\widetilde{Q}(\tau,y)\Big]\\
         &=&e^{\tau/2}e^{-\int_0^\tau J(s)ds}y\Big[\frac{e^{-y^2/4}}{2\sqrt{\pi}} \Big(\int_0^\infty \xi Q(0,\xi)d\xi+O(\varepsilon)\Big)+e^{-\tau/2}\widetilde{Q}(\tau,y)\Big],
\eeaa
where $|\tilde{Q}(\tau,y)|\leq C_K e^{-y^2/8}$ for all $\tau>0$ and $y\in K$ for any compact set $K$.
It follows that
\beaa
z(t,x)=\frac{{xe^{-\frac{c_u }{2d}x}} }{\sqrt{d t_0}}e^{-\int_0^{\ln[(t+t_0)/t_0]} J(s)ds}\left[Ce^{-x^2/[4d(t+t_0)]}+\widetilde{z}(t,x)\,\right],
\eeaa
where for any $0<a<b<\infty$,
\beaa
\limsup_{t\to\infty}|\widetilde{z}(t,x)|<\frac{C}{2}.
\eeaa
Furthermore, it is easily checked that there exist two positive constants $C_1$ and $C_2$ such that
\beaa
C_1\leq e^{-\int_0^{\ln[(t+t_0)/t_0]} J(s)ds}\leq C_2\quad\mbox{for all $t\geq0$}.
\eeaa
It follows that for any given $0<a<b<\infty$, $z(t,x)$ has a positive lower bound and a positive upper bound
for $x\in[a,b]$ and $t\geq t_0$, provided $t_0$ is large enough.
For $1\leq t\leq t_0$, one can use the strong maximum principle to assert
that $z(t,x)$ has a positive lower bound and a positive upper bound
for $x\in[a,b]$ and $t\in[1, t_0]$. The proof is thus complete.
\end{proof}

Based on Lemma~\ref{lem:z}, one can apply the argument in \cite{HNRR} to derive a lower estimate of $\min E_m(t)$ under the condition $c_u>c_v$.

\begin{lemma}\label{lem:lower-E}
Assume that $c_u>c_v$. For any $m\in(0,1)$, there exist $M>0$ and $T>0$ such that
\beaa
\min E_m(t)\geq c_u t-{\frac{3{d}}{c_u}}\ln t-M,\ \ \forall t\geq T.
\eeaa
\end{lemma}
\begin{proof}
Thanks to Lemma~\ref{lem:z}, we can follow the same line as that in \cite[Propositon 3.1 and Corollary 3.2]{HNRR}  to deduce that
there exists $M'>0$ and $T_0>0$ such that
\bea\label{E-lowbar}
\min \underline{E}_m(t)\geq c_u t-{\frac{3{d}}{c_u}}\ln t-M',\ \ \forall t\geq T_0,
\eea
where $\underline{E}_m(t)=\{x>0|\,\underline{u}(t,x)=m\}$ and $\underline{u}$ solves \eqref{subsol-eq} with
$\underline{u}(0,\cdot)\geq(\not\equiv)0$ and is of compact support. Using Lemma~\ref{lem:v-vanishing} and
taking $\underline{u}(0,\cdot)\leq u(T,\cdot)$ ($T$ is defined in Lemma~\ref{lem:v-vanishing}),
one can apply the comparison principle to deduce that
$u(t+T,\cdot)\geq \underline{u}(t,\cdot)$ for all $t\geq0$, which in turn implies that
$$\min E_m(t+T)\geq \min \underline{E}_m(t),\ \ \forall t\geq T_0.$$ By \eqref{E-lowbar}, we thus complete the proof.
\end{proof}

We are ready to prove Theorem~\ref{thm2}.

\begin{proof}[Proof of Theorem~\ref{thm2}]
By Lemma~\ref{lem:v-vanishing}, we have
$\lim_{t\to\infty}\sup_{x\in[0,\infty)}|v(t,x)|=0$. Also,
in view of Lemma~\ref{lem:upper-E} and Lemma~\ref{lem:lower-E}, we can safely follow the same analysis as that in \cite[Section 4]{HNRR} to conclude that there exist $C>0$ and a bounded function $\omega:\ [0,\infty)\to \mathbb{R}$ such that
\beaa
\lim_{t\to\infty}\sup_{x\in[0,\infty)}\Big|u(t,x)-U_{KPP}\Big(x-c_{u} t+{\frac{3{d}}{c_u}}\ln t+\omega(t)\Big)\Big|=0.
\eeaa
Thus, the proof is complete.
\end{proof}

{
The argument used in this subsection can provide
the propagating behavior of the fastest species for a general $n$-species competition-diffusion system.
We formulate it as follows:

\begin{cor}
Consider the following $n$-species competition-diffusion system:
\beaa
\begin{cases}
u^{i}_t=d_i u^{i}_{xx}+r_iu^{i}(1-\sum_{j=1}^{n}b_{ij}u^{j}),\quad t>0,\ x\in\mathbb{R},\ i=1,...,n,\\
u^i(0,x)=u^i_0(x),\ x\in\mathbb{R},\ i=1,...,n,
\end{cases}
\eeaa
where $d_i,r_i,b_{ij}>0$ for $i,j=1,\cdots,n$ and
$u^i_0\in C(\mathbb{R})\setminus\{0\}$, $u^i_0\geq0$ with compact support.
If
\beaa
c_{1}\leq c_{2} \leq\cdots \leq c_{n-1}< c_{n},
\eeaa
where $c_{i}:=2\sqrt{d_ir_i}$ for $i=1,\cdots,n$,
then for any small $\epsilon>0$, it holds that
\beaa
&&
\lim_{t\to\infty}\sup_{x\in[(c_{n-1}+\epsilon)t,\infty)}\Big|u^n(t,x)-U_{KPP}\Big(x-c_{n} t+
{\frac{3d_n}{c_n}}\ln t+\omega(t)\Big)\Big|=0,\\
&&\lim_{t\to\infty}\sum_{i=1}^{n-1}\sup_{x\in[(c_i+\epsilon)t,\infty)}|u^i(t,x)|=0,
\eeaa
where $\omega$ is a bounded function defined on $[0,\infty)$, and $U_{KPP}(x-c_n t)$ is a traveling wave solution of
the Fisher-KPP equation in \eqref{Single eq} {\rm(}with $d=d_n$ and $r=r_n${\rm)}
connecting $1$ and $0$.
\end{cor}
}

\medskip

\subsection{The proof of Theorem~\ref{thm3}}

In this subsection, combining some arguments used in the proof of Theorem~\ref{thm1} and Theorem ~\ref{thm2},
we shall establish Theorem~\ref{thm3}.

\begin{lemma}\label{lem:exp-deacy-u}
Assume that $c_v>c_u$. Then for any $c>c_{uv}$, there exist positive constants $C,\mu, T$ such that
\beaa
\sup_{x\in[ct,\infty)}u(t,x)\leq C e^{-\mu t},\quad\forall t\geq T.
\eeaa
\end{lemma}
\begin{proof}
First, we will show that
for each $c_{uv}<c^-<c^+<c_v$, there exist $C_1,\mu_1,T_1>0$
\bea\label{exp-decay-u}
\sup_{x\in[c^{-}t,c^{+}t]}u(t,x)\leq C_1 e^{-\mu_1 t},\quad \forall t\geq T_1.
\eea
Since $c_v>c_u$, thanks to \cite[Theorem 1]{Carrere} (following the proof there with slight modifications), we have
\bea\label{intermediate}
\lim_{t\to\infty}\sup_{c_1t\leq x\leq c_2t}\Big(|u(t,x)|+|v(t,x)-1|\Big)=0\quad \mbox{for all $c_{uv}<c_1<c_2<c_v$}.
\eea
Therefore, one can choose small $\varepsilon>0$ and $T_0\gg1$ such that
$0<u<\varepsilon$ and $v\geq 1-\varepsilon$ in $[T_0,\infty)\times [c_1t,c_2t]$.

For notational convenience, let us denote $$\rho:=-r[1-a(1-\varepsilon)].$$ Here
we may assume that $\rho>0$ since $a>1$ and $0<\varepsilon\ll1$.
This implies that
\bea\label{key-ineq}
u_t\leq d u_{xx}-\rho u\quad \mbox{in $[T_0,\infty)\times [c_1t,c_2t]$};\quad  {u}(t,c_it)\in{[0,\varepsilon]},\quad \mbox{for $t\geq T_0$, $i=1,2$}.
\eea

Set
\beaa
&&c^{*}:=(c_1+c_2)/2,\quad \hat{c}:=(c_2-c_1)/2,\\
&&y:=x-c^* t,\quad (\hat{u},\hat{v})(t,y):=(u,v)(t,y+c^* t).
\eeaa
By \eqref{key-ineq}, we have
\beaa
\begin{cases}
\hat{u}_t\leq d\hat{u}_{yy}+c^*  \hat{u}_y-\rho\hat{u}\quad \mbox{in $[T_0,\infty)\times [-\hat{c}t,\hat{c}t]$},\\
\hat{u}(t,\pm\hat{c}t)\in{ [0,\varepsilon]}\ \mbox{ for } t\geq T_0.
\end{cases}
\eeaa
Fix $T>T_0$ and consider
\bea\label{phi-sys}
\begin{cases}
\phi_t=d\phi_{yy}+c^*\phi_y-\rho\phi,\quad t>0,\ -\hat{c}T<y<\hat{c}T,\\
\phi(t,\pm\hat{c}T)=\varepsilon,\quad t>0,\\
\phi(0,x)=\varepsilon,\quad -\hat{c}T\leq y\leq \hat{c}T.
\end{cases}
\eea
Then, by comparison, we have
 $$\phi(t,{y})\geq \hat{u}(t+T,{y})\ \ \mbox{ for $t\geq0$ and $-\hat{c}T\leq y\leq \hat{c}T$.}$$

Let $\Phi(t,y)=e^{\rho t}(\varepsilon-\phi)$. Then {the} system \eqref{phi-sys} is reduced to
\beaa\label{Phi-sys}
\begin{cases}
\Phi_t=d\Phi_{yy}+c^*\Phi_y+\varepsilon\rho e^{\rho t},\quad t>0,\ -\hat{c}T<y<\hat{c}T,\\
\Phi(t,\pm\hat{c}T)=0,\quad t>0,\\
\Phi(0,x)=0,\quad -\hat{c}T\leq y\leq \hat{c}T.
\end{cases}
\eeaa
From the proof of \cite[Proposition 3.2]{KM2015}, we have: for any small $\sigma>0$, there exists $T^*\gg1$ and $\nu(\sigma)>0$ such that for $T\geq T^*$,
\beaa
\Phi(t,y)\geq \rho (e^{\rho t}-1)(1-C_1 e^{-\nu(\sigma) \hat{c}T}),\quad (t,x)\in D(\sigma),
\eeaa
where  $C_1$ is a positive constant, $\nu(\sigma)$ has a positive lower bound for all small $\sigma$ and
$$D(\sigma):=\left\{(t,y)\Big|\, 0<t<\frac{(\sigma\hat{c})^2T}{4\sqrt{d}},\ |y|\leq (1-\sigma)\hat{c}T\right\}.$$
It follows that
\beaa
\phi(t,y)\leq \varepsilon-\varepsilon(1-e^{-\rho t})(1-C_1 e^{-\nu\hat{c} T})\leq \varepsilon(C_1e^{-\nu\hat{c} T}+e^{-\rho t}),\quad (t,x)\in D(\sigma).
\eeaa

Taking $t=(\sigma\hat{c})^2 T/(4\sqrt{d})$ and $\sigma$ small enough such that $\nu\hat{c}>\rho(\sigma\hat{c})^2/(4\sqrt{d})$,
we obtain
\bea\label{phi-lower}
\phi(t,y)\leq\varepsilon(C_1+1)e^{-\rho(\sigma\hat{c})^2T/(4\sqrt{d})},\quad |y|\leq (1-\sigma) \hat{c}T.
\eea
Then, by comparison, $\hat{u}(t+T,y)\leq \phi(t,y)$, which together with \eqref{phi-lower} gives
\beaa
\hat{u}\Big(\frac{(\sigma \hat{c})^2 T}{4\sqrt{d}}+T,y\Big)&\leq&\varepsilon(C_1+1)e^{-\rho(\sigma\hat{c})^2T/(4\sqrt{d})},\quad |y|\leq (1-\sigma) \hat{c}T.
\eeaa
Note that
\beaa
t=\frac{(\sigma \hat{c})^2 T}{4\sqrt{d}}+T\quad \Longleftrightarrow\quad T=\Big(1+\frac{\sigma^2 \hat{c}^2}{4\sqrt{d}}\Big)^{-1}t.
\eeaa
Thus, we infer that
\beaa
\hat{u}(t,y)\leq\varepsilon(C_1+1)e^{-\delta t}\quad\mbox{ for $t\geq T^{**}$ and $|y|\leq (1-\sigma) \hat{c}t$},
\eeaa
where
\beaa
\delta:=\rho(\frac{\sigma^2 \hat{c}^2}{4\sqrt{d}})\Big(1+\frac{\sigma^2 \hat{c}^2}{4\sqrt{d}}\Big)^{-1}>0,\quad T^{**}:=T^{*}+\frac{\sigma^2\hat{ c}^2}{4\sqrt{d}}T^{*}.
\eeaa
Hence, it follows that
\beaa
u(t,x)\leq\varepsilon(C_1+1)e^{-\delta t}\quad\mbox{for $t\geq T^{**}$ and $[c^*-(1-\sigma)\hat c]t\leq x\leq[c^*+(1-\sigma)\hat c]t$}.
\eeaa
Since $\sigma>0$ can be arbitrarily small and $c_1$ (resp., $c_2$) can be arbitrarily close to $c_{uv}$ (resp., $c_v$) such that
$(1-\sigma)c_1<c^-<c^+<(1-\sigma) c_2$,
we see that \eqref{exp-decay-u} holds. Finally, due to the assumption $c_u<c_v$,
Lemma~\ref{lem:exp-deacy-u} follows from \eqref{exp-decay-u} and Lemma~\ref{lem:exp-decay}.
\end{proof}

Thanks to Lemma~\ref{lem:exp-deacy-u} and \eqref{intermediate}, one can follow the same lines as in Lemma~\ref{lem:u-rate} (with minor modifications) to obtain
the following result.

\begin{lemma}\label{lem:exp-deacy-v}
Assume that $c_v>c_u$. Then for any $c_{uv}<c_1<c_2<c_v$, there exist positive constants $C'$, $\nu$ and $T'$ such that
\beaa
\inf_{x\in[c_1 t,c_2 t]}v(t,x)\geq 1-C' e^{-\nu t},\quad\forall t\geq T'.
\eeaa
\end{lemma}

\begin{remark}\label{rk-sec4}
We remark that the parallel proof of Lemma~\ref{lem:exp-deacy-u} and Lemma~\ref{lem:exp-deacy-v} also shows that
if $c_v>c_u$, then there exist positive constants $C$, $\mu$, $\nu$ and $T$ such that
\beaa
&&\sup_{x\in(-\infty,-ct]}u(t,x)\leq C e^{-\mu t},\ \quad\forall t\geq T\ \ \mbox{if $c>c_{uv}$},\\
&&\inf_{x\in[-c_2 t,-c_1 t]}v(t,x)\geq 1-C e^{-\nu t},\ \quad \forall t\geq T\ \ \mbox{if $c_{uv}<c_1<c_2<c_v$}.
\eeaa
\end{remark}


\begin{lemma}\label{lem:cov-left}
Assume that $c_v>c_u$. Then for any $c\in(c_{uv},c_v)$,
there exists $h_1\in \mathbb{R}$ such that the solution of \eqref{LV-sys}-\eqref{LV-ic} satisfies
\beaa
&&\lim_{t\to\infty}\left[\sup_{x\in[0,ct)}
\Big|u(t,x)-U(x-c_{uv}t-h_1)\Big|+\sup_{x\in[0,ct)}\Big|v(t,x)-V(x-c_{uv}t-h_1)\Big|
\right]=0.
\eeaa
\end{lemma}
\begin{proof}
Let $(\hat{u},\hat{v})$ be the solution of \eqref{LV-sys} with initial datam $(\hat{u}_0, \hat{v}_0)$ satisfying
\bea\label{A1-A2-ic}
\hat{u}_0(x)=u_0(x),\quad \hat{v}_0(x)> v_0(x) \quad\mbox{in $\mathbb{R}$},
\eea
and $\hat{v}_0(\cdot)\geq \rho$ in $\mathbb{R}$ for some $\rho>0$.
Thanks to \eqref{A1-A2-ic}, we can compare $(\hat{u},\hat{v})$ with $(u,v)$ such that
\bea\label{cp-A1-A2}
\hat{u}(t,x)\leq u(t,x),\quad \hat{v}(t,x)\geq v(t,x)\quad \mbox{for $t\geq0$ and $x\in\mathbb{R}$}.
\eea

Denote $(\underline{u},\overline{v})$ by {\eqref{subsol-2}} such that {Lemma~\ref{lem:subsol-2}} holds.
Since $(\hat{u}_0, \hat{v}_0)$ satisfies {\bf(A1)}, one can
apply {Lemma~\ref{lem:order-4}} (with {a} suitable choice of parameters)
to {ensure} that $\underline{u}\leq\hat{u}$ and $\overline{v}\geq \hat{v}$ over $[T_0,\infty)\times[0,\infty)$ for some $T_0\gg1$.
 Together with \eqref{cp-A1-A2}, we have
\bea\label{goal-1}
\underline{u}(t,x) \leq u(t,x) ,\quad \overline{v}(t,x)\geq v(t,x) \quad \mbox{in $[T_0,\infty)\times[0,\infty)$}.
\eea

Next, {denote $(\overline{u},\underline{v})$ by \eqref{supersol-2} such that Lemma~\ref{lem:super-sol2} holds.}
For any given $c\in(c_{uv},c_v)$, we shall show that for some large $T^{**}\geq T^*$ {($T^*$ is defined in Lemma~\ref{lem:super-sol2})} and small $\beta^{**}$,
\bea\label{bdry cond-sec4}
\qquad \overline{u}(t,\pm ct)\geq u(t,\pm ct),\quad v(t,\pm ct)\geq \underline{v}(t,\pm ct)\quad \mbox{for all $t\geq T^{**}$, provided $\beta\in(0,\beta^{**})$},
\eea
It follows from {\eqref{key est:1-U-}} and Lemma~\ref{lem:exp-deacy-u} that for some ${T_1}>0$,
\beaa
\overline{u}(t,ct)-u(t,ct)&\geq& 1-U(-ct-c_{uv}t+\zeta(t))+\hat{p}(t)-u(t,ct)\\
&\geq& -K_1e^{-\lambda_{u}[(c+c_{uv})t-{\zeta_0}]}+\hat{p}_0 e^{-\beta t}- Ce^{-\mu t},\quad t\geq {T_1},
\eeaa
where $C$, $\mu$ are given in Lemma~\ref{lem:exp-deacy-u}.
Therefore, taking $\beta^{**}<\min\{\lambda_{u}(c+c_{uv}),\mu\}$ and ${T_1}$ larger if necessary,
we see that $\overline{u}(t,ct)\geq u(t,ct)$ for all $t\geq {T_1}$, provided $\beta\in(0,\beta^{**})$.
Thanks to {\eqref{key est:V-}} and Lemma~\ref{lem:exp-deacy-v}, there exists ${T_2}>0$ such that
\beaa
v(t,ct)-\underline{v}(t,ct)&\geq& 1-C' e^{-\nu t}-(1-\hat{q}(t))[1+V(-ct-c_{uv}t+\zeta(t))]\\
&\geq& 1-C' e^{-\nu t}-(1-\hat{q}_0e^{-\beta t})[1+K_2e^{-\lambda_{v}[(c+c_{uv})t-\zeta(0)]}],\quad t\geq {T_2},
\eeaa
where $C'$ and $\nu$ are given in Lemma~\ref{lem:exp-deacy-v}.
Taking $\beta^{**}$ smaller such that
\beaa
\beta^{**}<\min\{\lambda_{u}(c+c_{uv}),\mu, \nu, \lambda_{v}(c+c_{uv}) \}
\eeaa
and ${T_2}$ larger if necessary, we obtain that $v(t,ct)\geq \underline{v}(t,ct)$ for all $t\geq {T_2}$, provided $\beta\in(0,\beta^{**})$.

Since $\overline{u}(\cdot,t)$ and $\underline{v}(\cdot,t)$ are even, the similar process used in the above (see also Remark~\ref{rk-sec4}) can be applied to assert $\overline{u}(t,-ct)\geq u(t,-ct)$ and $v(t,-ct)\geq \underline{v}(t,-ct)$ for $t\geq {T_3}$,  provided $\beta\in(0,\beta^{**})$ ($\beta^{**}$ may become smaller), where ${T_3}$ is some large constant. Therefore, \eqref{bdry cond-sec4} follows with $T^{**}:=\max\{T^*,{T_1},{T_2},{T_3}\}$.

To use $(\overline{u},\underline{v})$ as a comparison function over $[T^{**},\infty)\times {[-ct, ct]}$,
we fix $\beta<\min\{\beta^*,\beta^{**}\}$. Then, taking $\zeta_0$ close to $-\infty$
(this does not affect the choice of $\beta^{*}$ and $\beta^{**}$), from the definition of $(\overline{u},\underline{v})$ we can easily see
\beaa
\overline{u}(T^{**},x)\geq u(T^{**},x),\quad v(T^{**},x)\geq \underline{v}(T^{**},x)\quad\mbox{for $x\in[-cT^{**}, cT^{**}]$}.
\eeaa
As a result, a simple comparison analysis yields
\bea\label{goal-2}
\overline{u}(t,x) \geq u(t,x) ,\quad {v}(t,x)\geq \underline{v}(t,x) \quad \mbox{in $[T^{**},\infty)\times{[-ct, ct]}$}.
\eea

Now, combining \eqref{goal-1} and \eqref{goal-2}, we obtain that for all large time and {$|x|\leq ct$},
{
\beaa
&&U(x-c_{uv}t+\zeta(t))+U(-x-c_{uv}t+\zeta(t))-1-\hat{p}(t)\\
&&\qquad \leq u(t,x)\leq U(x-c_{uv}t+\zeta(t))+U(-x-c_{uv}t+\zeta(t))-1+\hat{p}(t),\\
&& (1-\hat{q}(t))\Big[V(x-c_{uv}t+\zeta(t))+V(-x-c_{uv}t+\zeta(t))\Big]\\
&&\qquad\leq v(t,x)\leq (1+\hat{q}(t))\Big[V(x-c_{uv}t+\zeta(t))+V(-x-c_{uv}t+\zeta(t))\Big].
\eeaa
}
Then following the same line as in the proof {Proposition~\ref{prop1}}, we can finish the proof of Lemma~\ref{lem:cov-left} and may safely omit the details. This completes the proof.
\end{proof}

We are now in a position to verify Theorem~\ref{thm3}.
\begin{proof}[Proof of Theorem~\ref{thm3}]
We first show that, for any $c>c_{uv}$,
\bea\label{goal-3}
&&\lim_{t\to\infty}\left[\sup_{x\in[ct,\infty)}\Big|v(t,x)-V_{KPP}(x-c_{v} t+\frac{3}{c_v}\ln t+\omega(t))\Big|+\sup_{x\in[ct,\infty)}\Big|u(t,x)\Big|\right]=0,
\eea
where $\omega$ is a bounded function defined on $[0,\infty)$.
Indeed, by Lemma~\ref{lem:exp-deacy-u}, $u$ decays to zero exponentially for $x\in[ct,\infty)$, which allows us to
estimate $v$ along the process in Section 4.1 by
exchanging the role of $u$ and $v$ therein. Then we can deduce that
there exists a bounded function $\omega:\ [0,\infty)\to \mathbb{R}$ such that
\beaa
\lim_{t\to\infty}\sup_{x\in[ct,\infty)}\Big|v(t,x)-V_{KPP}\Big(x-c_{v} t+\frac{3}{c_v}\ln t+\omega(t)\Big)\Big|=0.
\eeaa
Hence, \eqref{goal-3} holds.

In view of $c_{uv}<c_u<c_v$ and $c_0=\frac{c_{uv}+c_v}{2}$, Theorem~\ref{thm3} follows immediately from Lemma~\ref{lem:cov-left} and \eqref{goal-3}. The proof is thus complete.
\end{proof}

\noindent{\bf Acknowledgments}
We appreciate {the valuable comments and suggestions of the reviewers and the editor}, which help us to improve the manuscript.
RP was partially supported by NSF of China (No. 11671175, 11571200), the Priority Academic Program Development of Jiangsu Higher Education Institutions, Top-notch Academic Programs Project of Jiangsu Higher Education Institutions (No. PPZY2015A013) and Qing Lan Project of Jiangsu Province; CHW was partially supported by the Ministry of Science and Technology of Taiwan (MOST 108-2636-M-009-009, MOST 109-2636-M-009-008);
MLZ was partially supported by the Australian Research
Council (No. DE170101410).

\medskip

{\noindent{\bf Appendix}}

{
We provide a proof of \eqref{segregation-uv}.
}

{
\begin{proof}[Proof of \eqref{segregation-uv}]
Given $c\in(0,c_{uv})$, {it} is well known that $c_{uv}$ has the continuous dependence property on parameters \cite{Kan-on95}.
Therefore, one can choose $c_{\e}\in(c,c_{uv})$ which is close to $c_{uv}$ such that there exists $(\underline{U},\overline{V})$ satisfying
\bea\label{perturb-TWsys}
\begin{cases}
c_{\e}\underline{U}'+d\underline{U}''+r\underline{U}[1-\e-\underline{U}-a\overline{V}]=0,\quad \xi\in\mathbb{R},\\
c_{\e}\overline{V}'+\overline{V}''+\overline{V}[1+\e-\overline{V}-b\underline{U}]=0,\quad \xi\in\mathbb{R},\\
(\underline{U},\overline{V})(-\infty)=(1-\e,0),\quad (\underline{U},\overline{V})(+\infty)=(0,1+\e),\\
\underline{U}'(\xi)<0,\quad \overline{V}'(\xi)>0,\quad \xi\in\mathbb{R}.
\end{cases}
\eea

We define a subsolution $(\underline{u},\overline{v})$ by
\beaa
\begin{cases}
\underline{u}(x,t)=\max\Big\{0,\underline{U}(x-c_{\e}t-\zeta(t))+\underline{U}(-x-c_{\e}t-\zeta(t))-(1-\e)-p(t)\Big\},\\
\overline{v}(x,t)=\overline{V}(x-c_{\e}t-\zeta(t))+\overline{V}(-x-c_{\e}t-\zeta(t))+q(t),
\end{cases}
\eeaa
where $p(t)=p_0 e^{-\mu t}$, $q(t)=q_0 e^{-\mu t}$ and $\zeta(t)=-\zeta_0+\zeta_1 e^{-\mu t}$, will be determined later.

Our goal is to show that one can choose suitable parameters given in the above such that for some large $T>0$,
\bea\label{Appendix-goal-1}
N_1[\underline{u},\overline{v}](t,x)\leq0,\quad N_2[\underline{u},\overline{v}](t,x)\geq0 \quad \mbox{for $t\geq T$ and $x\in\mathbb{R}$}.
\eea
By the symmetry, we can only consider $x\geq0$.
For convenience, we set
\beaa
\xi_\pm=\pm x-c_{\e}t-\zeta(t),\quad (U_{\pm},V_{\pm})=(\underline{U}(\xi_{\pm}),\overline{V}(\xi_{\pm})).
\eeaa

Take $\zeta_1>0$ such that $\zeta'<0$. Since $\underline{U}'<0$ and $\overline{V}'>0$, by Lemma~\ref{lem:AS-}, one has
\bea\label{App-key}
\begin{cases}
1-\e-U_-\leq 1-\e-\underline{U}(-c_{\e}t+\zeta_0)\leq K_1e^{-\lambda_{u}(c_{\e}t-\zeta_0)}\quad
\mbox{for all $x\geq0$ and $t\geq0$},\\
V_-\leq K_2e^{-\lambda_{v}(c_{\e}t-\zeta_0)}\quad \mbox{for all $x\geq0$ and $t\geq0$,}
\end{cases}
\eea
for some $\lambda_{u},\lambda_v,K_1,K_2>0$.

We first consider $\underline{u}>0$.
By direct computation and using the equations in \eqref{perturb-TWsys},
we have
\beaa
N_1[\underline{u},\overline{v}]&=&-\zeta'(U'_+ +U'_-)-p'+rU_+[-\e +U_- - (1-\e)-p+a(V_- +q)]
\\&&+rU_-[-\e+ U_+ - (1-\e)-p+a(V_+ +q)]
\\&&+r((1-\e)+p)[1-(U_+ +U_- -(1-\e)-p)-a(V_+ +V_-+q)]
\eeaa
and
\beaa
N_2[\underline{u},\overline{v}]&=&-\zeta'(V'_+ +V'_-)+q'+V_+[\e+(V_- +q)+b(U_{-} -(1-\e)-p)]\\
&&+V_-[\e+(V_+ +q)+b(U_+ -(1-\e)-p)]
\\&&-q[1-(V_+ + V_- +q)-b(U_+ +U_- -(1-\e)-p)].
\eeaa

We divide the discussion into three cases.
In the following, the positive constant $C$ is independent of $\mu$ and may change from line to line.

Let us take
\bea\label{choice-mu}
0<\mu<\min\{\lambda_{u}c_{\e}, \lambda_{v}c_{\e}, r(a-1), {b-1}\}.
\eea

{\bf Case 1:} $0\leq U_+\leq \delta$ and $1-\delta\leq V_+\leq 1$ for some small $\delta>0$. Note that $\zeta'U'_{\pm}>0$. Then
\beaa
N_1[\underline{u},\overline{v}]&\leq&
-p'-rU_+(p -aq -aV_-)-rU_-(1-U_+ -aV_+)-rU_-(p-aq)\\
&&+r(1-\e+p)[(1-U_+-aV_+)+(1-\e-U_-) +p -aq].
\eeaa
By taking $p_0=aq_0$ and collecting $V_-$, $1-\e -U_-$ and $p$ separately, we obtain
\beaa
N_1[\underline{u},\overline{v}]&\leq& -rp[-\mu/r+aV_+-1]+C[(1-\e-U_-)+V_-]\\
&\leq& -rp_0e^{-\mu t}[a(1-\delta)-\mu/r-1]+C[e^{-\lambda_{u}(c_{\e}t-\zeta_0)}+e^{-\lambda_{v}(c_{\e}t-\zeta_0)}],
\eeaa
for $0\leq U_+\leq \delta$ and $1-\delta\leq V_+\leq 1$, where we used \eqref{App-key}.
Therefore, by \eqref{choice-mu} and choosing $\delta$ small enough,
for some large $T_1$,  we have $N_1[\underline{u},\overline{v}]\leq0$  for $t\geq T_1$ within the range in Case 1.

Consider the inequality of $N_2[\underline{u},\overline{v}]$. Since $\zeta'V'_{\pm}<0$, we have
\beaa
N_2[\underline{u},\overline{v}]&\geq& q'+V_+[\e -b(1-\e-U_-)-bp]-CV_- -q\\
&\geq& -\mu q_0e^{-\mu t}+(1-\delta)\e -C e^{-\lambda_{u}(c_{\e}t-\zeta_0)} -C e^{-\mu t}  -Ce^{-\lambda_{v}(c_{\e}t-\zeta_0)} -q_0e^{-\mu t},
\eeaa
where we have used $V_{+}\geq1-\delta$ and \eqref{App-key}.
Hence,  one can find $T_2\gg1$ such that $N_2[\underline{u},\overline{v}]\geq0$ for $t\geq T_2$ within the range in Case 1.

{\bf Case 2:} $1-\delta\leq U_+\leq 1$ and $0\leq V_+\leq \delta$ for some small $\delta>0$. This case is similar to Case 1. Indeed,
it holds that
\beaa
N_1[\underline{u},\overline{v}]&\leq&-p'
-rU_+(\e+p -aq -aV_-)-rU_-(1-U_+ -aV_+) -rU_-(p-aq)\\
&&+r(1-\e+p)[(1-U_+ -aV_+) +(1-\e-U_-) +p-aq]\\
&\leq& -r(1-\delta)\e+Ce^{-\mu t}+Ce^{-\lambda_{v}(c_{\e}t-\zeta_0)}+Ce^{-\lambda_{u}(c_{\e}t-\zeta_0)},
\eeaa
where $p_0=aq_0$, $1-\delta\leq U_+$ and \eqref{App-key} are used.
Therefore,
for some large $T_3$,  we have $N_1[\underline{u},\overline{v}]\leq0$ for $t\geq T_3$.

On the other hand, we observe that
\beaa
N_2[\underline{u},\overline{v}]&\geq& q'+V_+[\e -b(1-\e-U_-)-bp]+CV_- -q(1-bU_+ +b(1-\e -U_-)+bp)\\
&\geq& -\mu q_0e^{-\mu t} +V_+[\e -C(e^{-\lambda_{u}(c_{\e}t-\zeta_0)}+e^{-\mu t})] +Ce^{-\lambda_{v}(c_{\e}t-\zeta_0)}
\\&&+q_0e^{-\mu t}[b(1-\delta)-1-C(e^{-\lambda_{u}(c_{\e}t-\zeta_0)}+e^{-\mu t})].
\eeaa
Hence, by \eqref{choice-mu} and choosing $\delta$ small enough, one can find $T_4\gg1$ such that $N_2[\underline{u},\overline{v}]\geq0$ for $t\geq T_4$.

{\bf Case 3:} the middle part: $\delta_1\leq U_+,V_+\leq1-\delta_2$ for some small $\delta_i>0$ for $i=1,2$.
In this case, one has: $U_+'<-\kappa$ and $V_+'>\kappa$ for some $\kappa>0$. Then using $p_0=aq_0$ it holds that
\beaa
N_1[\underline{u},\overline{v}]&\leq&
\zeta'\kappa-p'-rU_+(\e+p -aq -aV_-)-rU_-(1-U_+-aV_+) -rU_-(p-aq)\\
&&+r(1-\e+p)[(1-U_+-aV_+)+(1-\e-U_- +p)-a(V_- +q)]\\
&\leq& \zeta'\kappa-p'-rU_+(\e-aV_-)+r(1-\e -U_-+p)[1-U_+ -aV_+]\\
&&+r(1-\e +p)[1-\e -U_- ]\\
&\leq&  -\zeta_1\kappa\mu e^{-\mu t}+ \mu p_0  e^{-\mu t}  -rU_+(\e-Ce^{-\lambda_{v}(c_{\e}t-\zeta_0)})
\\&&+r(Ce^{-\lambda_{u}(c_{\e}t-\zeta_0)}+p_0e^{-\mu t})[1-\delta_1]+r(1-\e+p)Ce^{-\lambda_{u}(c_{\e}t-\zeta_0)}\\
&\leq& [-\zeta_1\kappa\mu + \mu p_0 +rp_0(1-\delta_1)]e^{-\mu t}-rU_+(\e-Ce^{-\lambda_{v}(c_{\e}t-\zeta_0)})+Ce^{-\lambda_{u}(c_{\e}t-\zeta_0)}.
\eeaa
By \eqref{choice-mu} and choosing $p_0$ small enough, one can find $T_5\gg1$ such that $N_1[\underline{u},\overline{v}]\leq0$ for $t\geq T_5$.

Next, we also find that
\beaa
N_2[\underline{u},\overline{v}]&\geq& -\zeta'\kappa+q'+V_+[\e -b(1-\e-U_-)-bp]+CV_- -q(1-bU_+ +b(1-\e -U_-)+bp)\\
&\geq& \zeta_1\kappa\mu e^{-\mu t}-\mu q_0e^{-\mu t} +V_+[\e -C(e^{-\lambda_{u}(c_{\e}t-\zeta_0)}+e^{-\mu t})] +Ce^{-\lambda_{v}(c_{\e}t-\zeta_0)}
\\&&-q_0e^{-\mu t}[1+Ce^{-\lambda_{u}(c_{\e}t-\zeta_0)}+bp_0 e^{-\mu t}].
\eeaa
By \eqref{choice-mu} and choosing $q_0$ small enough, for some $T_6\gg1$, $N_2[\underline{u},\overline{v}]\geq0$ for $t\geq T_6$.

From the above discussion and taking $T=\max\{T_1,T_2,T_3,T_4,T_5,T_6\}$, we have shown \eqref{Appendix-goal-1} if $\underline{u}>0$,
When $\underline{u}=0$, it is not hard to show \eqref{Appendix-goal-1}
still holds and we may omit the details here.

Let $T$ be fixed such that \eqref{Appendix-goal-1} holds. We shall take such $T$ as the initial time to compare $(\underline{u},\overline{v})(t,x)$
with the solution $(u,v)(t+\widehat{T},x)$ for some $\widehat{T}>0$ that can be chosen. To see this,
by the definition of $(\underline{u},\overline{v})(t,x)$, one may choose $L>0$ such that
$\underline{u}(T,x)=0$ and $\overline{v}(T,x)>1+\e$ for all $|x|\geq L$. For such fixed $L$, by {\bf(H3)},
$(u,v)(t,x)\to(1,0)$ as $t\to\infty$ uniformly for $x\in[-L,L]$. Also, by Lemma~\ref{lem:simple est},
$v(t,x)<1+\e$ for all large $t$ and $|x|\geq L$. Hence, we can choose $\widehat{T}\gg1$ such that
$u(T+\widehat{T},x)\geq \underline{u}(T,x)$ and $v(T+\widehat{T},x)\leq \overline{v}(T,x)$ for all $x\in\mathbb{R}$.
By comparison, we obtain that
\beaa
u(t+\widehat{T},x)\geq \underline{u}(t,x),\quad v(t+\widehat{T},x)\leq \overline{v}(t,x),\quad \mbox{for $t\geq T$ and $x\in\mathbb{R}$},
\eeaa
which implies that
\beaa
\lim_{t\to\infty}\inf_{|x|\leq c t}u(t,x)\geq 1-\e,\quad \lim_{t\to\infty}\sup_{|x|\leq c t}|v(t,x)|=0
\eeaa
for any $c\in(0,c_\e)>0$.
Since $\e>0$ can be arbitrarily small and $c_\e\to c_{uv}$ as $\e\to0$,
and using Lemma~\ref{lem:simple est},
we thus complete the proof of  \eqref{segregation-uv}.
\end{proof}
}


\end{document}